\documentclass[12pt]{article}

\usepackage{amsmath,amsfonts,amssymb,amsthm}
\usepackage[numbers,sort&compress]{natbib}
\usepackage{graphicx}
\usepackage[dvipsnames]{xcolor}
\usepackage{comment}
\usepackage[T1]{fontenc}
\usepackage{geometry}
\usepackage[colorlinks=true,
            linkcolor=DarkOliveGreen,
            citecolor=DarkOliveGreen,
            urlcolor=DarkOliveGreen]{hyperref}

\usepackage{palatino}
\definecolor{DarkOliveGreen}{rgb}{0.08,0.42,0.08}

\numberwithin{equation}{section}
\theoremstyle{plain}

\newtheorem{theorem}{Theorem}[section]
\newtheorem{lemma}[theorem]{Lemma}
\newtheorem{conj}[theorem]{Conjecture}
\newtheorem{proposition}[theorem]{Proposition}
\newtheorem{cor}[theorem]{Corollary}
\newtheorem{rem}[theorem]{Remark}

\theoremstyle{definition}
\newtheorem{definition}[theorem]{Definition}

\title{The Brownian marble}
\author{
Samuel G. G. Johnston\thanks{Department of Mathematics, King's College London. \texttt{samuel.g.johnston@kcl.ac.uk}}
\and Andreas Kyprianou\thanks{Mathematics Institute, University of Warwick. \texttt{andreas.kyprianou@warwick.ac.uk}}
\and Tim Rogers\thanks{Department of Mathematical Sciences, University of Bath. \texttt{ma3tcr@bath.ac.uk}}
\and Emmanuel Schertzer\thanks{Department of Mathematics, University of Vienna. \texttt{emmanuel.schertzer@univie.ac.at}}
}
\date{}

\begin{document}
\maketitle

\begin{abstract}
Let $R:(0,\infty) \to [0,\infty)$ be a measurable function. 
Consider coalescing Brownian motions started from every point in the subset $\{ (0,x) : x \in \mathbb{R} \}$ of $[0,\infty) \times \mathbb{R}$ (with $[0,\infty)$ denoting time and $\mathbb{R}$ denoting space) and proceeding according to the following rule: the interval $\{t\} \times [L_t,U_t]$ between two consecutive Brownian motions instantaneously fragments' at rate $R(U_t - L_t)$. At a fragmentation event at a time $t$, we initiate new coalescing Brownian motions from each of the points $\{ (t,x) : x \in [L_t,U_t]\}$. The resulting process, which we call the $R$-marble, is easily constructed when $R$ is bounded, and may be considered a random subset of the Brownian web. 

Under mild conditions, we show that it is possible to construct the $R$-marble when $R$ is unbounded as a limit as $n \to \infty$ of $R_n$-marbles where $R_n(g) = R(g) \wedge n$. The behaviour of this limiting process is mainly determined by the shape of $R$ near zero. The most interesting case occurs when the limit $\lim_{g \downarrow 0} g^2 R(g) = \lambda$ exists in $(0,\infty)$, in which case we find a phase transition. For $\lambda \geq 6$, the limiting object is indistinguishable from the Brownian web, whereas if $\lambda < 6$, then the limiting object is a nontrivial stochastic process with large gaps. 

When $R(g) = \lambda/g^2$, the $R$-marble is a self-similar stochastic process which we refer to as the \emph{Brownian marble with parameter $\lambda > 0$}. We give an explicit description of the spacetime correlations of the Brownian marble, which can be described in terms of an object we call the Brownian vein; a spatial version of a recurrent extension of a killed Bessel-$3$ process.
\end{abstract}

\vspace{1em}

\vspace{0.5em}

\begin{figure}[h!]
\centering
\includegraphics[width=0.55\textwidth]{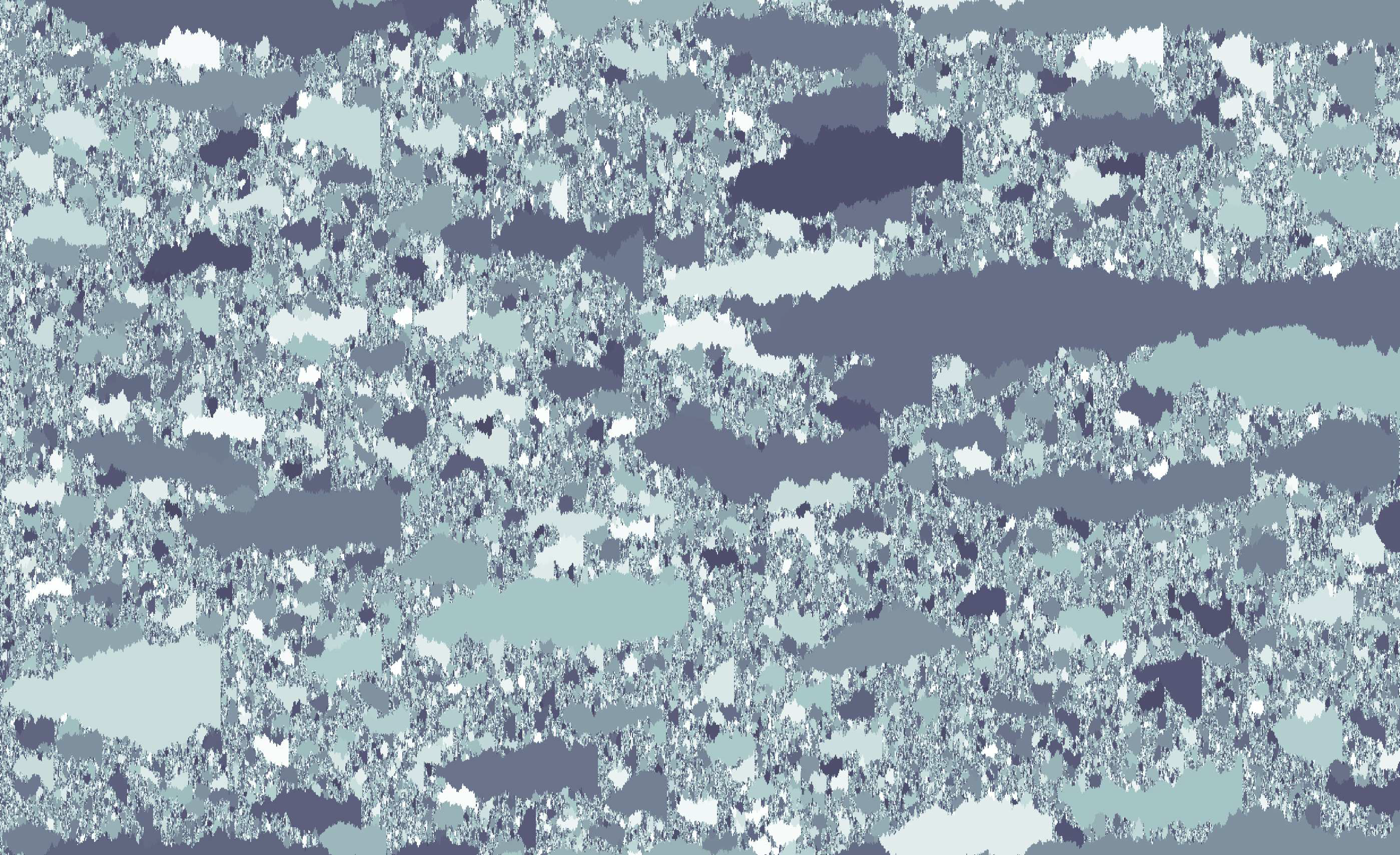}
\caption{The Brownian marble with parameter $\lambda = 3$.}
\label{fig:abstract}
\end{figure}


\section{Introduction and overview}

Consider a collection of Brownian motions in $\mathbb{R}$ that coalesce when they meet, meaning that any pair are independent up to the time that they first touch and are equal thereafter. At any particular moment $t > 0$, this collection defines a partition of $\mathbb{R}$ into countably many non-empty intervals of the form $[L_t,U_t)$, where $L_t$ and $U_t$ are consecutive distinct Brownian motions. If we were to track the evolution of the intervals overlapping some subset $[a,b]\subset\mathbb{R}$, we would observe many small intervals being progressively replaced by fewer, larger ones. This statement holds true even if we start at time zero with infinitely many Brownian motions, one issued from every point in $\mathbb{R}$. In fact, the number of intervals in our partition of $[a,b]$, while being infinite at time $t=0$, would instantaneously become finite at any $t>0$ (it can be shown that {the expected number of particles in $[a,b]$ at time $t$ is} $(b-a)/\sqrt{2\pi t}$ \cite[Section 6.2.4]{SSS}).

This instantaneous transition from a partition with infinitely many elements to a finite number is analogous to the `coming down from infinity' property of certain non-spatial processes such as the Kingman coalescent (there the number of elements of the partition shrinks like $2/t$ \cite{kingman}). In the coalescent theory literature there has been recent interest in processes that incorporate `fragmentation' dynamics, in which large partition elements are stochastically replaced with many smaller ones. These dynamics act in tension with coalescence and may cause the process to remain infinite for all time. Berestycki has characterised a general class of exchangeable fragmentation-coalescence processes, providing conditions under which they do or do not come down from infinity \cite{berestycki}. In \cite{KPRS} the authors examined the behaviour of the so-called `fast fragmentation-coalescence' process, in which each fragmentation event adds an infinite number of new elements to the partition. This process exhibits a phase transition in excursions from infinity, controlled by the rate of fragmentation. 

In the present article, we ask if an analogous phenomenon is possible in the spatial setting described by coalescing Brownian motions. It turns out that it is, and that the resulting object we obtain is a particular self-similar subset of the Brownian web, which we name the \textbf{Brownian marble}.

\subsection{Preliminaries}
Before defining the Brownian marble, let us walk through the definitions of \textbf{coalescing Brownian motion} and the \textbf{Brownian web}. Let $z_1 = (t_1,x_1),\ldots,z_k = (t_k,x_k)$ be distinct elements of the spacetime domain $[0,\infty) \times \mathbb{R}$ (with $[0,\infty)$ corresponding to time and $\mathbb{R}$ to space), and consider coalescing Brownian motions $B^1,\ldots,B^k$ initiated from these $k$ points so that $B^i_{t_i} = x_i$. These Brownian motions are initially independent but coalesce when they meet, and are coupled for the remainder of time, so that
\begin{align*}
B^i_s = B^j_s \text{ for some $s \geq 0$ } \implies B^i_u = B^j_u \text{ for all $u \geq s$}.
\end{align*}

It is possible to make sense of coalescing Brownian motions initiated from an infinite countable set. Loosely speaking, if the infinite countable set is a dense subset $\{z_1,z_2,\ldots \}$ of the $y$-axis $\{0\} \times \mathbb{R}$, we simply refer to the process as \textbf{coalescing Brownian motion}. If this infinite countable set is a dense subset of $[0,\infty) \times \mathbb{R}$ itself, we refer to the process as the \textbf{Brownian web}. (The Brownian web is usually defined on all of $\mathbb{R}^2$, but for our purposes it will be more suitable to consider it on the spacetime domain $[0,\infty) \times \mathbb{R}$.)

In the present article we will be interested in an intermediate class of processes in which coalescing Brownian motions are issued from a certain random subset of $[0,\infty) \times \mathbb{R}$. We give an informal definition now.

\begin{definition}[Informal definition of the $R$-marble]
Let $R:(0,\infty) \to [0,\infty)$ be a suitable function.
The \textbf{R-marble} is a random collection of paths associated with coalescing Brownian motion but with the additional dynamic that at instantaneous rate $R(U_t-L_t)$ an interval $[L_t,U_t]$ between two consecutive Brownian motions fragments into dust. At this fragmentation event, we initiate new coalescing Brownian motions (or paths) from every point of the spacetime interval $\{t\} \times [L_t,U_t]$.
\end{definition}

In the definition of the $R$-marble, all intervals experience fragmentation events independently of other intervals in the process. 

The construction of the $R$-marble with a bounded rate function is fairly straightforward (we outline the details in a moment in Section \ref{sec:marbleconstruction}). In this article however, we will be most interested in making sense of the $R$-marble with the unbounded rate function $R(g) = \lambda/g^2$, which grants the process self-similarity. Notwithstanding the technical challenges associated with making sense of such a process, we give an initial definition of the main object we study:

\begin{definition} \label{df:marble}
The \textbf{Brownian marble} with fragmentation parameter $\lambda \geq 0$ is the $R$-marble with fragmentation rate
\begin{align*}
R_\lambda(g) = \lambda/g^2.
\end{align*}
\end{definition}
Definition \ref{df:marble} is somewhat premature, in that it immediately raises several technical issues. First of all, the fragmentation rate $g \mapsto \lambda/g^2$ is unbounded as $g \downarrow 0$, so that any fragmentation event is immediately followed by a cascade of further fragmentation events as the small subintervals created by a fragmentation event themselves undergo further fragmentation. On the other hand, as fragmentations occur, countless small gaps begin to emerge, and, loosely speaking, there are infinitely many opportunities for a gap between two consecutive Brownian motions to avoid fragmentation and take on significant size. In this sense there is an intricate interplay between the infinite rate of fragmentation near zero, and the infinite opportunities for gaps between the Brownian motions to emerge. We will see in particular that a phase transition occurs at $\lambda = 6$: gaps emerge if and only if $\lambda < 6$.

\subsection{The Brownian marble: construction and existence} \label{sec:technical}

In this section we outline the definitions necessary to make more precise sense of the $R$-marble. 
Coalescing Brownian motion, the Brownian web, and the $R$-marble may each be constructed as an $\mathcal{H}$-valued random variable, where $\mathcal{H}$ is a collection of paths in spacetime. We will provide the full details in Section \ref{sec:spaces}, but will outline the key details here. 
Let 
\begin{align*}
\Pi := \{ \pi : [\sigma_\pi, \infty) \to \mathbb{R} \text{ continuous}, \sigma_\pi \geq 0 \}
\end{align*}
be the space of paths $\pi = (\pi_t)_{t \geq \sigma_\pi}$ in $[0,\infty) \times \mathbb{R}$. These are continuous real-valued functions defined on some subinterval $[\sigma_\pi,\infty)$ of $[0,\infty)$. It is possible to introduce a metric $\mathrm{d}_{\mathrm{path}}$ on paths: paths $\pi_1$ and $\pi_2$ are close if their starting times are close and they are uniformly close on compact intervals. One can consider the metric space 
\begin{align*}
\mathcal{H} := \{ \mathcal{M} : \mathcal{M} \subseteq \Pi \}
\end{align*}
of collections of paths endowed with the Hausdorff metric $\mathrm{d}_{\mathcal{H}}$ (i.e.\ two subsets $K_1,K_2$ of $\mathcal{H}$ are close to one another if each path of $K_1$ is close to some path in $K_2$, and vice versa). 

\subsubsection{The Brownian web}
The Brownian web may be cast as an $\mathcal{H}$-valued random variable $\mathcal{W}$ as follows. Let $\mathcal{D} = \{z_1,z_2,\ldots\}$ be a dense subset of $[0,\infty) \times \mathbb{R}$, and let $\mathcal{W}_n$ be the collection of paths associated with coalescing Brownian motions started from $\{z_1,\ldots,z_n\}$. We define the Brownian web $\mathcal{W}$ as the closure of the union $\bigcup_{n \geq 1} \mathcal{W}_n$ in the topology induced by the Hausdorff metric. 

\subsubsection{Coalescing Brownian motion}
To construct coalescing Brownian motion in place of the Brownian web, simply replace the dense subset of $[0,\infty) \times \mathbb{R}$ with a dense subset of $\{0\} \times \mathbb{R}$ in the previous construction.

\subsubsection{Bubbles} 
Before discussing the construction of the $R$-marble, we begin by introducing the notion of a \textbf{bubble} for a $\mathcal{H}$-valued random variable. For a collection of paths $\mathcal{M}$ in $\mathcal{H}$, we define the \textbf{trace} of $\mathcal{M}$ to be the set 
\begin{align} \label{eq:image}
    \mathrm{Tr}(\mathcal{M}) := \overline{ \{ (t,x) \in [0,\infty) \times \mathbb{R} : \exists \pi \in \mathcal{M} : t \geq \sigma_\pi \text{ and } \pi_t = x \} } ,
\end{align}
where the overline denotes closure in $[0,\infty) \times \mathbb{R}$. 
That is, $\mathrm{Tr}(\mathcal{M})$ is the closure of the set of all points in $[0,\infty) \times \mathbb{R}$ that are crossed by some path in $\mathcal{M}$. 

\begin{definition}
 A \textbf{bubble} of $\mathcal{M}$ is a maximal connected open subset in the complement $[0,\infty) \times \mathbb{R} - \mathrm{Tr}(\mathcal{M})$.
\end{definition}

If $\mathcal{W}$ is a copy of the Brownian web, then $\mathrm{Tr}(\mathcal{W}) = [0,\infty) \times \mathbb{R}$ almost surely. Accordingly, the Brownian web almost surely has no non-empty bubbles. Conversely, if $\mathcal{M}_0$ is a copy of coalescing Brownian motion, the Lebesgue measure of $\mathrm{Tr}(\mathcal{M}_0)$ is zero almost surely, and its complement in $[0,\infty) \times \mathbb{R}$ is equal to a countable union of open subsets of $[0,\infty) \times \mathbb{R}$. Each point $z = (t,x)$ of $(0,\infty) \times \mathbb{R}$ is almost surely contained in a bubble of $\mathcal{M}_0$. In fact, if we write $\mathcal{B}^z$ for the bubble of $\mathcal{M}_0$ containing $z$, then there exist start and end times $\sigma_z$ and $\tau_z$, and upper and lower boundaries $L^z$ and $U^z$ such that
\begin{align} \label{eq:bubbly}
\mathcal{B}^z = \{ (t,x) : \sigma_z < t < \tau_z, L^z_t < x < U^z_t \}.
\end{align}
We define the \textbf{height process} of a bubble to be the stochastic process $(g_t)_{\sigma_z < t < \tau_z}$ given by $g_t = U^z_t - L^z_t$. 

\subsubsection{The $R$-marble with bounded $R$} \label{sec:marbleconstruction}
With the notion of a bubble at hand, we are now ready to outline the construction of the $R$-marble with bounded and measurable rate function $R:(0,\infty) \to [0,\infty)$. 

Let $\mathcal{M}_0$ be a copy of coalescing Brownian motion. We will construct a sequence $\mathcal{M}_0 \subseteq \mathcal{M}_1 \subseteq \mathcal{M}_2 \subseteq \ldots $ of $\mathcal{H}$-valued random variables by using a dense sequence $\mathcal{D} = (z_1,z_2,\ldots)$ of points of $(0,\infty) \times \mathbb{R}$ to explore bubbles and potentially fragment them into smaller bubbles. The $R$-marble $\mathcal{M}_\infty$ will be constructed as the closure of the union $\bigcup_{n \geq 1} \mathcal{M}_n$ of this sequence.

We will refer to a bubble of $\mathcal{M}_k$ as \emph{permanent} if it is also a bubble in $\mathcal{M}_\infty$, and as \emph{temporary} otherwise. For coalescing Brownian motion, each point $z = (t,x)$ of $(0,\infty) \times \mathbb{R}$ is almost surely contained in a bubble (i.e., is not part of a path). Starting at step zero, take $\mathcal{M}_0$ and label all of its bubbles as `temporary'. 

Let $\mathcal{B}_1$ be the (temporary) bubble of $\mathcal{M}_0$ containing $z_1$. (It is almost surely the case that $z_1$ lies in a bubble $\mathcal{B}_1$.) Then we can write $\mathcal{B}_1 = \{ (t,x) : \sigma_1 < t < \tau_1, L^1_t < x < U^1_t \}$. Run a stochastic clock that rings at rate $R(U^1_t-L^1_t)_{t \in [\sigma_1,\tau_1]}$. This rate is well-defined as $R$ is bounded. If this clock does not ring during $[\sigma_1,\tau_1]$, set $\mathcal{M}_1 = \mathcal{M}_0$, and demarcate the bubble $\mathcal{B}_1$ as a permanent bubble. Otherwise, if the clock rings at some time $W_1 \in (\sigma_1,\tau_1)$, we declare $\mathcal{B}_1' := \mathcal{B}_1 \cap \{ (t,x) : t < W_1\}$ to be a permanent bubble, but initiate new coalescing Brownian motions from each point of $\{W_1\} \times [L^1_{W_1},U^1_{W_1}]$. These new paths coalesce with one another and with the paths of $\mathcal{M}_0$. Let $\mathcal{M}_1$ denote the union of $\mathcal{M}_0$ together with these new paths initiated from $\{W_1\} \times [L^1_{W_1},U^1_{W_1}]$. We note that, up to a set of Lebesgue measure zero, the temporary bubble $\mathcal{B}_1$ is now a union of the permanent bubble $\mathcal{B}_1'$ together with some new (in fact a countably infinite number of) temporary bubbles of $\mathcal{M}_1$. 

Generally, after creating $\mathcal{M}_k$, let $z_{k+1}$ be the $(k+1)^{\text{th}}$ point in $\mathcal{D}$. If $z_{k+1}$ lies in a permanent bubble, set $\mathcal{M}_{k+1} = \mathcal{M}_k$ and proceed to $z_{k+2}$. If not, let $\mathcal{B}_{k+1}$ be the temporary bubble of $\mathcal{M}_k$ containing $z_{k+1}$ and proceed according to the above rule by running a stochastic clock inside $\mathcal{B}_{k+1}$, that will either establish $\mathcal{B}_{k+1}$ as a permanent bubble or break into a permanent bubble and some new temporary bubbles. 

While, for every $k \geq 1$, there will exist almost surely some $k' > k$ such that $\mathcal{M}_{k'} \neq \mathcal{M}_k$, it is possible to show that the sequence $(\mathcal{M}_k)_{k \geq 1}$ converges locally in the following sense. Namely, we now argue that for each fixed $z  = (t,x)$ in spacetime, there exists almost surely some $k \geq 1$ such that the bubble of $\mathcal{M}_{k'}$ containing $z$ is the same for all $k' \geq k$. To see this, consider the temporary bubble of $\mathcal{M}_0$ containing $z$, and suppose that it has duration $\ell_z = \tau_z - \sigma_z$ (see \eqref{eq:bubbly}). 
Then the probability that when this bubble is first explored it is made permanent is at least $q := e^{ - R_{\mathrm{max}}\ell_z }$, where $R_{\mathrm{max}} := \sup_{g > 0 } R(g)$. 
If this bubble is not made permanent, a new temporary bubble containing $z$ of length $< \ell_z$ is created in some $\mathcal{M}_k$. The probability that this bubble is made permanent when it is next explored is also at least $q$. 
Proceeding like this, we see that the number of distinct temporary bubbles containing $z$ occurring in the sequence $(\mathcal{M}_k)_{k \geq 0}$ is bounded above by a geometric random variable with parameter $q > 0$. In other words, the bubble containing $z$ in the sequence $(\mathcal{M}_k)_{k\geq 0}$ is eventually constant.

It is also easy to see that the law of the limiting random variable $\mathcal{M}_\infty$ does not depend on the dense set $\mathcal{D}$ --- the dense set only changes the order in which the bubbles are explored, and not the law of the underlying limit random variable.

\subsection{The $R$-marble with unbounded $R$ and the phase transition}

Our next result states that under mild conditions, it is also possible to construct the $R$-marble with an unbounded rate function $R$. We say that a rate function $R$ is \textbf{quadratic-regular} if it is nonincreasing and it satisfies one of the three following mutually exclusive conditions: either
\begin{align} \label{eq:squareregular}
\liminf_{ g \downarrow 0} g^2 R(g) > 6 \quad \text{or} \quad \limsup_{ g \downarrow 0} g^2 R(g) < 6  \quad \text{or} \quad R(g) = 6/g^2 \text{ for all $g  > 0$}.
\end{align}
We emphasise that quadratic-regular functions are nonincreasing. 
We refer to quadratic-regular functions satisfying $\liminf_{g \downarrow 0} g^2 R(g) \geq 6$ (i.e., the first and last case of \eqref{eq:squareregular}) as \textbf{upper quadratic-regular}, and those satisfying the second condition $\limsup_{g \downarrow 0} g^2 R(g) < 6$ as
\textbf{lower quadratic-regular}. 

Let $a \wedge b$ denote the minimum of real numbers $a$ and $b$. Our main result states that if $R$ is quadratic-regular then we can construct the $R$-marble as an almost-sure limit of $R_n$-marbles with truncated version $R_n(g) = R(g) \wedge n$ of this original rate function, and that moreover, a phase transition occurs according to the limiting behaviour of $g^2R(g)$ as $g \downarrow 0$.

\begin{theorem} \label{thm:main}
Let $R$ be quadratic-regular. Then we may construct a probability space carrying a sequence of $\mathcal{H}$-valued random variables $(\mathcal{M}(R_n))_{n \geq 1}$ such that $\mathcal{M}(R_n)$ has the law of the $R_n$-marble with the truncated rate function $R_n(g) = R(g) \wedge n$, and such that we have the almost-sure convergence
\begin{align*}
\mathcal{M}(R_n) \to \mathcal{M}(R)
\end{align*}
to a limit random variable $\mathcal{M}(R)$ in the Hausdorff metric $\mathrm{d}_{\mathcal{H}}$. We call the limit random variable $\mathcal{M}(R)$ the $R$-marble.

Moreover, we have the following phase transition:
\begin{itemize}
\item If $\liminf_{g \downarrow 0} g^2R(g) \geq 6$, then the limit random variable $\mathcal{M}(R)$ has the law of the Brownian web, and in particular has no bubbles.
\item If $\limsup_{g \downarrow 0} g^2R(g)  < 6$, then the limit random variable $\mathcal{M}(R)$ has the property that every point $z = (t,x)$ of $(0,\infty) \times \mathbb{R}$ is almost surely contained in a bubble.
\end{itemize}
\end{theorem}

A few comments are in order. First of all, having only defined thus far the $R$-marble for bounded $R$ in Section \ref{sec:marbleconstruction}, Theorem \ref{thm:main} allows us to speak of the $R$-marble with any quadratic-regular rate function $R$, even if it is unbounded; this $R$-marble is constructed as a limit of $R_n$-marbles, where $R_n(g) = R(g) \wedge n$ is the truncated rate function.  
We emphasise that Theorem \ref{thm:main} is stronger than simply stating that the random variables $\mathcal{M}(R_n)$ converge in distribution; rather, it says that we can take a sequence of random variables $(\mathcal{M}(R_n))_{n \geq 1}$ in the same probability space (such that $\mathcal{M}(R_n)$ has the law of the $R_n$-marble) that converge in the Hausdorff metric almost surely.

The most interesting aspect of Theorem \ref{thm:main}, however, is the phase transition that occurs at $6$ for the limiting value of $g^2 R(g)$ as $g \downarrow 0$. On the one hand, Theorem \ref{thm:main} states that 
 when $\liminf_{g \downarrow 0} g^2 R(g) \geq 6$, the fragmentation mechanism is so strong as to prevent the formation of any gaps between the Brownian motions of any significant size. Since there are no gaps between Brownian motions, the process simply consists of coalescing Brownian motions initiated from every point of $[0,\infty) \times \mathbb{R}$; i.e.\ the process is the Brownian web. When $\limsup_{g \downarrow 0} g^2 R(g) < 6$, however, the coagulation is sufficiently strong compared to the fragmentation for bubbles to have a chance to emerge. This phenomenon may be regarded as a form of \emph{coming down from infinity} \cite{KPRS, BMS}.

The attentive reader will notice that we have omitted from our consideration the critical case where $R$ is a decreasing rate function satisfying $\limsup_{g \downarrow 0} g^2R(g) = 6$ but not taking the form $R(g) = 6/g^2$. We are not sure what happens in this case, but provisionally conjecture it is also the case here that $\mathcal{M}(R)$ is equal in law to the Brownian web.

\begin{figure}[h!]
\centering
\includegraphics[width=0.9\textwidth]{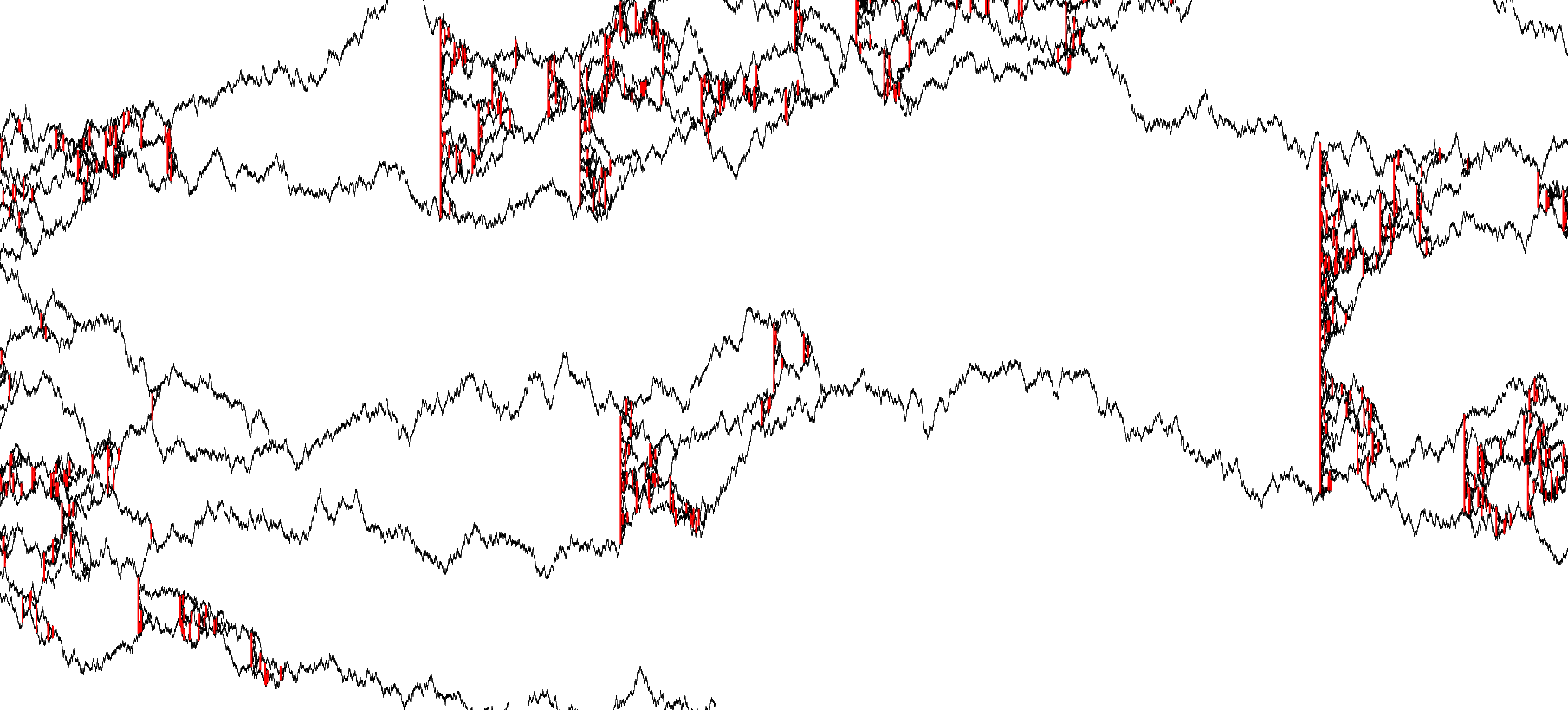}
\caption{A simulation of the Brownian marble with parameter $\lambda = 3$. We initiate coalescing Brownian motions at each point of $\{0\} \times \mathbb{R}$. Thereafter, at rate $\lambda/(U_t-L_t)^2$ an interval $[L_t,U_t]$ between two consecutive paths fragments. At such a fragmentation event we initiate new coalescing Brownian paths from each point $\{t\} \times [L_t,U_t]$. The fragmentation events are depicted in red. 
}
\label{fig:timmarble}
\end{figure}

\subsection{The Brownian marble and its distributional properties} \label{sec:distributional}

As highlighted above, Theorem \ref{thm:main} licences us to speak of the $R$-marble for any quadratic-regular rate function $R$. In particular, we may now speak of the \textbf{Brownian marble} with parameter $\lambda \geq 0$ as the $R_\lambda$-marble $\mathcal{M}(R_\lambda)$ with rate function
\begin{align*}
R_\lambda(g) = \lambda/g^2.
\end{align*}

Of course, the Brownian marble with parameter $\lambda = 0$ is simply coalescing Brownian motion. Conversely, by Theorem \ref{thm:main}, the Brownian marble with any parameter $\lambda \geq 6$ is identical in law to the Brownian web on $[0,\infty) \times \mathbb{R}$. However, for each intermediate value $\lambda$ in $(0,6)$, we have an $\mathcal{H}$-valued random variable $\mathcal{M}(R_\lambda)$ whose behaviour is somewhere between that of coalescing Brownian motion and the Brownian web. Every point $(t,x)$ of $(0,\infty) \times \mathbb{R}$ almost surely lies in a bubble of the Brownian marble whenever $\lambda \in [0,6)$. See the figure on the title page for a simulation of the Brownian marble with parameter $\lambda = 3$ where the different bubbles are shaded in different colours, and see Figure \ref{fig:timmarble} for a simulation of a Brownian marble with $\lambda = 3$ with an emphasis on the Brownian paths separating the bubbles.

The Brownian marble $\mathcal{M}(R_\lambda)$ is self-similar in the sense that the process is invariant under any rescaling in which space is rescaled by $c$ and time is rescaled by $c^{-2}$. More specifically, for $\pi \in \mathcal{M}(R_\lambda)$ define a new path $\pi_c$ with starting point $\sigma_{\pi_c} = c^2 \sigma_\pi$ and 
\begin{align*}
\pi_c(t) := \frac{1}{c} \pi(c^2 t).
\end{align*}
Then for any $c > 0$, the rescaled collection of paths $\mathcal{M}(R_\lambda)_c := \{ \pi_c : \pi \in \mathcal{M}(R_\lambda) \}$ has the same law as $\mathcal{M}(R_\lambda)$. 
This follows from the Brownian scaling property, together with the fact that the fragmentation rate function $R(g) = \lambda/g^2$ has the simple scaling property $R_\lambda(cg) = \frac{1}{c^2} R_\lambda(g)$. 

\begin{figure}[h!]
\centering
\includegraphics[width=0.8\textwidth]{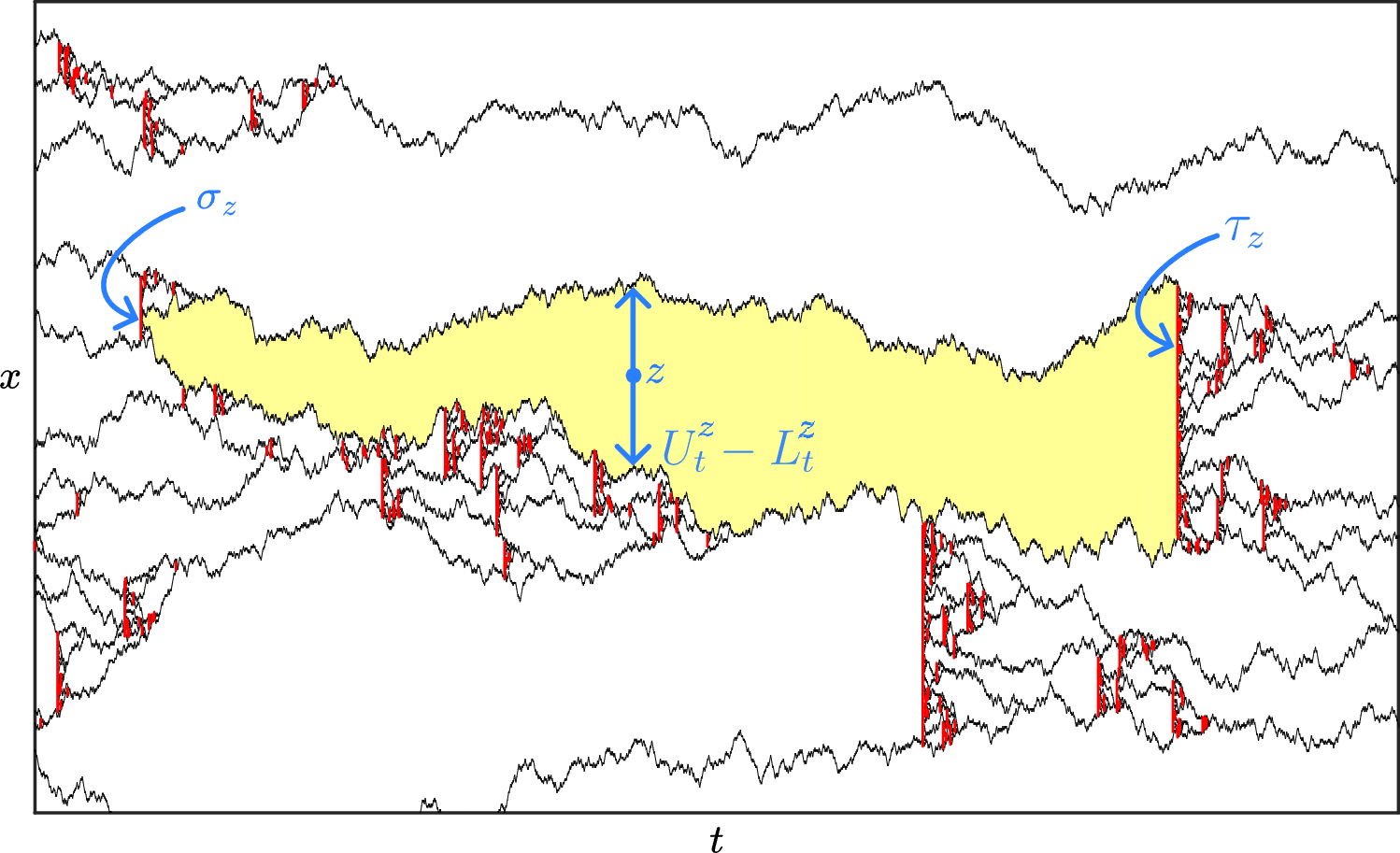}
\caption{The bubble $\mathcal{B}^z$ containing a fixed point $z = (t,x)$ is depicted in yellow. The birth time of this bubble is $\sigma_z$, the death time is $\tau_z$, and for $s \in [\sigma_z,\tau_z]$, the height process of this bubble is the stochastic process $(g^z_s)_{s \in [\sigma_z,\tau_z]}$ defined by $g^z_s =U^z_s - L^z_s$.}
\label{fig:bubble}
\end{figure}

This self-similarity endows the process with a host of nice distributional properties. 
For example, let $\mathbb{P}^{(\lambda)}$ denote the law of the Brownian marble with parameter $\lambda \in [0,6)$. Then we are able to give a precise probabilistic description of the 
marginal law under $\mathbb{P}^{(\lambda)}$ of the bubble $\mathcal{B}^z$ containing a point $z= (t,x)$ of $(0,\infty) \times \mathbb{R}$. Recall from \eqref{eq:bubbly} that we write $\mathcal{B}^z := \{ (s,y) : \sigma_z < s < \tau_z , L^z_s < y < U^z_s \}$
for the bubble containing $z$. See Figure \ref{fig:bubble} for a diagram. Then various functionals of this bubble have explicit distributions under $\mathbb{P}^{(\lambda)}$. For example, we will show that the law of the birth time $\sigma_z$ of the bubble containing $z = (t,x)$ satisfies
\begin{align} \label{eq:bubblestart}
\mathbb{P}^{(\lambda)} \left( \sigma_z/t < \zeta \right) := \frac{\sin (\pi \beta)}{\beta} \int_0^\zeta u^{-(1-\beta)} (1-u)^\beta \mathrm{d}u \quad \text{where $\beta = \frac{1}{2}(\sqrt{4 \lambda+1}-1)$.}
\end{align}
That is, $\sigma_z/t$ is $\mathrm{Beta}(\beta_\lambda,1-\beta_\lambda)$ distributed under $\mathbb{P}^{(\lambda)}$. 
Observe that as $\lambda \uparrow 6$, the parameter $\beta = \beta_\lambda \uparrow 1$, and the random variable $\sigma_z/t$ converges in distribution to the point mass at one. In other words, as $\lambda$ increases up to $6$, the bubble containing a given point $z = (t,x)$ becomes more and more likely to have started shortly before time $t$. 

To provide another example here, we have the expression
\begin{align} \label{eq:bubbleheight}
\mathbb{P}^{(\lambda)}( (U^z_t-L^z_t)^2/2t < \zeta ) = \Gamma(c_\lambda)^{-1} \int_0^\zeta u^{c_\lambda} e^{ -u} \mathrm{d}u/u, \quad c_\lambda = (6-\lambda)/4,
\end{align}
for the law of the height $U^z_t - L^z_t$ of the bubble containing $z$ at the vertical cross-section containing $z$. In other words, $(U^z_t-L^z_t)^2/2t$ is Gamma distributed with shape parameter $c_\lambda$ (and rate parameter $1$). We note that the parameter $c_\lambda$ is decreasing in $\lambda$. Since Gamma random variables with larger shape parameters tend to be larger, this reflects the fact that the bubbles of the Brownian marble are typically larger when $\lambda$ is smaller, and that the size of a typical bubble converges to zero as $\lambda \uparrow 6$. In fact, \eqref{eq:bubbleheight} can be shown to follow from \eqref{eq:bubblestart} and the following more general result:

\begin{theorem} \label{thm:dbessel}
Let $z = (t,x)$ be a point of $(0,\infty) \times \mathbb{R}$, and let $\sigma_z$ be the birth time of the bubble containing $z$. Define a stochastic process by setting
\begin{align*}
X^z_s :=  \frac{U^z_s-L^z_s}{\sqrt{2}}, \qquad s \in [\sigma_z,\tau_z].
\end{align*}
Conditional on $\sigma_z$, the stochastic process $(X^z_s)_{s \in [\sigma_z,t]}$ has the law of a Bessel process of dimension \[d_\lambda = 3+\sqrt{4 \lambda + 1}\] initiated from $X^z_{\sigma_z} = 0$, but with its law size-biased by $X_t^{-\beta_\lambda}$, where $\beta_\lambda := \frac{1}{2}(\sqrt{4 \lambda + 1} - 1)$. After time $t$, $(X^z_s)_{s \in [t, \tau_z]}$ has the law of a standard Brownian motion killed at rate $\lambda/(2(X^z_s)^2)$, with $\tau_z$ denoting the killing time. 
\end{theorem}

Note that the self-similarity of the Brownian marble is implicit in the equations \eqref{eq:bubblestart}, \eqref{eq:bubbleheight} and Theorem \ref{thm:dbessel}.

In fact, for any fixed points $z_1,\ldots,z_k$ of $[0,\infty) \times \mathbb{R}$ we are able to give an intrinsic probabilistic description of the joint distribution of the bubbles $\mathcal{B}^{z_1},\ldots,\mathcal{B}^{z_k}$ in terms of objects which we call interacting Brownian veins. Of course, it is entirely possible if $z_i$ is close to $z_j$ that these points are contained in the same bubble, so that $\mathcal{B}^{z_i} = \mathcal{B}^{z_j}$. See Section \ref{sec:intrinsic} for further details.



\subsection{The branching analogue}\label{sect:branching}

We close our introductory discussion by mentioning that the $R$-marble has a natural branching analogue, which is obtained, roughly speaking, by letting the sizes of the intervals in the $R$-marble fluctuate independently of one another. Of course, in the $R$-marble the sizes of neighbouring intervals are dependent in that the boundary between them fluctuates according to a common Brownian motion, so that if one gets larger the other gets smaller.

We will outline the construction using a slightly different truncation which will be more suitable for the branching structure.
 Let $N \geq 1$ be our degree of truncation. Our branching process begins at time zero with, say, $k$ particles with respective masses $a_1,\ldots,a_k > 0$. As time passes, the particles in the process experience the following dynamics, independently of all other particles in the system:
\begin{itemize}
\item The mass process $(g_t)_{t \geq 0}$ of each particle fluctuates according to $\sqrt{2}$ times a standard Brownian motion. If the size of a particle ever hits zero, it is killed. 
\item At instantaneous rate $R(g_t)$, a particle of mass $g_t$ breaks into $N$ equally sized particles of mass $g_t/N$. 
\end{itemize}
We write $P^{R,N}_{a_1,\ldots,a_k}$ for the law of this process starting with $k$ particles of sizes $a_1,\ldots,a_k$. 

This process may be formulated as a measure-valued branching process $(\mu_t)_{t \geq 0}$, where, if at time $t$ there are $N(t)$ particles alive in the system of sizes $a_1,\ldots,a_{N(t)}$ (listed in some order), then $\mu_t$ is the sum of Dirac masses
\begin{align*}
\mu_t := \sum_{i=1}^{N(t)} \delta_{a_i}.
\end{align*}
This process has the branching property in that for suitably regular functions $\phi:(0,\infty) \to \mathbb{R}$, writing $\langle \phi, \mu \rangle := \int_{(0,\infty)} \phi(x) \mu_t(\mathrm{d}x)$ we  have
\begin{align*}
\mathbf{E}_{\mu_0} [ \exp \left\{ - \langle \phi, \mu_t \rangle \right\} ]= \exp \left\{ - \langle v_t(\phi) , \mu_0 \rangle \right\},
\end{align*}
for a certain semigroup of operators $(v_t)_{t \geq 0}$. See Li \cite{li} or Etheridge \cite[Chapter 1]{etheridge} for further information.

Drawing on an analogy with Theorem \ref{thm:main}, we conjecture the following about this measure-valued stochastic process:

 \begin{conj} \label{conj:branching}
Let $R$ be quadratic-regular, in that it satisfies one of the conditions in \eqref{eq:squareregular}. 
Let $(\mu^N)_{N \geq 1}$ be a sequence of measure-valued stochastic processes $\mu^N := (\mu_t^N)_{t \geq 0}$  such that for each $N$, $\mu^N$ is distributed according to $P^{R,N}_{1/N,\ldots,1/N}$, the subscript denoting that we begin with $N$ distinct particles of size $1/N$. 
Then we have the following phase transition: 
\begin{itemize}
\item If $\limsup_{g \downarrow 0} g^2R(g)  < 6$, then $(\mu^N_t)_{t >0}$ converges in distribution to a nondegenerate stochastic process $(\mu_t)_{t \geq 0}$ such that for each $t > 0$, $\mu_t$ is almost-surely a finite point measure on $(0,\infty)$. 
\item If $\liminf_{g \downarrow 0} g^2R(g) \geq 6$, $(\mu_t^N)_{t \geq 0}$ converges in distribution to the trivial stochastic process $(\mu_t)_{t \geq 0}$ such that $\mu_t$ is almost surely the zero measure on $(0,\infty)$ for all $t \geq 0$.
\end{itemize}
\end{conj}

This conjecture is based on the understanding that a bubble of the $R$-marble is somehow equivalent to the lifetime of a particle of positive size in the above branching process.

To make precise sense of the convergence in distribution in Conjecture \ref{conj:branching}, one may formulate each $(\mu_t^N)_{t \geq 0}$ as a stochastic process in the Skorokhod topology $D((0,\infty),\mathrm{M}(0,\infty))$ where 
$\mathrm{M}(0,\infty)$, is the space of  measures on $(0,
\infty)$ equipped with the vague topology.



On the event that Conjecture \ref{conj:branching} is true, it would then be particularly interesting to look at the case where the rate function takes the form $R(g) = \lambda/g^2$, as in this case the branching process would enjoy self-similarity in space and time.

\subsection{Related work}

The notion of the Brownian web was first considered by Arratia \cite{arratia} in his PhD thesis, but 
it was not until work of Fontes et al.\ \cite{FINR} (preceded by Toth and Werner \cite{toth}), that the precise definition of the process was established in a suitable topology. To this day, there is a large body of work on the Brownian web and its variants \cite{SSS, SS, FINR, FINR2, EF, BGS, HW, RSS, CSST, CH, CH2,VV}.

A related object called the \emph{Brownian net} 
can be seen as another instance of branching structure embedded in the Brownian web. At the discrete level, the discrete net is obtained by considering nearest-neighbour one dimensional random walks branching into two particles at neighbouring sites with probability $\epsilon$. When 
$\epsilon\to0$, the resulting set of branching-coalescing paths converges to the Brownian net after diffusive scaling. See \cite{SSS} for a review.

 

The $R$-marble process may be regarded as a spatially dependent growth-fragmentation process. 
Such processes have been studied extensively by Bertoin, see e.g.\ the research monograph \cite{bertoinbook}, as well as work by Berestycki \cite{berestycki}. 

Kyprianou et al.\ \cite{KPRS} consider a variant of Kingman's coalescent starting with singleton blocks indexed by the integers, and where every pair of blocks coalesce at rate $c > 0$, and every block breaks into its constituent singletons at rate $\lambda > 0$. They show that for each fixed $t > 0$, the block counting process is almost surely finite 
if and only if $\lambda/2c < 1$. See also \cite{KPR}. 

The intrinsic probabilistic description we provide for the emergence of bubbles in the $R$-marble relies heavily on a beautiful identity which appears explicitly in Warren \cite{warren}, but appears implicitly in earlier work due to Dubedat \cite{dubedat} and Soucaliuc, Toth and Werner \cite{STW00}. See also \cite{toth, SW, BN}. Related processes involving interlacing and intertwining diffusions are considered in \cite{ANM, GS, oconnell}. Discrete time and space analogues of these identities have also been noted in the integrable probability literature, see e.g.\ \cite{assiotis}. 

The Brownian marble is a spatial analogue of a self-similar Markov process \cite{rivero1, rivero2, CPR, BY, JY, patie, CKPR}. The weak convergence of L\'evy processes to self-similar processes (which is similar at heart to our work in Section 2), is considered in \cite{CC}.

\subsection{Overview}
The remainder of this article is structured as follows.
\begin{itemize}
\item In Section \ref{sec:bessel} we construct a process called the $R$-Bessel process. This is a Bessel-$3$ process $X := (X_t)_{t \geq 0}$ with the additional dynamic that it jumps back to the origin at instantaneous rate $R(\sqrt{2}X_t)$. We will see in the sequel that $X$ is a proxy for the emergence of bubbles in the $R$-marble. The main result of this section, Theorem \ref{thm:newtransition}, states that the $R$-Bessel process undergoes a phase transition like that in Theorem \ref{thm:main}.

\item In the following section, Section \ref{sec:vein}, we introduce a generalisation of the $R$-Bessel process called the $R$-vein.   
\item In Section \ref{sec:explore} we discuss how the $R$-marble can have its bubbles explored by using dual web paths emitted backwards in time from a dense set. The emergence of bubbles along these dual paths leads to copies of the $R$-veins introduced in Section \ref{sec:vein}. 

\item In Section \ref{sec:mainproof}, we reverse the exploration of the previous section, and use veins leading to a dense set of points to \emph{construct} the $R$-marble. This construction is robust to different choices of rate function $R$, and allows the construction of the $R$-marble for a quadratic-regular $R$ as a limit of truncated processes. This leads to a proof of our main result, Theorem \ref{thm:main}.

\item In a brief final section, Section \ref{sect:spine}, 
we show that the $R$-Bessel process can be interpreted as the spinal process for the growth-fragmentation process as introduced in Section \ref{sect:branching}. This observation is a natural first step towards proving Conjecture \ref{conj:branching}.
\end{itemize}

A word on notation: throughout the article we will use blackboard letters $\mathbb{P}^R$ and $\mathbb{P}^{(\lambda)}$ to refer to the probability laws for the $R$-marble and the Brownian marble with parameter $\lambda \geq 0$, and $\mathbf{P}^{(\alpha)}_x,\mathbf{Q}_x^R$ etc.\ for probability laws governing (jump-)diffusions. 

\section{The $R$-Bessel process} \label{sec:bessel}

We will ultimately prove our main results on the construction, existence, phase transition, and distributional properties of the Brownian marble by tracing paths of the dual Brownian web backwards in time, and then, loosely speaking, following these paths forwards in time and observing how bubbles may emerge on these paths. It transpires that the formation of bubbles along these paths is intimately related to Bessel processes and their changes of measure. 

\subsection{Bessel processes of dimension $d = 2\alpha+1$} 

Recall that the Bessel-$3$ process (or the Bessel process of dimension $d=3$) can be defined by a change of measure of Brownian motion. If, under a probability law $\mathbf{P}^{(0)}_x$, $(X_t)_{t \geq 0}$ is a Brownian motion starting from $x > 0$, then we may define a change of measure $\mathbf{P}^{(1)}_x$ by setting 
\begin{equation} \label{eq:b3eq}
\mathrm{d}\mathbf{P}^{(1)}_x/\mathrm{d}\mathbf{P}^{(0)}_x := (X_t/x)\mathrm{1}\{ X_s > 0 ~\forall s\in[0,t] \}.
\end{equation}
Under $\mathbf{P}^{(1)}_x$, $(X_t)_{t \geq 0}$ has the law of a Bessel-$3$ process started from $x$. An alternative description is that under $\mathbf{P}^{(1)}_x$, $(X_t)_{t \geq 0}$ has the law of a Brownian motion started from $x$ and conditioned (in the sense of Doob) to never hit zero, i.e.\
\begin{align*}
\mathbf{P}^{(1)}_x(A ) := \lim_{t \uparrow \infty} \mathbf{P}^{(0)}_x ( A | X_s > 0 ~\forall ~s \in [0,t] ) \text{ for events $A$ in } \bigcup_{u \geq 0} \sigma( X_s : s \leq u ).
\end{align*}
Plainly then, $\mathbf{P}^{(1)}_x(X_t > 0 ~\forall t ) = 1$. 

It is possible to make sense of the limiting law $\mathbf{P}^1_0 := \lim_{ x \downarrow 0} \mathbf{P}_x^1$; we call this process the standard Bessel-$3$ process.

More generally, there is a notion of a Bessel process with dimension $d=  2\alpha+1$ (the parameter $d$ need not be an integer). We write $\mathbf{P}^{(\alpha)}_x$ for the law of this process initiated from $x > 0$. This law may be defined by a change of measure relative to the law of Brownian motion initiated from $x > 0$ via the Radon-Nikodym derivative
\begin{align} \label{eq:beschange}
\frac{ \mathrm{d} \mathbf{P}_x^{(\alpha)} }{ \mathrm{d} \mathbf{P}_x^{(0)} } \Bigg|_{\mathcal{F}_t} = N_t^{(\alpha)} := \mathrm{1} \{ X_s > 0 ~ \forall s \in (0,t] \} (X_t/x)^\alpha \exp \left\{ - \frac{\alpha(\alpha-1)}{2} \int_0^t \frac{\mathrm{d}s}{X_s^2} \right\};
\end{align}
see \cite{lawler}. Note \eqref{eq:b3eq} is the special case $\alpha=1$ of \eqref{eq:beschange}. 
It is a straightforward exercise using the It\^{o} formula to establish that $N_t^{(\alpha)}$ is indeed a $\mathbf{P}_x^{(0)}$-martingale. 
If $(X_t)_{t \geq 0}$ has the law $\mathbf{P}_x^{(\alpha)}$, then the stochastic process
\begin{align} \label{eq:brownian}
B^{(\alpha)}_t = X_t - X_0 - \alpha \int_0^t \frac{1}{X_s} \mathrm{d}s  
\end{align}
is a standard Brownian motion under  $\mathbf{P}_x^{(\alpha)}$.
The transition density of $(X_t)_{t \geq 0}$ under $\mathbf{P}_x^{(\alpha)}$ is given by  
\begin{align} \label{eq:besseltran}
\mathbf{P}^{(\alpha)}_x ( X_t \in \mathrm{d}y )  = \frac{1}{t} \frac{y^{\alpha+1/2}}{x^{\alpha-1/2}} \exp \left\{ - \frac{x^2+y^2}{2t} \right\} I_{\alpha - 1/2} (xy/t) \mathrm{d}y,
\end{align}
where for $\nu \geq 0$, $I_{\nu}$ is the modified Bessel function of the first kind with parameter $\nu$:
\begin{align*}
I_\nu(u) = \sum_{ k =0}^\infty \frac{ (u/2)^{2k+\nu}}{\Gamma(k+1) \Gamma(k+\nu +1)};
\end{align*}
see e.g.\ \cite[Chapter XI]{RY}.

Finally, using the chain rule for Radon-Nikodym derivatives, we may extract from \eqref{eq:beschange} the relation 
\begin{align} \label{eq:urquell}
\frac{ \mathrm{d} \mathbf{P}_x^{(\alpha)} }{ \mathrm{d} \mathbf{P}_x^{(1)} } \Bigg|_{\mathcal{F}_t} = N_t^{(\alpha/1)} := \mathrm{1} \{ X_s > 0 ~ \forall s \in [0,t] \} (X_t/x)^{\alpha-1} \exp \left\{ - \frac{\alpha(\alpha-1)}{2} \int_0^t \frac{\mathrm{d}s}{X_s^2} \right\}. 
\end{align}

\subsection{Recurrent extensions of Bessel-$3$ processes} \label{sec:introrec}

 We now introduce the following process, which we will see in the sequel acts as a proxy for the height process of an emerging bubble in the $R$-marble. Throughout the remainder of the article, $R$ and $R'$ will always denote measurable functions $R,R':(0,\infty) \to [0,\infty)$.

\begin{definition} \label{df:Rbes}
Let $R:(0,\infty) \to [0,\infty)$ be bounded. Under a probability law $\mathbf{Q}^{R}_x$, we define the \textbf{R-Bessel process} as the Markov process with the following dynamics:
\begin{itemize}
\item Starting from $x \geq 0$, $(X_t)_{t \geq 0}$ diffuses according to a Bessel-$3$ process.
\item At rate $R(\sqrt{2}X_t)$ the process jumps back to zero. 
\end{itemize}
\end{definition}

(The scaling by $\sqrt{2}$ in Definition \ref{df:Rbes} simplifies various formulas in the sequel.) 

Definition \ref{df:Rbes} characterises a well-defined jump-diffusion for any bounded function $R$.
We now aim to construct the $R$-Bessel process for any decreasing function $R$, even if it is unbounded. 
We begin with the following lemma:
\begin{lemma}
Let $R,R':(0,\infty) \to (0,\infty)$ be bounded and decreasing measurable functions satisfying $R'(g) \geq R(g)$. Then there is a coupling of stochastic processes $X = (X_t)_{t \geq 0}$ and $X' = (X_t)_{t \geq 0}$ with respective laws ${\bf Q}_x^R$ and ${\bf Q}_x^{R'}$ such that $X'_t \leq X_t$ for all $t \geq 0$. 
\end{lemma}
\begin{proof}
First of all, for periods under which both processes are continuous, we run the processes as independent Bessel-$3$ processes that coalesce when they meet. 
Now note that the fact that $R$ and $R'$ are decreasing means that we can couple the processes in such a way that will guarantee that $X'$ will jump to zero whenever $X$ does.
\end{proof} 

As a corollary, we make sense of the $R$-Bessel process for unbounded decreasing $R$.
\begin{lemma} \label{lem:beslim}
Let $R:(0,\infty) \to [0,\infty)$ be a decreasing function. Then there is a probability space carrying a sequence of stochastic processes $X^n := (X^n_t)_{t \geq 0}$ such that each $(X^n_t)_{t \geq 0}$ is distributed according to $\mathbf{Q}_x^{R \wedge n}$ and such that we have 
\begin{align*}
X^n_t \to X_t \text{ pointwise for each $t \geq 0$},
\end{align*}
for some limiting stochastic process $(X_t)_{t \geq 0}$. 
    \end{lemma}

\begin{proof}
Let $n' \geq n$. Then since $R(g) \wedge n'$ and $R(g) \wedge n$ are both decreasing and satisfy $R(g) \wedge n' \geq R(g) \wedge n$, the previous lemma guarantees a natural coupling between processes $X^{n'}$ and $X^n$ with respective laws $\mathbf{Q}_x^{R \wedge n'}$ and $\mathbf{Q}_x^{R \wedge n}$ such that $X^{n'}_t \leq X^n_t$ for each $t \geq 0$. Thus for each $t$ we have a sequence of random variables $X^1_t \geq X^2_t \geq \ldots \geq 0$. It follows that for each $t$ these random variables converge to a limit; we call this limit $X_t$. 
\end{proof}

If $R$ is decreasing, we write $\mathbf{Q}_x^R$ for the law of stochastic process occurring as a limit in Lemma \ref{lem:beslim}, and refer to this process as the $R$-Bessel process. Thus the $R$-Bessel process is now defined for all $R$ either bounded or decreasing.

Lemma \ref{lem:beslim} does not preclude the possibility that the limit process $(X_t)_{t \geq 0}$ is equal to the zero process under $\mathbf{Q}_0^R$. 
Indeed, bearing in mind our claim that $(\sqrt{2}X_t)_{t \geq 0}$ may be considered a proxy for the height process of the evolution of a sequence of bubbles in time in the $R$-marble, the following result may be regarded as a preliminary version of the phase transition characterised in our main result, Theorem \ref{thm:main}:

\begin{theorem} \label{thm:newtransition}
Let $R:(0,\infty) \to [0,\infty)$ be a quadratic-regular rate function (see \eqref{eq:squareregular}).  Then we have the phase transition: 
\begin{itemize}
\item If $\liminf_{g \downarrow 0} g^2R(g) \geq 6$,  then $\mathbf{Q}_0^R ( X_t = 0  ~ \forall t  \geq 0) = 1$.
\item If $\limsup_{g \downarrow 0} g^2R(g)  < 6$, then for all $t > 0$ we have $\mathbf{Q}_0^R( X_t = 0) = 0$.
\end{itemize}
\end{theorem}

The majority of Section \ref{sec:bessel} is dedicated to the proof of Theorem \ref{thm:newtransition}.

Our proof idea will involve coupling stochastic processes with law $\mathbf{Q}_0^R$ with that of $\mathbf{Q}_0^{R_\lambda}$ where $R_\lambda(g) = \lambda/g^2$ for a suitable value of $\lambda$. In fact, once we prove Theorem \ref{thm:newtransition} in the special case where $R$ takes the form $R=R_\lambda$ for some $\lambda > 0$, the general case will follow quickly from a coupling argument. 

Thus we will be particularly interested in the laws $\mathbf{Q}_0^{R_\lambda}$ associated with $R_\lambda(g) = \lambda/g^2$.

After the proof of Theorem \ref{thm:newtransition}, we will explore some distributional properties of the laws $\mathbf{Q}_0^{R_\lambda}$.

\subsection{Long term survival under $\mathbf{Q}_0^{R_\lambda \wedge \lambda/2}$} \label{sec:survival}

We now consider the probability law $\mathbf{Q}_x^{R_\lambda \wedge (\lambda/2)}$ associated with a fragmentation rate 
\begin{align} \label{eq:Rtrunc}
R(g) = (\lambda/g^2) \wedge \lambda/2. 
\end{align}
Truncating at this level simplifies various calculations, since with $R$ as in \eqref{eq:Rtrunc} we have 
\begin{align*}
R(\sqrt{2}x) = \frac{\lambda}{2} \frac{1}{(x \vee 1)^2}.
\end{align*}
In other words, the truncation point happens at the height $x=1$. 

Our next proposition describes, under $\mathbf{Q}_0^{R_\lambda \wedge \lambda/2}$, the asymptotic probability of no fragmentation occurring for a large period of time, and the conditional behaviour of $(X_s)_{s \geq 0}$ on this event.

\begin{proposition} \label{prop:larget}
Define $\alpha = \alpha_\lambda$ to be the positive solution to $\alpha(\alpha-1) = \lambda$, so that $\alpha_\lambda = (1 + \sqrt{4 \lambda+1})/2$. Then for some $C_\lambda > 0$ we have 
\begin{align} \label{eq:Funct30}
\lim_{t \to \infty} t^{\frac{\alpha-1}{2} } \mathbf{Q}_0^{R_\lambda \wedge \lambda/2}  \left(  X_s > 0~ \forall s \in (0,t] \right) = C_\lambda.
\end{align}
Moreover, given a bounded and measurable functional $G( X_u : 0 \leq u \leq 1) $ on continuous paths $(X_u: 0 \leq u \leq 1)$ we have
\begin{align} \label{eq:Funct40}
&\lim_{t \to \infty} \mathbf{Q}_0^{R_\lambda \wedge \lambda/2}  \left[G(t^{-1/2}X_{ut} : 0 \leq u \leq 1  ) \Big|  X_s > 0 ~\forall s \in (0,t] \right] \nonumber \\
&=\mathbf{P}_0^{(\alpha)} \left[  \frac{X_1^{-(\alpha-1)}}{ \mathbf{P}_0^{(\alpha)} [ X_1^{-(\alpha-1)} ] } G(X_u : 0 \leq u \leq 1 ) \right],
\end{align}
where $\mathbf{P}_0^{(\alpha)}$ is the law of a Bessel process with dimension $2\alpha+1$.
\end{proposition}

In the proof of Proposition \ref{prop:larget} we will require the following fact.

\begin{lemma} \label{lem:momentcontrol}
For $\alpha \geq 1$, $x \geq 0$ we have $\mathbf{P}_x^{(\alpha)}[X_s^{-(\alpha-1)}] \leq s^{ - \frac{\alpha-1}{2}}$.
\end{lemma}
\begin{proof}
A (lengthy but straightforward) calculation using the transition density \eqref{eq:besseltran} of the Bessel process with dimension $2\alpha+1$ and making good use of the gamma integral tells us that for $\beta > - (2\alpha+1)$ we have 
\begin{align*}
\mathbf{P}_x^{(\alpha)}[X_s^\beta] = (2s)^{\beta/2} e^{-x^2/2s} \sum_{k \geq 0} \frac{ \Gamma(k+\alpha+1/2+\beta/2) }{ \Gamma(k+1) \Gamma(k+\alpha+1/2) } (x^2/2s)^k.
\end{align*}
Setting $\beta = - (\alpha-1)$ we obtain 
\begin{align} \label{eq:lisz}
    \mathbf{P}_x^{(\alpha)}[X_s^{-(\alpha-1)}] = (2s)^{-(\alpha-1)/2} e^{-x^2/2s} \sum_{k \geq 0} \frac{ \Gamma(k+\alpha/2+1) }{ \Gamma(k+1) \Gamma(k+\alpha+1/2) } (x^2/2s)^k.
\end{align}
Since $\alpha \geq 1$, by using the inequality $\Gamma(k+\alpha/2+1) \leq \Gamma(k+\alpha+1/2)$, we see that the sum in \eqref{eq:lisz} is bounded above by $e^{x^2/2s}$; after being generous with constants, from this bound we obtain the result.
\end{proof}
\begin{proof}[Proof of Proposition \ref{prop:larget}]
Suppose that $G(X_u : 0 \leq u \leq 1)$ is a functional on paths taking the form \begin{align} \label{eq:Gform}
G(X_u : 0 \leq u \leq 1 ) = \mathrm{1} \{ X_{u_1} > x_1, \ldots, X_{u_k} > x_k \}
\end{align}
for some $k \geq 0$, $x_i \geq 0$ and $0 < u_1 < \ldots < u_k \leq 1$. We include the possibility $k=0$, in which case $G(\cdot) \equiv 1$. Let us define
\begin{align} \label{eq:presandwich}
S_G(t) := \mathbf{Q}_0^{R_\lambda \wedge \lambda/2}\left[ G(t^{-1/2}X_{ut} : 0 \leq u \leq 1  ) \mathrm{1}_{\{ X_s > 0~ \forall s \in (0,t] \}} \right].
\end{align}
Note that since $(X_s)_{s \geq 0}$ diffuses according to a Bessel-$3$ process (i.e.\ a process with law $\mathbf{P}_0^{(1)}$) and is sent back to the origin at rate $(\lambda/2X_s^2) \wedge \lambda/2 = \lambda/(2(X_s \vee 1)^2)$, we have 
\begin{align} \label{eq:Funct}
S_G(t) = \mathbf{P}_0^{(1)} \left[ G_t \exp \left\{ - \int_0^t \frac{\lambda}{2(X_s \vee 1)^2} \mathrm{d}s \right\} \right],
\end{align}
where we are using the shorthand $G_t := G \left( t^{-1/2}X_{ut} : 0 \leq u \leq 1\right)$. 

For the purposes of truncation, consider the random variable
\begin{align*}
\mathcal{I}_r := \mathrm{1} \{ X_t \geq 1 \text{ for all $t \geq r$} \}.
\end{align*}
Since a Bessel-$3$ process is transient, the $\mathbf{P}_0^{(1)}$-expectation $I(r)$ of $\mathcal{I}_r$ satisfies $I(r) := \mathbf{P}_0^{(1)}[ \mathcal{I}_r ] \uparrow  1$ as $r \uparrow \infty$. Define
\begin{align} \label{eq:2Funct}
S_G(t,r) := \mathbf{P}_0^{(1)} \left[ G_t \exp \left\{ - \int_0^t \frac{\lambda}{2(X_s \vee 1)^2} \mathrm{d}s \right\} \mathcal{I}_r \right].
\end{align}
Considered as functionals of a continuous path $(X_s : 0 \leq s \leq t)$, each of 
\begin{equation*}
G_t, ~~ \exp \left\{ - \int_0^t \lambda/(2(X_s \vee 1)^2) \mathrm{d}s \right\}, \text{ and }  \mathcal{I}_r 
\end{equation*}
are increasing. By this we mean that if $X = (X_s : 0 \leq s \leq t)$ and $X' = (X_s' : 0 \leq s \leq t)$ are two continuous paths satisfying $X_s \leq X'_s$ for all $s$, then the functional applied to $X'$ will be at least as large as the functional applied to $X$. In particular, the random variables $G_t$, $\exp \left\{ - \int_0^t \frac{\lambda}{2(X_s \vee 1)^2} \mathrm{d}s \right\}$ and $ \mathcal{I}_r$ are nonnegatively correlated. It follows that
\begin{align} \label{eq:3Funct}
S_G(t,r) \geq \mathbf{P}_0^{(1)} \left[ G_t \exp \left\{ - \int_0^t \frac{\lambda}{2(X_s \vee 1)^2} \mathrm{d}s \right\} \right] \mathbf{P}_0^{(1)} \left[  \mathcal{I}_r  \right].
\end{align}
(The inequality \eqref{eq:2Funct} is a form of Fortuin–Kasteleyn–Ginibre (FKG) inequality for diffusions, see e.g.\ \cite{barbato} or \cite{legrand} for similar results.) Combining \eqref{eq:3Funct} with the definitions in \eqref{eq:Funct} and \eqref{eq:2Funct}, and using the fact that $\mathcal{I}_r \leq 1$, we obtain the simple sandwich inequality
\begin{align} \label{eq:sandwich}
S_G(t)  \geq S_G(r,t) \geq S_G(t)I(r).
\end{align}
Since $I(r) \uparrow 1$ as $r \uparrow \infty$, it follows that $S_G(r,t) \uparrow S_G(t)$ as $r \uparrow \infty$. 

Fix $r > 0$, and suppose that $t$ is sufficiently large so that $r < u_1 t$. Then by conditioning on $X_r$ and using the Markov property we can write
\begin{align} \label{eq:crsta}
S_G(r,t) = \mathbf{P}_0^{(1)} \left[ \exp \left\{ - \int_0^r \frac{\lambda}{2(X_s \vee 1)^2} \mathrm{d}s \right\} F_{r,t}(X_r) \right]
\end{align}
where
\begin{align*}
F_{r,t}(x) := \mathbf{P}_x^{(1)} \left[ G_{r,t} \exp \left\{ - \int_0^{t-r} \frac{\lambda}{2X_s^2} \mathrm{d}s \right\}  \mathrm{1} \{ X_s \geq 1 ~\forall s \geq 0 \} \right],
\end{align*}
and
\begin{align*}
G_{r,t} := \mathrm{1} \left\{ \frac{X_{u_1t -r }}{\sqrt{t}} > x_1,\ldots, \frac{X_{u_kt-r}}{\sqrt{t}} > x_k \right\}.
\end{align*}
Now applying \eqref{eq:urquell} to instead express $F_{r,t}(x)$ as a $\mathbf{P}_x^{(\alpha)}$ expectation, we can write
\begin{align} \label{eq:sekrit}
F_{r,t}(x) := x^{\alpha-1} \mathbf{P}_x^{(\alpha)} \left[ X_{t-r}^{-(\alpha-1)} G_{r,t} \mathrm{1} \{ X_s \geq 1 ~\forall s \geq 0 \} \right].
\end{align}
Note that using \eqref{eq:sekrit} to obtain the first inequality below, and then Lemma \ref{lem:momentcontrol} to obtain the second, for $t > r$ we have
\begin{align} \label{eq:sekrit2}
F_{r,t}(x) \leq x^{\alpha-1} \mathbf{P}_x^{(\alpha)}[X_{t-r}^{-(\alpha-1)}] \leq x^{\alpha-1}(t-r)^{-(\alpha-1)/2}.
\end{align}
We will use \eqref{eq:sekrit2} in a moment.

Consider now that for each fixed $r,x$, the random variables
\begin{align*}
\mathrm{1}\{ X_s \geq 1 ~\forall s \geq 0 \}  \quad \text{and} \quad \mathrm{1} \left\{ \frac{X_{u_1t -r }}{\sqrt{t}} > x_1,\ldots, \frac{X_{u_kt-r}}{\sqrt{t}} > x_k \right\} \frac{X_{t-r}^{-(\alpha-1)}}{\sqrt{t-r}},
\end{align*}
are asymptotically independent under $\mathbf{P}_x^{(\alpha)}$ as $t \to \infty$. It follows that as $t \to \infty$ we have
\begin{align} \label{eq:crsta20}
\lim_{t \to \infty} t^{(\alpha-1)/2} F_{r,t}(x) &= g(x) \tilde{C}_\lambda(G),
\end{align}
where
\begin{align*}
g(x) =  x^{\alpha-1} \mathbf{P}_x^{(\alpha)} \left( X_s \geq 1 ~ \forall s > 0 \right),
\end{align*}
and
\begin{align*}
    \tilde{C}_\lambda(G) = \mathbf{P}_0^{(\alpha)} \left[ X_1^{-(\alpha-1)} \mathrm{1} \{ X_{u_1} > x_1,\ldots, X_{u_k} > x_k \} \right].
\end{align*}

Note that $g(x) \leq x^{\alpha-1}$. We would like to apply the dominated convergence theorem and take the limit in \eqref{eq:crsta20} inside \eqref{eq:crsta}. Define the $\sigma( X_s : 0 \leq s \leq r)$-measurable random variable
\begin{align*}
Z^{(r)}_t = t^{(\alpha-1)/2} F_{r,t}(X_r).
\end{align*}
Then $Z^{(r)}_t$ converges almost-surely to $Z^{(r)}_\infty = \tilde{C}_\lambda(G) g(X_r)$ as $t \to \infty$.

From \eqref{eq:sekrit2} we see that if $t \geq 2r$ we have 
\begin{align} \label{eq:crsta2}
t^{(\alpha-1)/2} F_{r,t}(x) \leq C x^{\alpha-1}
\end{align}
for some constant $C$ depending on $\alpha$. In particular, $Z_t^{(r)} \leq C X_r^{\alpha-1}$ for all $t \geq 2r$. Since the latter random variable is $\mathbf{P}_0^{(1)}$-integrable, we can use the dominated convergence theorem to take the limit in \eqref{eq:crsta20} inside \eqref{eq:crsta} and obtain
\begin{align*}
\lim_{t \to \infty} t^{(\alpha-1)/2}S_G(r,t) = \tilde{C}_\lambda(G) K(r),
\end{align*} 
where
\begin{align*}
K(r) := \mathbf{P}_0^{(1)} \left[ g(X_r)\exp \left\{ - \int_0^r \frac{\lambda}{2(X_s \vee 1)^2} \mathrm{d}s \right\}\right].
\end{align*}
It follows from \eqref{eq:sandwich} and the fact that $I(r) \uparrow 1$ as $r \uparrow \infty$ that $K(r) \uparrow K(\infty) < \infty$ as $r \uparrow \infty$. In particular, we conclude again using \eqref{eq:sandwich} that
\begin{align} \label{eq:finality}
\lim_{t \to \infty} t^{(\alpha-1)/2} S_G(t) = K(\infty) \tilde{C}_\lambda(G) =: C_\lambda(G).
\end{align}
Letting $G \equiv 1$ in \eqref{eq:finality} and writing $C_\lambda := C_\lambda(1)$ we obtain \eqref{eq:Funct30}. 

To prove \eqref{eq:Funct40}, first note that it is sufficient to establish \eqref{eq:Funct40} only for functionals of the form in \eqref{eq:Gform}, since these generate the $\sigma$-algebra $\sigma(X_u : 0 \leq u \leq 1)$. Now note that for such functionals,  
\begin{align*} 
&\lim_{t \to \infty} \mathbf{Q}_0^{R_\lambda \wedge \lambda/2} \left[G(t^{-1/2}X_{ut} : 0 \leq u \leq 1  ) \Big|  X_s > 0 ~\forall s \in (0,t] \right] = \lim_{t \to \infty} \frac{S_G(t)}{ S_1(t)}\\
&= \frac{ C_\lambda(G)}{ C_\lambda(1) } = \frac{ \tilde{C}_\lambda(G)}{ \tilde{C}_\lambda(1)} = \frac{\mathbf{P}_0^{(\alpha)} \left[ X_1^{-(\alpha-1)} G(X_u : 0 \leq  u \leq 1 ) \right]}{ \mathbf{P}_0^{(\alpha)} [ X_1^{-(\alpha-1)} ] },
\end{align*}
completing the proof of \eqref{eq:Funct40}.
\end{proof}

Consider that if $(X_t)_{t \geq 0}$ is a Bessel-$3$ process starting from $x \geq 0$, then $(X'_t)_{t \geq 0} := (c^{-1}X_{c^2t})_{t \geq 0}$ has the law of a Bessel-$3$ process started from $x/c$. More generally given $R:(0,\infty) \to [0,\infty)$, we have the scaling relation
\begin{align} \label{eq:scala0}
(X_t)_{t \geq 0} \sim \mathbf{Q}_x^R \implies (X'_t)_{t \geq 0} := ( c^{-1}X_{c^2t} )_{t \geq 0} \sim \mathbf{Q}_{x/c}^{R^{(c)}}
\end{align}
where
\begin{align*}
R^{(c)}(g) = c^{2} R( cg ).
\end{align*}
In particular, setting $c = \sqrt{2n/\lambda}$ and $x = 0$ we see that
\begin{align} \label{eq:scala}
(X_t)_{t \geq 0} \sim \mathbf{Q}_0^{R_{\lambda} \wedge \lambda/2} \implies (X'_t)_{t \geq 0} := (\sqrt{\lambda/2n} X_{2nt/\lambda } )_{t \geq 0} \sim \mathbf{Q}_{0}^{R_\lambda \wedge n}.
\end{align}

We now use the scaling relation \eqref{eq:scala} to prove the following corollary of Proposition \ref{prop:larget}.

\begin{cor} \label{cor:larget}
With $\alpha = \alpha_\lambda$ as the solution to $\alpha(\alpha-1) = \lambda$, for fixed $t > 0$ we have 
\begin{align} \label{eq:Funct300}
\lim_{n \to \infty} n^{\frac{\alpha-1}{2} } \mathbf{Q}_0^{R_\lambda \wedge n} \left(  X_s > 0 ~\forall s \in (0,t] \right) = C_\lambda' t^{ - (\alpha-1)/2}
\end{align}
where, with $C_\lambda$ as in \eqref{eq:Funct30}, we have $C_\lambda' = (\lambda/2)^{(\alpha-1)/2}C_\lambda$. Moreover, \begin{align} \label{eq:Funct400}
\lim_{n \to \infty} \mathbf{Q}_0^{R_\lambda \wedge n} \left[G(X_s : 0 \leq s \leq t   ) |  X_s > 0 ~\forall s \in (0,t] \right] &=\mathbf{P}_0^{(\alpha)} \left[  \frac{X_t^{-(\alpha-1)}}{ \mathbf{P}_0^{(\alpha)} [ X_t^{-(\alpha-1)} ] } G(X_s : 0 \leq s  \leq t ) \right],
\end{align}
where $\mathbf{P}_0^{(\alpha)}$ is the law of a Bessel process with dimension $2\alpha+1$.
\end{cor}
\begin{proof}
We prove \eqref{eq:Funct300}; the proof of \eqref{eq:Funct400} is similar. Note by virtue of \eqref{eq:scala} that 
\begin{align} \label{eq:lake}
\mathbf{Q}_0^{R_\lambda \wedge n} \left(  X_s > 0 ~\forall s \in (0,t] \right) &= \mathbf{Q}_0^{R_\lambda \wedge \lambda/2}  \left(  X_s > 0 ~\forall s \in (0,2nt/\lambda] \right).
\end{align}
Now by \eqref{eq:Funct30} we have
\begin{align} \label{eq:lake2}
\lim_{n \to \infty} ( 2nt/\lambda)^{\frac{\alpha-1}{2}} \mathbf{Q}_0^{R_\lambda \wedge \lambda/2} \left( X_s > 0 ~\forall s \in \left(0,2nt/\lambda\right] \right) = C_\lambda.
\end{align}
Combining \eqref{eq:lake} and \eqref{eq:lake2} and rearranging, we obtain \eqref{eq:Funct30}. 
\end{proof}

\subsection{Proof of Theorem \ref{thm:newtransition}} \label{sec:singleproof}
    In this section we begin working towards our proof of Theorem \ref{thm:newtransition}. As mentioned above, the main task is proving Theorem \ref{thm:newtransition} for the special case where $R = R_\lambda$ for some $\lambda$. The key fact to note in the context of Corollary \ref{cor:larget} is that the probability that an excursion lasts for an $O(1)$ amount of time under $\mathbf{Q}_0^{R_\lambda \wedge n}$ behaves like $n^{-\beta}$, where
\begin{align*}
\beta_\lambda := (\alpha_\lambda-1)/2 = \frac{1}{4} \left( \sqrt{4\lambda+1}-1\right),
\end{align*}
where $\alpha_\lambda = \frac{1}{2} \left( 1 + \sqrt{4\lambda + 1} \right)$ as above. Note that as $\lambda$ increases from $0$ to $6$, $\alpha_\lambda$ increases from $1$ to $3$, and the exponent $\beta_\lambda$ increases from $0$ to $1$. In particular,
\begin{align*}
\lambda \in (0,6) \iff \beta_\lambda \in (0,1).
\end{align*}

For a large value of $n$, consider listing the consecutive excursions of the stochastic process $(X_t)_{t \geq 0}$ under $\mathbf{Q}_0^{R_\lambda \wedge n}$. Since the rate of jumps to zero is bounded by $n$, there are an almost surely finite number of excursions on any bounded time interval. Let $E_1,E_2,\ldots$ denote the durations of these excursions. The random variables $E_1,E_2,\ldots,$ are i.i.d., and, rewriting \eqref{eq:Funct300}, we see that their common distribution satisfies 
\begin{align} \label{eq:Econv}
\lim_{n \to \infty} n^{\beta } \mathbf{Q}_0^{R_\lambda \wedge n} \left(E_1 > t \right) = C_\lambda' t^{-\beta},
\end{align}
with \eqref{eq:Funct400} describing the conditional law of $(X_s)_{s \geq 0}$ on one such excursion. 

We would like to tackle the question of describing the asymptotic behaviour, as $n \to \infty$, of the excursion that straddles a fixed time $t > 0$. 

With this in mind, consider the asymptotic behaviour as $n \to \infty$ of the random walk
\begin{align} \label{eq:RW}
S_k = E_1 + \ldots + E_k
\end{align}
under $\mathbf{Q}_0^{R_\lambda \wedge n}$. Equivalently, $S_k$ is the $k^{\text{th}}$ return time of $(X_s)_{s \geq 0}$ to zero. Define
\begin{align} \label{eq:kdef}
k(t) := \inf \{ k \geq 1 : S_k \geq t \}
\end{align}
to be the index of the excursion straddling $t$. 

\subsubsection{Disappearance of excursions for $\lambda \geq 6$}
In this section we establish the following result:
\begin{proposition} \label{prop:zero}
Let $\lambda \geq 6$. Then for any $h, t > 0$ we have
\begin{align*}
\lim_{n \to \infty} \mathbf{Q}_0^{R_\lambda \wedge n}\left( \max_{1 \leq i \leq k(t) } E_i > h \right) = 0.
\end{align*}
\end{proposition}

\begin{proof}
By the Brownian scaling property, an equivalent formulation of the statement in the proposition is that for any $h, t > 0$ we have
\begin{align*}
\lim_{n \to \infty} \mathbf{Q}_0^{R_\lambda \wedge \lambda/2}  \left( \max_{1 \leq i \leq k(nt) } E_i > hn \right) = 0.
\end{align*}
However, it is a standard result (see e.g.\ \cite{FKZ}) in extreme value theory that if $E_1,E_2,\ldots$ are a sequence of independent and identically distributed random variables under a probability measure $ \mathbf{P}$ with tails satisfying $\lim_{y \uparrow \infty} y^\beta \mathbf{P}( E_i > y ) = c$ for some constant $c$ and some $\beta \geq 1$, then with $k(t)$ as in \eqref{eq:kdef} we have
\begin{align*}
\lim_{y \to \infty} \mathbf{P} \left( \max_{1 \leq i \leq k(y) } E_i > b y \right) = 0 \qquad \text{for every $b > 0$}.
\end{align*}
Since $\lambda \geq 6$ implies that $\beta_\lambda \geq 1$, the result follows from \eqref{eq:Econv}.
\end{proof}

\subsubsection{Emergence of excursions for $\lambda \in (0,6)$}
In this section we establish that when $\lambda < 6$, $\mathbf{Q}_0^{R_\lambda}$ governs a nontrivial stochastic process.

Our next result will relate the large-$n$ asymptotic behaviour of the random walk in \eqref{eq:RW} to a subordinator, which is a nondecreasing L\'evy process.
A subordinator starting from $0$ under a probability law $P$ satisfies $P[e^{-\theta \xi_r}] = e^{ - r \varphi(\theta)}$ for some Bernstein function $\varphi:[0,\infty) \to \mathbb{R}$ defined for $\theta \geq 0$. The function $\varphi$ is called the Laplace exponent of the process. For further information on L\'evy processes, see \cite{kyprianou}.

In the statement of the following result, recall that $\lambda \in (0,6) \iff \beta_\lambda \in (0,1)$. 
\begin{proposition} \label{prop:subor}
Let $\lambda \in (0,6)$ and let $\beta = \beta_\lambda := \frac{1}{2}( \sqrt{4\lambda+1} - 1)$. With $(S_k)_{k \geq 0}$ as in \eqref{eq:RW}, define a stochastic process $(\xi^{(n)}_r)_{r \geq 0}$ by 
\begin{align*}
\xi^{(n)}_r := S_{ \lfloor n^\beta r \rfloor},  \qquad r \geq 0.
\end{align*}
Then as $n \to \infty$, $(\xi^{(n)}_r)_{r \geq 0}$ under $\mathbf{Q}_0^{R_\lambda \wedge n}$ converges in distribution to a nondecreasing L\'evy process $(\xi_r)_{r \geq 0}$ whose Laplace exponent takes the form 
\begin{align*}
\varphi(\theta) = C_\lambda'' \theta^{\beta},
\end{align*} 
for some constant $C_\lambda''$.
\end{proposition}

Proposition \ref{prop:subor} is a variant on standard results on the convergence of random walks with heavy-tailed increments to subordinators. Before its proof, we need to extract the following hard bound from the proof of Proposition \ref{prop:larget}.

\begin{lemma} \label{lem:k1}
There is a constant $C > 0$ such that for all $t \geq 2$ we have $\mathbf{Q}_0^{R_\lambda \wedge \lambda/2} ( X_s > 0 ~ \forall s \in (0,t]) \leq C t^{ - \beta}$. 
\end{lemma}
\begin{proof}
In the notation of \eqref{eq:presandwich} we have $\mathbf{Q}_0^{R_\lambda \wedge \lambda/2} ( X_s > 0 ~ \forall s \in (0,t]) = S_1(t)$, where we are considering the case $G(\cdot) \equiv 1$. By setting $r=1$ in \eqref{eq:sandwich} we have
\begin{align} \label{eq:ui1}
S_1(t) \leq C S_1(1,t),
\end{align}
where $C := I(1)^{-1}$. Using \eqref{eq:crsta} to obtain the first inequality below, and then \eqref{eq:crsta2} to obtain the second, we have
\begin{align} \label{eq:ui2}
S_1(1,t) \leq \mathbf{P}_0^{(1)} \left[ F_{r,t}(X_r) \right] \leq C' t^{-\beta} \mathbf{P}_0^{(1)}[ X_1^{\alpha-1}] \leq C'' t^{-\beta}. 
\end{align}
Combining \eqref{eq:ui1} and \eqref{eq:ui2}, we obtain the result.

\end{proof}

\begin{proof}[Proof of Proposition \ref{prop:subor}]
For $\theta > 0$, $r \geq 0$, we have
\begin{align} \label{eq:olive1}
\mathbf{Q}_0^{R_\lambda \wedge n}[ e^{ - \theta \xi^{(n)}_r } ] = \mathbf{Q}_0^{R_\lambda \wedge n}[ e^{ - \theta E_1 } ]^{ \lfloor n^\beta r \rfloor}.
\end{align}
Integrating by parts to obtain the second equality below we have 
\begin{align} \label{eq:olive2}
\mathbf{Q}_0^{R_\lambda \wedge n}[e^{-\theta E_1}] &= - \int_0^\infty e^{ - \theta t} \frac{\mathrm{d}}{\mathrm{d}t}  \mathbf{Q}_0^{R_\lambda \wedge n} \left( E_1 > t \right) \mathrm{d}t \nonumber \\
&= - \int_0^\infty  \frac{\mathrm{d}}{\mathrm{d}t}  \left \{ e^{ - \theta t} \mathbf{Q}_0^{R_\lambda \wedge n} \left( E_1 > t \right) \right\}  +  \theta e^{ - \theta t} \mathbf{Q}_0^{R_\lambda \wedge n} \left( E_1 > t \right) \mathrm{d}t\nonumber \\
&= 1 - \theta n^{-\beta} \int_0^\infty f(t,n) e^{- \theta t} \mathrm{d}t,
\end{align}
where $f(t,n) :=  n^\beta \mathbf{Q}_0^{R_\lambda \wedge n} \left(E_1 > t \right)$. By \eqref{eq:Econv}, $\lim_{n \to \infty} f(t,n) = C_\lambda' t^{ - \beta}$.

We would like to argue that we can interchange the orders of limits and integration to obtain $\int_0^\infty f(t,n)e^{-\theta t} \mathrm{d}t \to \int_0^\infty C_\lambda' t^{-\beta} e^{ - \theta t} \mathrm{d}t$. In this direction, note using \eqref{eq:lake} to obtain the first equality below, and then Lemma \ref{lem:k1} to obtain the second, provided $2nt/\lambda \geq 2$ (i.e., $t \geq \lambda/n$) we have
\begin{align} \label{eq:orange}
f(t,n) = n^\beta \mathbf{Q}_0^{R_\lambda \wedge \lambda/2}  \left(  X_s > 0 ~\forall s \in (0,2nt/\lambda] \right) \leq n^\beta (2nt/\lambda)^\beta \leq C t^{-\beta}. 
\end{align}
For $t$ in the range such that the bound is \eqref{eq:orange} does not hold, i.e., for $t \in [0,\lambda/n]$, we can simply take $f(t,n) \leq n^\beta$, so that $\int_0^{\lambda/n} f(t,n)e^{-\theta t} \mathrm{d}t \leq \lambda n^{\beta-1} = o(1)$. It now follows from the bounded convergence theorem that 
\begin{align} \label{eq:olive3}
\int_0^\infty f(t,n)e^{-\theta t} \mathrm{d}t \to \int_0^\infty C_\lambda' t^{-\beta} e^{ - \theta t} \mathrm{d}t =: C_\lambda'' \theta^{-(1-\beta)},
\end{align}
where $C_\lambda'' = C_\lambda' \Gamma(1-\beta)$. Using \eqref{eq:olive2} and \eqref{eq:olive3} together in \eqref{eq:olive1} we obtain for every $\theta \geq 0$
\begin{align*} 
\lim_{n \to \infty} \mathbf{Q}_0^{R_\lambda \wedge n}[ e^{ - \theta \xi^{(n)}_r } ] = e^{ - r C_\lambda'' \theta^{\beta}}.
\end{align*}
It is easy to extend this calculation to multiple increments, thereby showing that $(\xi^{(n)}_r)_{r \geq 0}$ converges in distribution to the purported L\'evy process. 
\end{proof}

A brief calculation tells us that the Laplace exponent appearing in Proposition \ref{prop:subor} has L\'evy-Khintchine representation
\begin{align*}
\varphi(\theta) = C_\lambda'' \theta^{\beta} = C_\lambda''' \int_0^\infty (1-e^{-\theta x}) \mathrm{d}x / x^{1+\beta},
\end{align*} 
for some other constant $C_\lambda'''$. 

We are now ready to prove Theorem \ref{thm:newtransition}:

\begin{proof}[Proof of Theorem \ref{thm:newtransition}]
First we will prove Theorem \ref{thm:newtransition} in the special case where $R = R_\lambda$, and thereafter obtain the general case by means of a coupling argument.

Suppose first that $R = R_\lambda$ where $\lambda \geq 6$. We saw in Proposition \ref{prop:zero} that as $n \to \infty$, the size of the largest excursion up to time $t$ of this process away from zero converges to zero. It follows that $(X^{(n)}_t)_{t \geq 0}$ converges in distribution to zero.

Conversely, suppose now that $R = R_\lambda$ where  $\lambda < 6$. In this case, by Proposition \ref{prop:subor}, as $n \to \infty$ the lengths of the excursions converge in distribution to the jumps of a subordinator. It follows that $\mathbf{Q}_0^{R_\lambda}(X_t = 0)$ coincides with the probability that $t > 0$ lies in the range of the infinite-activity pure-jump subordinator $(\xi_r)_{r \geq 0}$, which is zero \cite[Theorem 5.9]{kyprianou}.

We turn to proving the general case. If $\liminf_{g \downarrow 0} g^2 R(g) > 6$, then choose $\varepsilon > 0$ such that $R(g) \geq 6/g^2 = R_6(g)$ for all $0 < g \leq \varepsilon$. 
Then using the coupling argument used in the proof of Lemma \ref{lem:beslim}, there is a coupling of stochastic processes $X^n := (X^n_t)_{t \geq 0}$ and $Y^n := (Y_t^n)_{t \geq 0}$ with respective laws $\mathbf{Q}_0^{R \wedge n}$ and $\mathbf{Q}_0^{R_\lambda \wedge n}$ such that $X_t^n \leq Y_t^n$ for all $t \leq \tau_\varepsilon$, where $\tau_\varepsilon$ is the first time $Y_t$ exceeds $\tau_\varepsilon$. Taking $n \to \infty$, it follows that $X_t^n$ and $Y_t^n$ converge in distribution to stochastic processes $(X_t)_{t \geq 0}$ and $(Y_t)_{t \geq 0}$ with respective laws $\mathbf{Q}_0^R$ and $\mathbf{Q}_0^{R_6}$. Moreover, we have $X_t \leq Y_t$ for all $t \leq \tau_\varepsilon$, where again $\tau_\varepsilon$ is the first time that $(Y_t)_{t \geq 0}$ hits $\varepsilon$. However, $Y_t$ is the zero process under $\mathbf{Q}_0^{R_6}$, and accordingly, it follows that $X_t$ must also be the zero process.

Finally, suppose that $\limsup_{g \downarrow 0} g^2R(g) < 6$. Choose $\varepsilon> 0$ such that $R(g) \leq \lambda/g^2$ for $g \in (0,\varepsilon]$ and for some $\lambda < 6$. Then by using a similar argument to the previous case, one can show that a stochastic process $(X_t)_{t \geq 0}$ with law $\mathbf{Q}_0^R$ may be coupled with a stochastic process $(Y_t)_{t \geq 0}$ with law $\mathbf{Q}_0^{R_\lambda}$ in such a way that $X_t \geq Y_t$ for all $t \in [0,\tau_\varepsilon]$. Moreover, for all $t > 0$ we have $\mathbf{Q}_0^{R_\lambda}(Y_t > 0) = 1$. It follows that $(X_t)_{t \geq 0}$ always escapes from zero under $\mathbf{Q}_0^R$.
\end{proof} 

\subsection{The $R$-Bessel process as a recurrent extension of a self-similar Markov process}
In the special case of \eqref{eq:scala0} where $R = R_\lambda$, we have the self-similarity
\begin{align} \label{eq:scala2}
(X_t)_{t \geq 0} \sim \mathbf{Q}_x^{R_\lambda} \implies (X'_t)_{t \geq 0} := ( c^{-1}X_{c^2t} )_{t \geq 0} \sim \mathbf{Q}_{x/c}^{R_\lambda}.
\end{align}
In other words, the probability laws $\{ \mathbf{Q}_x^{R_\lambda} : x \geq 0 \}$ govern a self-similar Markov process on $[0,\infty)$. 
In the language of self-similar Markov processes, the limit process $(X_t)_{t \geq 0}$ under $\mathbf{Q}^{R_\lambda}_x$ is a \textbf{recurrent extension} of the self-similar Markov process given by a Bessel-$3$ process killed at rate $\lambda/2X_s^2$. In this brief section, we outline how the probability laws $\{ \mathbf{Q}^{R_\lambda}_x : x \geq 0\}$ fit into the broader framework of recurrent extensions developed by Rivero and coauthors \cite{PPR,rivero1,rivero2}. We are particularly interested in identifying how the phase transition at $\lambda = 6$ manifests in this context.

Consider a Bessel-$3$ process $(X_t)_{t \geq 0}$. As a self-similar Markov process it can be written
\[
X_t = e^{\xi_{\phi(t)}} \quad \text{where} \quad \phi(t) = \int_0^t X_s^{-2} \, ds = \inf \left\{ s > 0 : \int_0^s e^{2\xi_u} \, du > t \right\}.
\]
It is well known that $\xi_t = B_t + \frac{1}{2} t$, where $B$ is a Brownian motion (see e.g.\ \cite[Exercise 13.10]{kyprianou}). Equivalently, the Laplace exponent of $\xi$ is given by
\[
\log \mathbb{E} \left[ e^{\theta \xi_1} \right] =: \psi(\theta) = \frac{1}{2} \theta^2 + \frac{1}{2} \theta.
\]

Now we introduce a killing time $\zeta$ such that
\[
\mathbb{P}(\zeta > t) = \exp\left( - \int_0^t \frac{\lambda}{2 X_s^2} \, ds \right) = \exp\left( -\frac{\lambda}{2} \phi(t) \right).
\]

The question as to whether a Bessel-3 process $(X_t)_{t \geq 0}$ killed at rate $\lambda/(2X_t^2)$ can enter from zero can now be formulated as follows: Can the self-similar Markov process whose underlying Lévy process via the Lamperti transform has Laplace exponent
\[
\phi(\theta) = \frac{1}{2} \theta^2 + \frac{1}{2} \theta - \frac{\lambda}{2}
\]
enter from zero? In other words, does the aforesaid self-similar Markov process have a recurrent extension? Rivero \cite{rivero2} gives a precise answer to this: there is a recurrent extension if and only if there exists a $0 < \theta < 2$ such that $\phi(\theta) = 0$. However, the roots of $\phi(\theta)$ are given by $\theta = - \frac{1}{2} \pm \frac{1}{2} \sqrt{1 + 4\lambda}$. In particular, there is a root $\theta_0$ lying in $(0,2)$ if and only if $0 < \lambda < 6$.

\subsection{Distributional properties under $\mathbf{Q}_0^{R_\lambda}$}

Bearing in mind that Theorem \ref{thm:newtransition} states that $X = (X_t)_{t \geq 0}$ under $\mathbf{Q}_0^{R_\lambda}$ is a nontrivial stochastic process if and only if $\lambda < 6$, the next result describes the marginal law of $X_t$ under $\mathbf{Q}_0^{R_\lambda}$ in this regime. 

\begin{proposition} \label{prop:sslaw}
Let $\lambda < 6$. Let $(X_t)_{t \geq 0}$ be a stochastic process distributed according to $\mathbf{Q}_0^{R_\lambda}$. 
Let 
\begin{align*}
\sigma_t := \sup \{ s \leq t : X_s = 0 \}
\end{align*}
denote the start time of the excursion straddling $t$, so that $t- \sigma_t$ is the length of the excursion so far. Then under $\mathbf{Q}_0^{R_\lambda}$:
\begin{enumerate}
\item The ratio $\sigma_t/t$ is Beta distributed with parameters $\beta_\lambda$ and $(1-\beta_\lambda)$. 
\item Given $\sigma_t$, we have the identity in law 
\begin{align} \label{eq:qual}
X_t \stackrel{d}{=} \sqrt{2 (t-\sigma_t) V},
\end{align}
where $V$ is a Gamma random variable with parameter $\alpha_\lambda/2+1$ independent of $\sigma_t$. 
\item The marginal law of $X_t$ is
\begin{align*}
X_t \stackrel{d}{=} \sqrt{2 t Z  },
\end{align*}
where $Z$ is a Gamma random variable with parameter $(6-\lambda)/4$. 
\end{enumerate}

\end{proposition}
\begin{proof} (1) By Proposition \ref{prop:subor} we have the equality in distribution
\begin{align} \label{eq:eqe}
\sigma_t \stackrel{d}{=} \xi_{\tau_t - },
\end{align}
where $(\xi_r)_{r \geq 0}$ is a L\'evy process with Laplace exponent $\varphi(\theta) = C_\lambda''\theta^{\beta_\lambda}$ and $\tau_t := \inf \{ r \geq 0 : \xi_r \geq t \}$ is the time $r$ at which this L\'evy process crosses a height $t$. According to the results of Section 5.5 of Kyprianou \cite{kyprianou}, we have
\begin{align*}
\mathbf{P} \left( \frac{ t - \xi_{\tau_t-}}{ t } \in \mathrm{d}u \right) =\frac{ \sin ( \pi \beta_\lambda)}{ \pi} u^{-\beta_\lambda} ( 1- u)^{-(1-\beta_\lambda)} \mathrm{d}u.
\end{align*}
In other words, the ratio $\frac{ t - \xi_{\tau_t-}}{ t }$ is Beta distributed with parameters $(1-\beta_\lambda)$ and $\beta_\lambda$. The first statement now follows from \eqref{eq:eqe}.

(2) Taking $x \downarrow 0$ in \eqref{eq:besseltran} we see that if $\mathbf{P}^{(\alpha)}_0$ is the law of a Bessel process of dimension $2\alpha+1$ started from zero, then 
\begin{align} \label{eq:besseltran0}
\mathbf{P}^{(\alpha)}_0 ( X_t \in \mathrm{d}y )  = \frac{1}{t\Gamma(\alpha+1/2)} y^{2 \alpha} \exp \left\{ - \frac{y^2}{2t} \right\} \mathrm{d}y.
\end{align}
Now, \eqref{eq:Funct400} says that under $\mathbf{Q}_0^{R_\lambda}$ the conditional law of $(X_t)_{t \geq 0}$ on the initial subinterval of length $r$ of an excursion of length at least $r$ is that of $\mathbf{P}^{(\alpha)}_0$, though with law at time $r$ size-biased by $y \mapsto y^{ - (\alpha-1)}$. It follows that we have
\begin{align} \label{eq:conq}
\mathbf{Q}_0^{R_\lambda}( X_t \in \mathrm{d}y | \sigma_t = s )  =C_{t-s} y^{ - (\alpha-1)} y^{2\alpha} \exp \left\{ - \frac{y^2}{2(t-s)} \right\} \mathrm{d}y,
\end{align}
for some constant $C_{t-s}$ depending on $t-s$ but not $y$. The equation \eqref{eq:conq} amounts to \eqref{eq:qual}. 

(3) We can write $(t-\sigma_t)\stackrel{d}{=} tW$, where $W$ is Beta distributed with parameters $(1-\beta_\lambda)$ and $\beta_\lambda$. Recall that if $A$ is a Gamma random variable with parameter $\mu$ and independent of $W$, then the product $WA$ is Gamma distributed with parameter $(1-\beta_\lambda)\mu$. In particular, it follows from \eqref{eq:qual} that $(t-\sigma_t)Y = tZ$ where $Z = WY$ is Gamma distributed with parameter $(1-\beta_\lambda)(\alpha_\lambda/2+1)$. Using the definitions $\beta_\lambda = (\alpha_\lambda-1)/2$ and $\alpha_\lambda(\alpha_\lambda-1)=\lambda$, it follows that $(1-\beta_\lambda)(\alpha_\lambda/2+1) = (6-\lambda)/4$, completing the proof.
\end{proof}

We note that it is possible using part (3) of the previous result to show that if 
$(X_t^{(\lambda)})_{t \geq 0}$ has law $\mathbf{Q}_0^{R_\lambda}$, then $X^{(\lambda)}$ converges to the zero process as $\lambda \uparrow 6$. This follows from the fact that $X_t \stackrel{d}{=} \sqrt{2 Z_\lambda t}$ where $Z_\lambda$ is Gamma distributed with parameter $(6-\lambda)/4$: as $\lambda \uparrow 6$, $Z_\lambda$ converges in distribution to zero.

\section{The $R$-vein} \label{sec:vein}

In the previous section we introduced the $R$-Bessel process, a stochastic process that diffuses according to a Bessel-$3$ process, though at rate $R(\sqrt{2}X_t)$ undergoes jumps back to the origin. The motivation behind defining this process will become fully apparent in the present section, where we introduce a process called the $R$-vein. Where the $R$-Bessel process describes the height process of an emergent bubble along a dual web path, the $R$-vein will describe its \emph{shape}.

\subsection{The $R$-vein}

\begin{definition}[The $R$-vein] \label{df:Rvein}
Let $R:(0,\infty) \to [0,\infty)$ be a bounded or decreasing function. The \textbf{$R$-vein} is a triple of 
stochastic processes $(L_t,C_t,U_t)_{ t \geq 0}$ starting from $L_0 = \ell \leq C_0 = c \leq U_0 = u$, satisfying $L_t \leq C_t \leq U_t$ for all $t \geq 0$, and constructed as follows:
\begin{itemize}
\item $(C_t)_{t \geq 0}$ is a Brownian motion
\item Conditional on $(C_t)_{t \geq 0}$, $(L_t)_{t \geq 0}$ and $(U_t)_{t \geq 0}$ diffuse according to Brownian motions reflected off the path of $(C_t)_{t \geq 0}$ from below and above respectively. 
\item At instantaneous rate $R(U_{t-}-L_{t-})$, the processes $L$ and $U$ both jump to the position of $C$, with $U$ (resp.\ $L$) reinitiated `just above' (resp.\ `just below') the path of $C$.
\end{itemize}
We write $\mathbf{Q}_{\ell,c,u}^R$ for the law of the $R$-vein starting from $(\ell,c,u)$. We refer to the probability law $\mathbf{Q}^R_{0,0,0}$ as governing the \textbf{standard $R$-vein}.  
\end{definition}
See Figure \ref{fig:timvein} for a simulation of an $R$-vein with $R_3(\lambda) = 3/\lambda^2$ initiated from $\ell = c = u = x \in \mathbb{R}$.

For a precise definition of \emph{reflection} (sometimes called Skorokhod reflection) of a Brownian path off another Brownian path in terms of local times, see e.g.\ \cite{warren, SSS}. 
\begin{figure}[h!]
\centering
\includegraphics[width=0.9\textwidth]{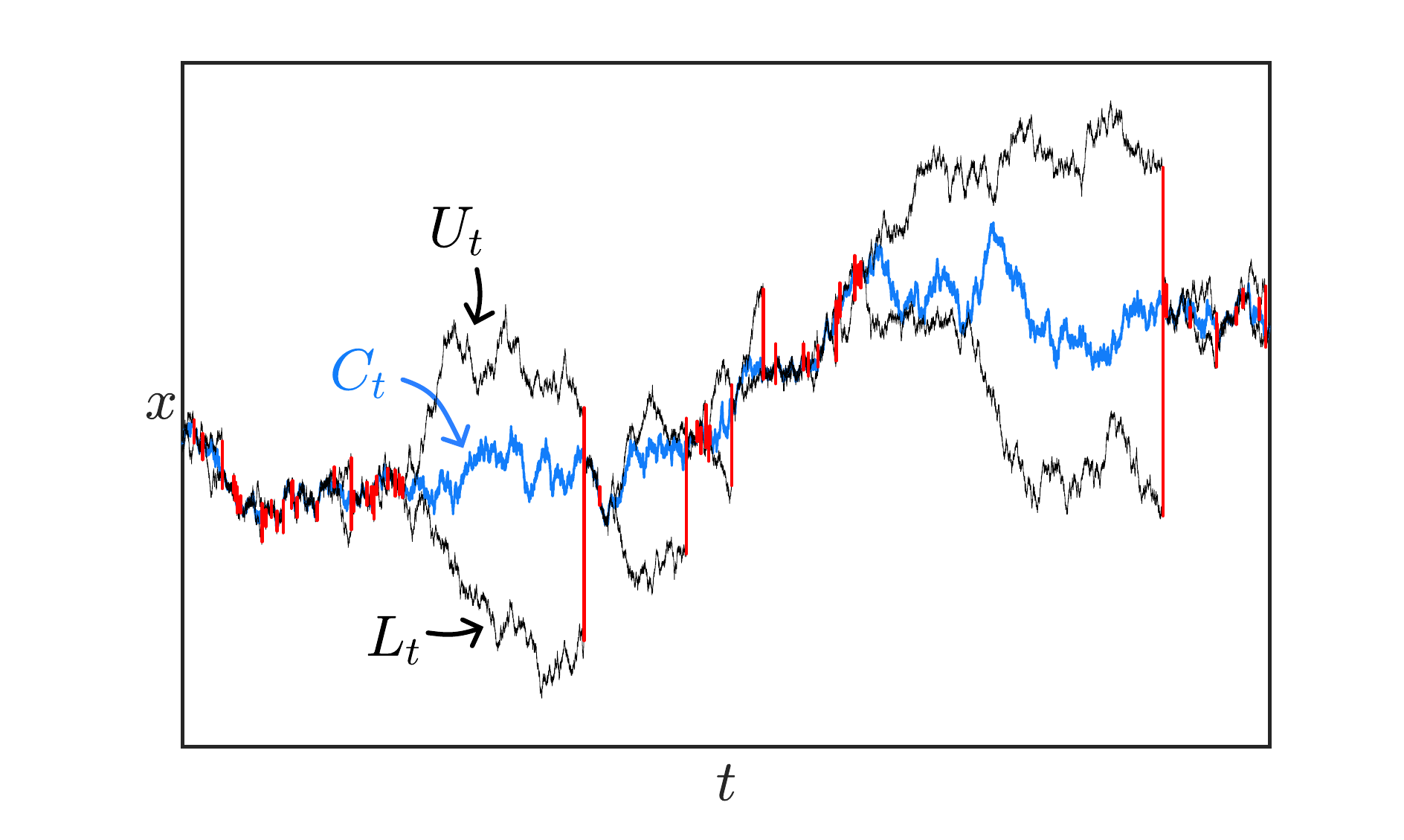}
\caption{The $R$-vein is a stochastic process consisting of lower and upper Brownian motions that reflect off a central Brownian motion. At instantaneous rate $R(U_t-L_t)$, the upper and lower processes jump to the position of the central Brownian motion.}
\label{fig:timvein}
\end{figure}

Again, the $R$-vein is well-defined for bounded $R$. For unbounded but decreasing $R$, it is possible to construct the $R$-vein using a coupling argument as was done for the $R$-Bessel process in the previous section. The key idea is that if $R'$ and $R$ are decreasing and bounded with $R'(g) \geq R(g)$ there is a coupling between veins $(L_t,C_t,U_t)$ and $(L_t',C_t',U_t')$ with respective laws $\mathbf{Q}^R_{\ell,c,u}$ and $\mathbf{Q}^{R'}_{\ell,c,u}$ 
such that 
\begin{align*}
C_t' = C_t  \quad \text{and} \quad L_t \leq L_t' \leq U_t' \leq U_t
\end{align*}
for all $t \geq 0$. Using this coupling one can construct the law $\mathbf{Q}_{\ell,c,u}^R$ for a nonincreasing but possibly unbounded $R$ by truncation. 

\subsection{Probabilistic properties of the $R$-vein}

In the terminology of Warren \cite{warren}, the motion $(C_t)_{t \geq 0}$ behaves like a \emph{heavy} particle, and the motions $(L_t)_{t \geq 0}$ and $(U_t)_{t \geq 0}$ behave like \emph{light particles}. 

It is by no means obvious from the definition, but it turns out that the difference process $(U_t-L_t)_{t \geq 0}$ in the $R$-vein is intimately related to the $R$-Bessel process. Using an identity of Warren \cite{warren} (see also earlier work by Dubédat \cite{dubedat} and Soucaliuc, T\'oth and Werner \cite{STW00}) we are able to prove the following result:

\begin{lemma} \label{lem:law}
Let $(L_t,C_t,U_t)_{t \geq 0}$ be an $R$-vein such that each coordinate starts from the same location $\ell=c=u$. Then the processes $(L_t,U_t)_{t \geq 0}$ have the following marginal description:
\begin{itemize}
\item $(L_t)_{t \geq 0}$ and $(U_t)_{t \geq 0}$ diffuse according to Brownian motions conditioned never to collide with one another.
\item At instantaneous rate $R(U_{t}-L_{t})$, $L_t$ and $U_t$ jump to a uniformly chosen point in the interval $[L_t,U_t]$, with $U_t$ reinitiated `just above' $L_t$. 
\item The marginal law of $(X_t := (U_t-L_t)/\sqrt{2})_{t \geq 0}$ is precisely that of $\mathbf{Q}_0^R$, i.e.\ the $R$-Bessel process starting from $0$.
\end{itemize}
Moreover, we can recover the time-$t$ law of $C_t$ from $(L_t,U_t)$ via:
\begin{itemize}
\item If the processes start from $\ell = c= u$, then conditional on $(L_t,U_t)$, $C_t$ is uniformly distributed on the interval $[L_t,U_t]$. 
\end{itemize}
\end{lemma}
We now outline how Lemma \ref{lem:law} follows from known results in the literature.
\begin{proof}
According to the special case $n=1$ of \cite[Proposition 5]{warren}, if $R = 0$, then $(L_t,U_t)_{t \geq 0}$ behave like Brownian motions conditioned to never collide. Equivalently, their difference behaves like $\sqrt{2}$ times a Brownian motion conditioned to stay positive, or like $\sqrt{2}$ times a Bessel-$3$ process. Moreover, it is a consequence of \cite[Lemma 4]{warren} that the marginal distribution of $C_t$ given $[L_t,U_t]$ is uniform. 

At instantaneous rate $R(U_t-L_t)$, the processes $U_t$ and $L_t$ both jump to the position of $C_t$, which, conditional on $(L_t,U_t)$, is uniform on $[L_t,U_t]$. That establishes the stipulated dynamics.
\end{proof}

\subsection{The phase transition for the $R$-vein}
We now state and prove the following result, which follows quickly from Theorem \ref{thm:newtransition} and Lemma \ref{lem:law}. 

\begin{theorem} \label{thm:newtransition2}
Let $R:(0,\infty) \to [0,\infty)$ be a quadratic-regular rate function (see \eqref{eq:squareregular}), and let $\mathbf{Q}^{R}_{0,0,0}$ be the law of the standard $R$-vein.
Then we have the phase transition: 
\begin{itemize}
\item If $\liminf_{g \downarrow 0} g^2R(g) \geq 6$,  then $\mathbf{Q}_{0,0,0}^R ( L_t = C_t = U_t   ~ \forall t \geq 0 ) = 1$.
\item If $\limsup_{g \downarrow 0} g^2R(g)  < 6$, then  for all $t > 0$ we have $\mathbf{Q}_{0,0,0}^R( L_t < C_t < U_t ) = 1$.
\end{itemize}

\end{theorem}

\begin{proof}
According to Lemma \ref{lem:law}, the marginal law of $X = ((U_t-L_t)/\sqrt{2})_{t \geq 0}$ is that of $\mathbf{Q}_{0,0,0}^R$.

If $\liminf_{g \downarrow 0 } g^2 R(g) \geq 6$, then by Theorem \ref{thm:newtransition}, $X$ is the zero process. Since $L_t \leq C_t \leq U_t$, it must follow that $L_t = C_t = U_t$ for all $t \geq 0$ almost surely.

Conversely, if $\limsup_{ g \downarrow 0} g^2 R(g) < 6$, then by Theorem \ref{thm:newtransition}, for any fixed $t > 0$, $X_t$ is almost surely strictly positive. Accordingly, $L_t < U_t$ almost surely. Now according to the last statement in Lemma \ref{lem:law}, $C_t$ is uniformly distributed on $[L_t,U_t]$, so that in particular, $L_t < C_t < U_t$ almost surely. 
\end{proof}

\section{Extrinsic construction of the $R$-marble} \label{sec:explore}

In this section we introduce our key proof idea: the $R$-marble (for bounded $R$) can be explored using a sequence of dual web paths initiated from a dense subset of $(0,\infty) \times \mathbb{R}$.

We will ultimately reverse this procedure in Section \ref{sec:mainproof} to \emph{construct} the $R$-marble simultaneously for different values of $R$, showing that this construction is stable in the limit $n \to \infty$ for $R_n(g) = R(g) \wedge n$.

A large fraction of Section \ref{sec:spaces}, \ref{sec:dual} and Section \ref{sec:exceptional} are lifted from \cite{SSS}.

\subsection{Spaces of collections of paths} \label{sec:spaces}

In Section \ref{sec:technical} we mentioned that coalescing Brownian motion, the Brownian web and the $R$-marble could all be considered as $\mathcal{H}$-valued random variables, where $\mathcal{H}$ is the space of collections of continuous paths in $[0,\infty) \times \mathbb{R}$ endowed with a suitable metric. In the present section we supply the technical details necessary to make precise sense of this construction. 

First, we consider the completion $(\overline{[0,\infty) \times \mathbb{R}}, \rho)$ of $[0,\infty) \times \mathbb{R}$ with respect to the metric
\begin{equation}\label{rho}
\rho\big((x_1, t_1), (x_2,t_2)\big) = \left|\tanh(t_1)-\tanh(t_2)\right|
\ \vee\ \left|\frac{\tanh(x_1)}{1+|t_1|}-\frac{\tanh(x_2)}{1+|t_2|}\right|.
\end{equation} 
Next, a \textbf{path} $\pi$ in $\overline{[0,\infty) \times \mathbb{R}}$, whose starting time we denote by
$\sigma_\pi\in [0,\infty]$, is a mapping $\pi :
[\sigma_\pi,\infty] \to [-\infty, \infty]$ such that
$t
\to (t,\pi(t))$ is a continuous map from $[\sigma_\pi,\infty]$ to $(\overline{[0,\infty) \times \mathbb{R}}, \rho)$. We then define $\Pi$ to be the space of all
paths in $\overline{[0,\infty) \times \mathbb{R}}$ with all possible starting times in $[-\infty,\infty]$.
If we endow $\Pi$ with the metric $\mathrm{d}_{\mathrm{path}}:\Pi \times \Pi \to [0,\infty)$ given by 
\begin{equation}\label{PId}
\!
\mathrm{d}_{\mathrm{path}}(\pi_1, \pi_2) = \Big|\!\tanh(\sigma_{\pi_1})-\tanh(\sigma_{\pi_2})\Big|
\ \vee \sup_{t\geq \sigma_{\pi_1} \wedge \sigma_{\pi_2}}\!
\left|\frac{\tanh(\pi_1(t\vee \sigma_{\pi_1}))}{1+|t|}
-\frac{\tanh(\pi_2(t\vee \sigma_{\pi_2}))}{1+|t|}\right|, \!\!\!
\end{equation}
then $(\Pi,\mathrm{d}_{\mathrm{path}})$ is a complete and separable metric space. Note that convergence in the metric $\mathrm{d}_{\mathrm{path}}$ can be described as local uniform convergence
of paths plus convergence of starting times. 

Let ${\mathcal H}$ denote the {\em space of compact subsets of $(\Pi, \mathrm{d}_{\mathrm{path}})$},
equipped with the Hausdorff metric
\begin{equation}\label{dH}
\mathrm{d}_{{\mathcal H}}(K_1, K_2) = \sup_{\pi_1\in K_1} \inf_{\pi_2\in K_2}\!\!
\mathrm{d}_{\mathrm{path}}(\pi_1, \pi_2)\ \vee  \sup_{\pi_2\in K_2}
\inf_{\pi_1\in K_1} \mathrm{d}_{\mathrm{path}}(\pi_1, \pi_2),
\end{equation}
and let ${\mathcal B}_{\mathcal H}$ be the Borel $\sigma$-algebra associated with $\mathrm{d}_{\mathcal H}$.

The following result is the standard characterization of the Brownian web, adapted for $[0,\infty) \times \mathbb{R}$ in place of $\mathbb{R}^2$. 

\begin{theorem}[Characterization of the Brownian web]\label{T:webchar}
There exists an $({\mathcal H}, {\mathcal B}_{\mathcal H})$-valued random variable ${\mathcal W}$, called
the standard Brownian web, whose distribution is uniquely determined
by the following properties:
\begin{itemize}
\item[{\rm (a)}] For each deterministic $z \in [0,\infty) \times \mathbb{R}$, almost surely
there is a unique path $\pi^z \in {\mathcal W}$.

\item[{\rm (b)}] For any finite deterministic set of points $z_1, \ldots, z_k
  \in[0,\infty) \times \mathbb{R}$, the collection $(\pi^{z_1}, \ldots, \pi^{z_k})$
  is distributed as coalescing Brownian motions.

\item[{\rm (c)}] For any deterministic countable dense subset
${\mathcal D} \subset (0,\infty) \times \mathbb{R}$, almost surely, ${\mathcal W}$ is the closure of
$\{\pi^z : z\in{\mathcal D}\}$ in $(\Pi, \mathrm{d}_{\mathcal{H}})$.
\end{itemize}
\end{theorem}

\subsection{The dual web} \label{sec:dual}
From the Brownian web $\mathcal{W}$ on $[0,\infty) \times \mathbb{R}$ we can construct the \textbf{dual web} $\hat{\mathcal{W}}$ on $[0,\infty) \times \mathbb{R}$. We will be content to outline the key ideas here, deferring to \cite{SSS} for the precise technical details. 

The key idea is that it is possible to show that for each $z = (t,x) \in (0,\infty) \times \mathbb{R}$ there is almost surely a unique path $\hat{\pi}^z$ travelling backwards in time until it hits the $\{0\} \times \mathbb{R}$ axis, that does not cross any path $\pi$ of the original Brownian web $\mathcal{W}$. This path is a set of points
\begin{align*}
\hat{\pi}^z = \{ (s,C_s^z) : s \in [0,t] \}
\end{align*}
satisfying $C^z_t = x$. We call the process $(C^z_s)_{s \in [0,t]}$ the forwards in time indexing of $\hat{\pi}^z$. The dual web path $\hat{\pi}^z$ emitted from $z$ collides with the $t=0$ axis at a random location $(0,C^z_0)$. This procedure creates a second set of paths 
\begin{align*}
\hat{\mathcal{W}} := \{ \hat{\pi}^z : z \in (0,\infty) \times \mathbb{R} \}
\end{align*}
which we call the \textbf{dual (Brownian) web} on $[0,\infty) \times \mathbb{R}$. We call each backward path $\hat{\pi}^z$ a dual web path. See \cite[Theorem 2.4]{SSS} for further details.

Conditional on a forwards in time indexing $(C^z_s)_{s \in [0,t]}$ of a dual web path to $z = (t,x)$, a path $\pi$ of the Brownian web diffuses according to a Brownian motion reflected off this path. More generally, we have the following  remark:

\begin{rem} \label{rem:partial2}
Let $\{z,w, z_1,\ldots,z_k, w_1,\ldots,w_j \}$ be distinct elements of $[0,\infty) \times \mathbb{R}$. Then conditional on the paths $(\pi^{z_1},\ldots,\pi^{z_k})$ of $\mathcal{W}$ emitted from $z_1,\ldots,z_k$ together with the dual web paths $(\hat{\pi}^{w_1},\ldots,\hat{\pi}^{w_j})$ emitted from $w_1,\ldots,w_j$:
\begin{enumerate}
    \item The conditional law of the web path $\pi^z$ emitted from $z$ coalesces with all of the paths $\pi^{z_i}$ and reflects off all of the paths $\hat{\pi}^{w_j}$. 
    \item The conditional law of the dual web path $\hat{\pi}^w$ emitted from $w$ reflects off all of the paths $\pi^{z_i}$ and coalesces with all of the paths $\hat{\pi}^{w_j}$. 
\end{enumerate}
\end{rem}

See \cite{SSS} for further details on Skorokhod reflection of web paths with the dual web. 
\subsection{Exceptional points in the Brownian web} \label{sec:exceptional}

For each $z \in [0,\infty) \times \mathbb{R}$, there is almost surely a unique path $\pi^z$ in the Brownian web $\mathcal{W}$ emitted from $z$.

Likewise, for each $z \in (0,\infty) \times \mathbb{R}$, there is almost surely a unique dual web path $\hat{\pi}^z$ initiated backwards in time from $z$. There do, however, exist exceptional points from which more than one path may be emitted forwards or backwards in time.
We give a brief overview of the classification of these points here which is lifted from Section 6.2.5 of \cite{SSS} (which the reader might consult for further information and precision).

We say that a path $\pi$ in $\mathcal{W}$ \textbf{enters} $z = (t,x)$ if $\sigma_\pi < t$ and $\pi_t = x$, and \textbf{leaves} $z $ if $\sigma_\pi \leq t$ and $\pi_t =x$. Two paths $\pi$ and $\pi'$ leaving $z$ are said to be equivalent, denoted $\pi' \sim^z_{\mathrm{out}} \pi$, if $\pi' = \pi$ on $[t,\infty)$. Two paths $\pi$ and $\pi'$ entering $z$ are said to be equivalent, denoted $\pi' \sim^z_{\mathrm{in}} \pi$, if $\pi' = \pi$ on $[t-\varepsilon,\infty)$ for some $\varepsilon > 0$. We write $m_{\mathrm{in}}(z)$ and $m_{\mathrm{out}}(z)$ for the respective number of equivalence classes of paths entering and leaving $z$. Given a realisation of the web $\mathcal{W}$ on $[0,\infty) \times \mathbb{R}$, we can classify points $z$ of $[0,\infty) \times \mathbb{R}$ according to the value of $(m_{\mathrm{in}}(z),m_{\mathrm{out}}(z))$. We can, of course, also perform an analogous construction for the dual $\hat{\mathcal{W}}$ of the Brownian web $\mathcal{W}$, to obtain a pair $(\hat{m}_{\mathrm{in}}(z),\hat{m}_{\mathrm{out}}(z))$. We refer to $z$ as a $(m_{\mathrm{in}}(z),m_{\mathrm{out}}(z))/(\hat{m}_{\mathrm{in}}(z),\hat{m}_{\mathrm{out}}(z))$ point. 

For each fixed $z$, $z$ is almost surely a $(0,1)/(0,1)$ point. However, there do exist almost surely other types of points which can be shown to necessarily satisfy the relation 
\[
m_{\mathrm{out}}(z) = \widehat{m}_{\mathrm{in}}(z) + 1 \quad \text{and} \quad \widehat{m}_{\mathrm{out}}(z) = m_{\mathrm{in}}(z) + 1
\]
almost surely for all $z \in [0,\infty) \times \mathbb{R}$;
see \cite[Theorem 6.2.11]{SSS}. More specifically, there exist almost surely uncountably many $(1,1)/(0,2)$ points $(0,2)/(1,1)$ points, and there even exist almost surely countably many $(2,1)/(0,3)$ and $(0,3)/(2,1)$ points. For any $(i,j)$ with $i+j \geq 4$ however, there almost surely do not exist any $(i,j)$ points. 

Let us highlight in particular $(0,2)/(1,1)$ points, which we will call $(0,2)$ points for short, and which will play a prominent role in the following section. These points are precisely points that lie in the interior of the path of some dual web element and are not points where two dual web paths coalesce. From each such point two forward paths are initiated, informally these paths are initiated `just above' and `just below' the path of the dual web. By Remark \ref{rem:partial2}, these paths are reflected off the dual web path. We refer to these respective paths as the `upper' and `lower' paths emitted from the $(0,2)$ point.

\subsection{The bubble containing $z$} \label{sec:bubble}

\begin{figure}[h!]
\centering
\includegraphics[width=0.9\textwidth]{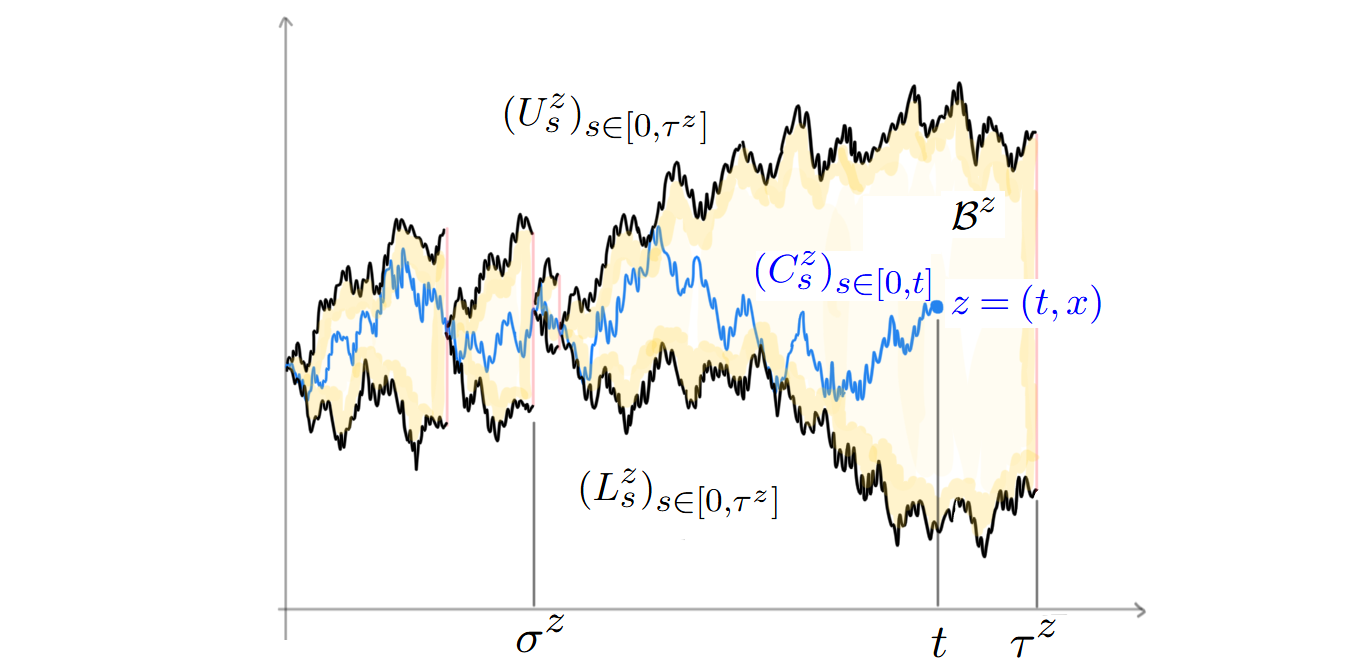}
\caption{The dual web path leading to a point $z$ gives rise to a stochastic process which we call the vein to $z$.}
\label{fig:timveinz}
\end{figure}
For bounded $R:(0,\infty) \to [0,\infty)$, we now explore the $R$-marble using the dual web. Recall from Section \ref{sec:marbleconstruction} our construction of the $R$-marble $\mathcal{M}(R)$ for a bounded rate function $R$. By construction, for any point $z = (t,x) \in (0,\infty) \times \mathbb{R}$, $z$ is almost surely contained in a bubble of the form
\begin{align} \label{eq:bubble}
\mathcal{B}^z := \{ (s,y) : \sigma_z < s < \tau_z, L^z_s < y < U^z_s \}. 
\end{align}

There is a convenient way of describing the probabilistic structure of the bubble containing $z = (t,x)$ using the dual web. Let $\hat{\pi}^z := (s,C^z_s)_{s \in [0,t]}$ be the dual web path emitted from $z$. We now consider not just the bubble containing $z$, but also all historic bubbles that have contained some segment of the path $\hat{\pi}^z$; see Figure \ref{fig:timveinz}. We can extend the definitions of $L^z_s$ and $U^z_s$ --- thus far only defined for $s \in [\sigma_z,\tau_z]$ --- to the interval $[0,t]$ as well by letting, for each $s \in [0,t]$, $[L^z_s,U^z_s]$ be the time-$s$ space interval containing the point $(s,C^z_s)$. In other words, 
\begin{align} \label{eq:L}
L^z_s := \sup \{ y \leq C^z_s : (s,y) \in \mathrm{Tr}(\mathcal{M}(R)) \}
\end{align}
and
\begin{align} \label{eq:U}
U^z_s := \inf \{ y \geq C^z_s : (s,y) \in \mathrm{Tr}(\mathcal{M}(R)) \},
\end{align}
where we recall from \eqref{eq:image} that for an $\mathcal{H}$-valued random variable $\mathcal{M}$, $\mathrm{Tr}(\mathcal{M})$ is the closure of the subset of points of $[0,\infty) \times \mathbb{R}$ lying in some path of $\mathcal{M}$. 
This procedure generates, for any point $z = (t,x)$ of $(0,\infty) \times \mathbb{R}$, a pair of stochastic processes 
\begin{align*}
V^z := (L^z_s,C^z_s,U^z_s)_{s \in [0,t]} \quad \text{and} \quad \tilde{V}^z := (L^z_s,U^z_s)_{s \in [t,\tau_z]}.
\end{align*}
We call $V^z$ the \textbf{vein to $z$} and refer to $\tilde{V}^z$ as its \textbf{continuation}.

We emphasise that $C^z_s$ is not defined for $s$ in $(t,\tau_z]$. 

We now relate the vein $V^z$ with the $R$-vein construction we saw in Section \ref{sec:vein}. We recall from Definition \ref{df:Rvein} that the $R$-vein is a triple $V 
 =(L,C,U) := (L_s,C_s,U_s)_{s \geq 0}$ where $C$ is a Brownian motion, and conditional on $C$, $L$ and $U$ fluctuate according to Brownian motions reflected off the path of $C$. At instantaneous rate $R(U_t-L_t)$, both $L$ and $U$ jump to the position of $C$, with $L$ reinitiated `just below' and $U$ reinitiated `just above' the location of the path of $C$ at the time of the jump. We say that a vein $(L_s,C_s,U_s)_{s \geq 0}$ is \textbf{standard} if it starts from $L_0 = C_0 = U_0 = 0$. 
The following remark relates the probabilistic behaviour of the vein $V^z$ to $z$ with the $R$-vein described in Section \ref{sec:bessel}:

\begin{lemma} \label{lem:veiny}
For any $z = (t,x)$ in $(0,\infty) \times \mathbb{R}$:
\begin{enumerate}
\item The recentered triple 
\begin{align*}
({L}^z_s-{C}^z_0, {C}^z_s - {C}^z_0 , {U}^z_s - {C}^z_0)_{s \in [0,t]} \stackrel{d}{=} (L'_s,C'_s,U'_s)_{s \in [0,t]}
\end{align*}
has the law of the standard $R$-vein run until time $t$. 
\item Conditional on $V^z$, the law of $\tilde{V}^z = (L^z_s,U^z_s)_{s \in [t,\tau_z]}$ is given as follows: $(L^z_s,U^z_s)_{s \in [t,\tau_z]}$ 
fluctuate according to independent coalescing Brownian motions initiated at time $t$ from $(L^z_t,U^z_t)$, and $\tau_z$ has the law of the minimum of the first collision time of $L^z$ and $U^z$ and a stochastic clock with instantaneous rate $R(U^z_t-L^z_t)$. 
\end{enumerate}
\end{lemma}
\begin{proof}
(1) Note that since a Brownian motion minus its final position run backwards in time is again a Brownian motion, it follows that $(C^z_s-C^z_0)_{s \in [0,t]}$ is a standard Brownian motion.  Moreover, paths of the Brownian web are reflective off paths of the dual web. It follows in particular that given the path $({C}^z_s)_{s \in [0,t]}$, the paths ${L}^z$ and ${U}^z$ diffuse forwards in time according to Brownian motions that reflect off the path of ${C}^z$. Moreover, since these paths describe the boundary of the bubble containing the point $(s,{C}^z_s)$, at instantaneous rate $R({U}^z_s-{L}^z_s)$ they both jump to the position ${C}^z_s$. Thus after adjusting the starting position, this is precisely the law of the $R$-vein.

(2) This follows immediately from the definition of the $R$-marble. 
\end{proof}

We note we have the following corollary:
\begin{cor}
The law of the height of the bubble containing a point $z = (t,x)$ in the $R$-marble is $\sqrt{2}$ times $X_t$, where $(X_t)_{t \geq 0}$ has law $\mathbf{Q}_0^R$.
\end{cor}
\begin{proof}
Note that the height at time $t$ of the bubble containing $z$ is given by ${U}^z_t - {L}^z_t$, which, by the previous lemma, has the same law as $U_t - L_t$ under $\mathbf{Q}_{0,0,0}^R$. However, by Lemma \ref{lem:law}, $U_t-L_t$ has the law of $\sqrt{2}$ times $X_t$ where $(X_s)_{s \geq 0}$ has law $\mathbf{Q}_0^R$.
\end{proof}

We also note as a consequence of Lemma \ref{lem:law} that $C^z_t$ is uniformly distributed on $[L^z_t,U^z_t]$, and in particular, it follows that $x$ (the height of the point $z  = (t,x)$) is uniformly distributed on the random interval $[{L}^z_t,{U}^z_t]$.


\subsection{An extrinsic construction of the veins leading to $k$ points} \label{sec:extrinsic}

In the previous section we saw how to construct the vein $V^z$ to a single point $z$. In this section, for $k$ distinct points $z_1,\ldots,z_k$ of $(0,\infty) \times \mathbb{R}$, we consider the joint behaviour of the $k$ veins $V^{z_1},\ldots,V^{z_k}$ (and their continuations) leading to these points.

To lighten notation, we will write
\begin{align*}
V^j := (L^j_s,C^j_s,U^j_s)_{s \in [0,t_j]} \quad \text{and}  \quad \tilde{V}^j := (L^j_s,U^j_s)_{s \in [t_j,\tau_j]}
\end{align*}
for the vein leading to $z_j$ and its continuation. We will also write $\sigma_j$ and $\tau_j$ for the start and end time of the bubble containing $z_j$. 

\begin{figure}[h!]
\centering
\includegraphics[width=1.1\textwidth]{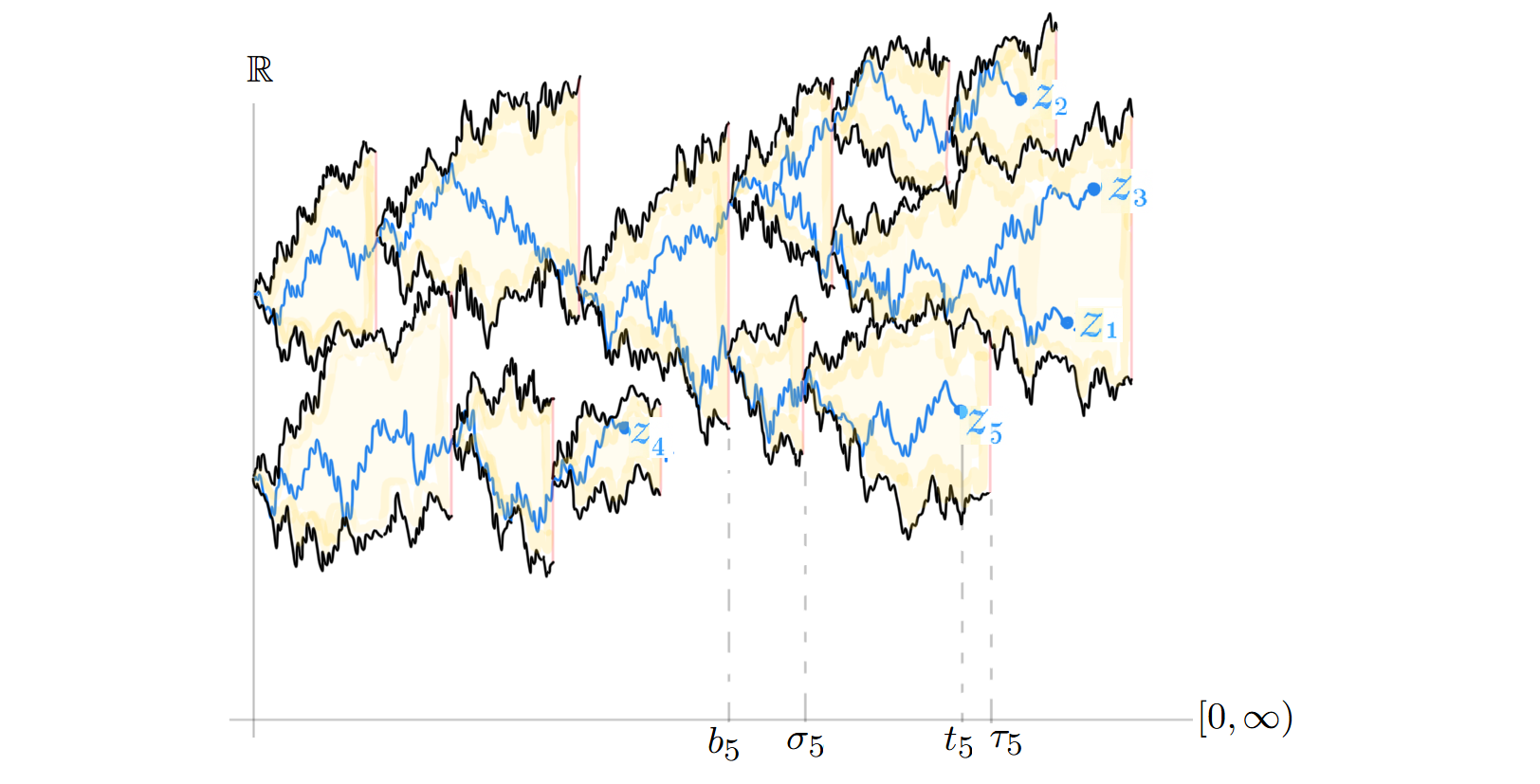}
\caption{The veins leading to $k=5$ distinct points. The times $b_5,\sigma_5,t_5$ and $\tau_5$ are listed on the time axis.}
\label{fig:8vein}
\end{figure}

Note that these veins may interact in several possible ways:
\begin{itemize}
\item It is entirely possible that some $z_j$ and $z_k$ (with $j \neq k)$ lie in the same bubble of $\mathcal{M}(R)$. In this case, we will have $\sigma_j = \sigma_k, \tau_j=  \tau_k$ and $L^k_s=L^j_s$ and $U^k_s = U^j_s$ for all $s \in [0,\tau_j] = [0,\tau_k]$. In this case, we will also have $C^j_s = C^k_s$ for all $s \in [0,c_{jk}]$, where $\sigma_j < c_{jk} < \min \{ t_j, t_k \}$ is the point in time at which the dual web paths emitted from $z_j$ and $z_k$ coalesce.  See for instance Figure \ref{fig:8vein}, where $z_1$ and $z_3$ are contained in the same bubble.
\item At some point in time, the lower part of one vein may be equal to the upper part of another vein, i.e.\ we may have $L^j_s = U^k_s$ for some $s$ is a non-empty time window. Going forwards in time, these paths will coincide until the first discontinuity event of either the $j^{\text{th}}$ or $k^{\text{th}}$ vein. {This occurs on many occasions in Figure \ref{fig:8vein}. For instance, $U^4$ coincides with $L^1$ for an initial period near the beginning of time.}

\item More generally, suppose that the dual paths associated with $z_j$ and $z_k$ coalesce at some spacetime point contained in a bubble of $\mathcal{M}(R)$ not containing both $z_j$ and $z_k$. Then the $j^{\text{th}}$ and $k^{\text{th}}$ vein will agree up until this bubble, and then separate afterwards. {We write $b_k$ for the first time that the $k^{\text{th}}$ vein is contained in a new bubble. See for instance in Figure \ref{fig:8vein}, where $b_5$ is the first time the vein associated with $k=5$ splits from the previous veins.}
\end{itemize}

We now describe how one can reconstruct the $R$-marble from these veins. Loosely speaking, we let
\begin{align*}
\alpha^k,\beta^k := \text{Lower, Upper web paths emitted from the $(0,2)$ point $(\sigma_k,C^k_{\sigma_k})$}.
\end{align*}
To be more precise, recall that $\sigma_k$ is the birth time of the bubble containing $z_k$. Then $(\sigma_k,C^k_{\sigma_k})$ is a $(0,2)$ point of the Brownian web, and thus there are respective lower and upper Brownian web paths $\alpha^k := (\alpha^k_t)_{t \geq \sigma_k}$ and $\beta^k := (\beta^k_t)_{t \geq \sigma_k}$ emitted from this $(0,2)$ point. These paths satisfy
\begin{align*}
\alpha^k_t = L_t^k \quad \text{and} \quad \beta^k_t = U_t^k \qquad \text{for $t \in [\sigma_k,\tau_k)$},
\end{align*}
though we emphasise that $\alpha^k_t $ and $\beta^k_t$ are defined for all $t \geq \sigma_k$. 
In particular, we have the strict inequality $\alpha^k_t < \beta^k_t$ for $t \in (\sigma_k,\tau_k)$. 

\begin{figure}[h!]
\centering
\includegraphics[width=0.9\textwidth]{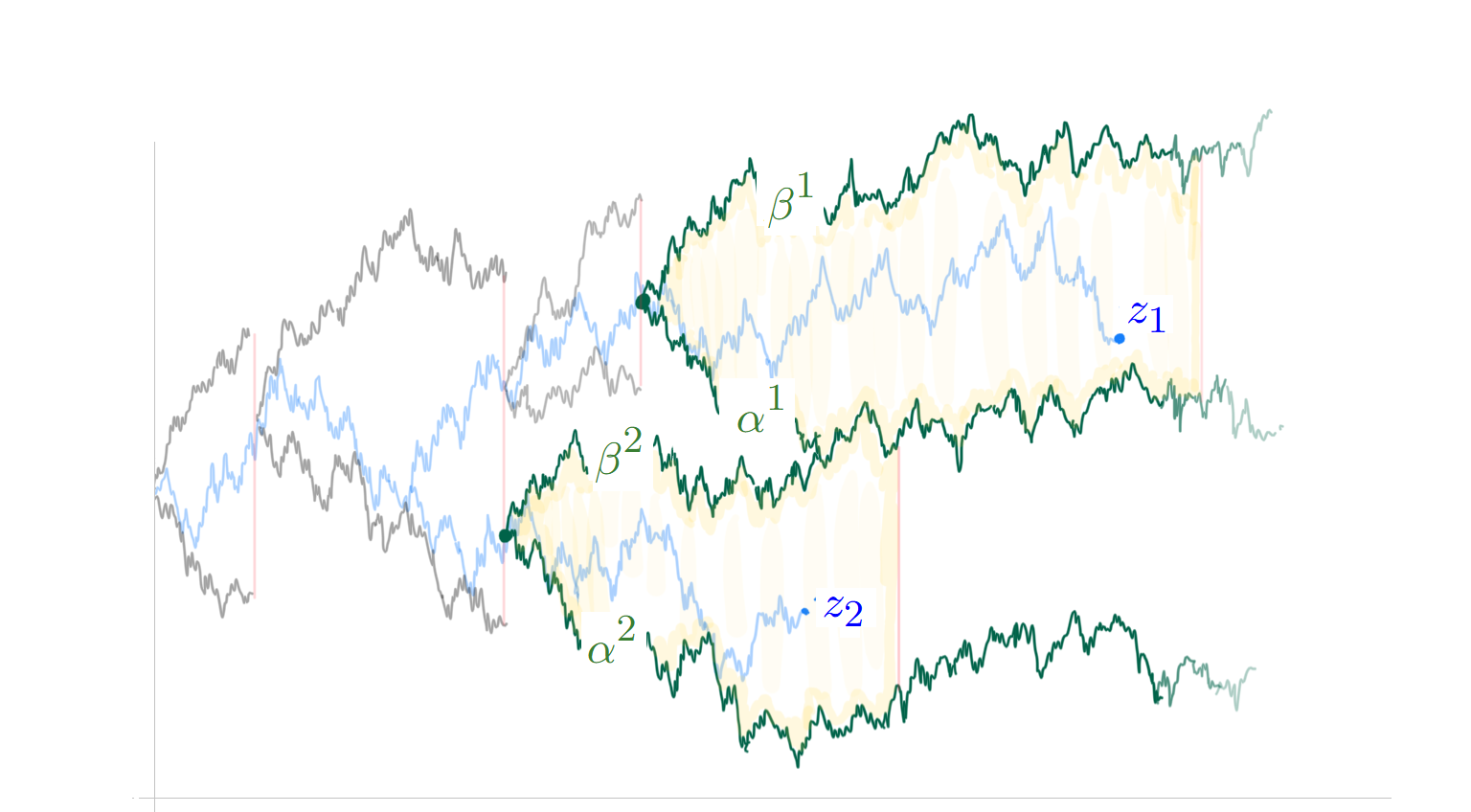}
\caption{The upper and lower paths emitted at the initial point of bubbles containing $z_1$ and $z_2$ are depicted in green.}
\label{fig:timvein}
\end{figure}

Recall that $\tau_k$ may simply be the moment at which Brownian paths $L_t^k$ and $U_t^k$ coalesce after time $t_k$, or it may denote the bubble `bursting'. In the latter case, we note that while we will still have $\alpha^k_{\tau_k} < \beta^k_{\tau_k}$, since $\alpha^k$ and $\beta^k$ are paths of the Brownian web, they will inevitably coalesce at some point in the future after time $t_k$. 

If $z_j$ and $z_k$ lie in the same bubble of $\mathcal{M}(R)$, we will have $(\alpha^k,\beta^k) = (\alpha^j,\beta^j)$. 

Define
\begin{align*}
\mathcal{M}_k(R) := \bigcup_{j=1}^k \{ \alpha^j, \beta^j \}
\end{align*}
to be the union of the $2k$ paths emitted from the $k$ $(0,2)$ start points of bubbles containing $z_1,\ldots,z_k$. 

Given a sequence $(\mathcal{N}_k)_{k \geq 1}$ of elements of $\mathcal{H}$ satisfying $\mathcal{N}_k \subseteq \mathcal{N}_{k+1}$, in the remainder of the article we will use the notation 
\begin{align*}
\overline{ \lim_{k \to \infty} }  \mathcal{N}_k := \text{Closure of the union } \bigcup_{k \geq 1} \mathcal{N}_k,
\end{align*}
where the closure is taken in the $\mathrm{d}_{\mathcal{H}}$ topology. 

The following lemma ensures that we can reconstruct the $R$-marble from these paths. 

\begin{proposition} \label{prop:reconstruct}
{Let $R$ be bounded.} 
We have
\begin{align} \label{eq:closedrel}
\mathcal{M}(R) := \overline{\lim_{k \to \infty}}\mathcal{M}_k(R).
\end{align}

\end{proposition}
\begin{proof}
Since each $\alpha^k$ and $\beta^k$ is an element of $\mathcal{M}(R)$, each $\mathcal{M}_k(R)$ is a subset of $\mathcal{M}(R)$. Since $\mathcal{M}(R)$ is closed, the set on the right-hand side of \eqref{eq:closedrel} is a subset of that on the left.

\begin{figure}[h!]
\centering
\includegraphics[width=0.9\textwidth]{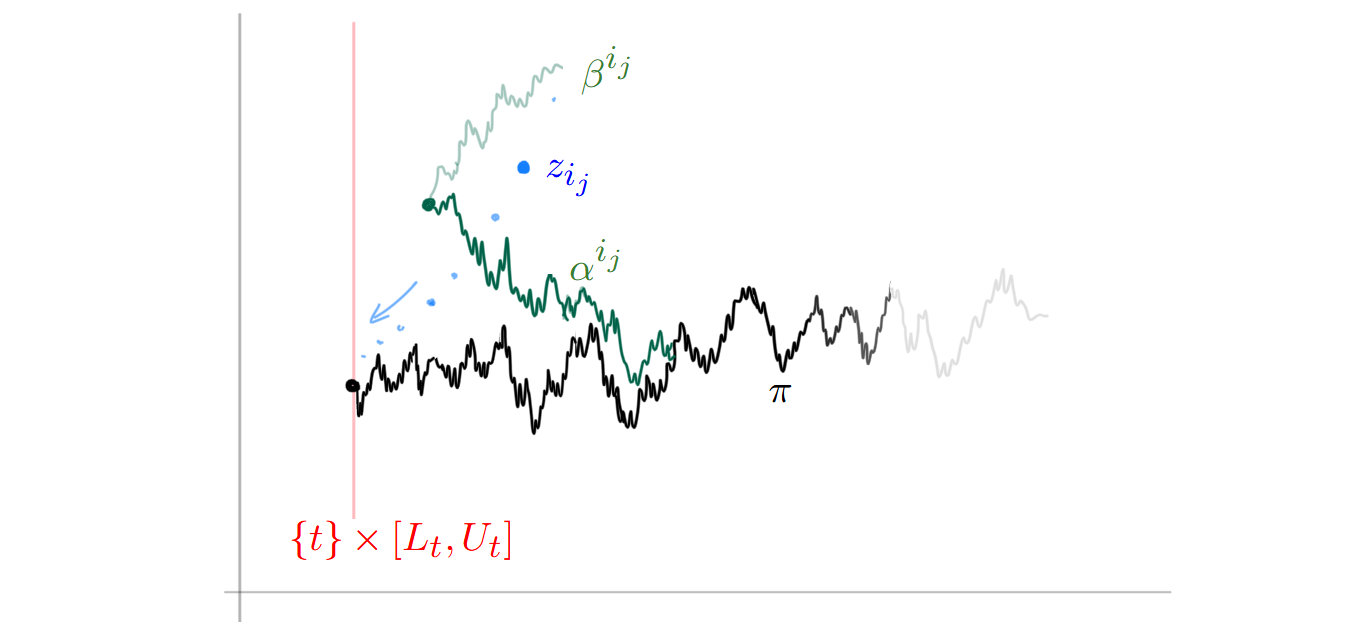}
\caption{As a subsequence of points $z_{i_j}$ approaches the start location of a path $\pi$ from above and to the right, the associated lower paths $\alpha_{i_j}$ converge to $\pi$.}
\label{fig:pathsabove}
\end{figure}

As for the converse inclusion, 
it is enough to show that for every path $\pi$ in $\mathcal{M}(R)$, and every $\varepsilon > 0$, there exists $k$ and some $\pi'$ in $\mathcal{M}_k(R)$ such that $\mathrm{d}_{\mathrm{path}}(\pi,\pi') < \varepsilon$. From the construction of the marble in Section \ref{sec:marbleconstruction}, we may assume that either $\sigma_\pi = 0$ or that $\pi$ was created at an event where an interval of the form $\{t\} \times [L_t,U_t]$ fragmented, in which case $\sigma_\pi = t$. Consider now a dense subsequence of points $z_{i_j} = (t_{i_j},x_{i_j})$ approaching the start point $(\sigma_\pi,\pi_{\sigma_\pi})$ from the right and from above the path (i.e. $t_{i_j} > \sigma_\pi$ and $x_{i_j} > \pi_{t_{i_j}}$ for each $i_j$). Each such point gives rise to a path $\alpha^{i_j}$ in $\mathcal{M}_{i_j}(R)$ that initiates to the left of $t_{i_j}$ and runs below the point $z_{i_j}$, and necessarily runs above the interval $\{t\} \times [L_t,U_t]$. Since the $z_{i_j}$ may be taken to approach $(\sigma_\pi, \pi_{\sigma_\pi})$ from any angle, it follows that there are a sequence of paths in $\mathcal{M}_{i_j}(R)$ converging to $\pi$; see Figure \ref{fig:pathsabove}.  It follows that $\pi$ is an element of $\overline{\lim_{k\to \infty}} \mathcal{M}_k(R)$, as required.


%
\end{proof}

We note from the proof of Proposition \ref{prop:reconstruct} it is in fact evident that $\mathcal{M}(R)$ may be recontructed using either just the lower or just the upper paths emitted from at the beginning of a bubble, i.e.\ we have
\begin{align*}
\mathcal{M}(R) = \overline{\lim_{k \to \infty}} \{ \alpha^1(R),\ldots,\alpha^k(R) \} = \overline{\lim_{k \to \infty}} \{ \beta^1(R),\ldots,\beta^k(R) \}.
\end{align*}

\section{An intrinsic construction of the veins leading to $k$ points and proof of our main results} \label{sec:mainproof}

We note that the construction in Section \ref{sec:explore} of the first $k$ veins was \emph{extrinsic} in the sense that in order to construct the law of the $k^{\text{th}}$ vein, we had to first construct the entire $R$-marble $\mathcal{M}(R)$, and thereafter track the progress of the bubbles leading up to the bubble containing $z_k$.

In this section, we will present an alternative \emph{intrinsic} probabilistic construction of the joint veins leading up to points $z_1,\ldots,z_k$. Thereafter, we can construct the marble from these veins as $k \to \infty$. This construction will have the advantage that we will be able to couple the joint veins leading up to these points for different rate functions $R$ simultaneously. This coupling will be robust to taking limits in truncated rate functions $R_n(g) = R(g) \wedge n$ in that our veins will converge as $n \to \infty$. 

\subsection{The intrinsic vein construction} \label{sec:intrinsic}
Let $R$ be a bounded and measurable rate function. 
Let $(\mathcal{P}_i : i \geq 1)$ be a sequence of independent and identically distributed Poisson processes on $[0,\infty)^2$ with Lebesgue intensity. Let $(\mathcal{W},\hat{\mathcal{W}})$ be a realisation of the Brownian web on $[0,\infty) \times \mathbb{R}$ together with its dual. Let $\mathcal{D} = \{ z_1,z_2,\ldots \}$ be a deterministic countable dense subset of $(0,\infty) \times \mathbb{R}$. We now describe how to use the series of Poisson processes $(\mathcal{P}_i : i \geq 1)$ to \emph{intrinsically} construct veins $V^1,\ldots,V^k$ and their continuations leading up to the points $z_1,\ldots,z_k$. The $j^{\text{th}}$ vein and its continuation will take the form
\begin{align*}
V^j = (L^j_s,C^j_s,U^j_s)_{s \in [0,t_j]} \quad \text{and} \quad \tilde{V}^j = (L^j_s,U^j_s)_{s \in [t_j,\tau_j]},
\end{align*}
so that $(t_j,C^j_{t_j}) = (t_j,x_j) = z_j$. The bubble containing $z_j$ will be the set
\begin{align*}
\mathcal{B}^j:= \{ (t,x) : \sigma_j < t < \tau_j : L^j_t < x < U^j_t \}.
\end{align*}

Having constructed the first $k-1$ veins and their continuations, the $k^{\text{th}}$ vein is constructed as follows: 
\begin{enumerate}
\item Let $\hat{\pi}^k := \{ (s,C^k_s) : s \in [0,t_k] \}$ be the dual web path emitted from $z_k = (t_k,x_k)$. 

\item Let
\begin{align*}
b_k := \inf \{ t \geq 0: C_t^k \notin [L_t^j,U_t^j] \quad \text{for all $j=1,\ldots,k-1$, $t \in [0,t_j]$} \}
\end{align*}
be the \textbf{separation time} of the $k^{\text{th}}$ vein. If $z_k$ lies in some bubble $\mathcal{B}^j$ for some $1 \leq j \leq k-1$, then we set $b_k = \inf \varnothing :=+ \infty$. We note that the separation time, if finite, is almost surely a fragmentation time of some bubble. 
\item For $t < b_k$, if $C^k_t \in [L^j_t,U^j_t]$ for some $j$, set $L^k_t=L^j_t$ and $U^k_t = U^j_t$. 
\item For $t \geq b_k$, $(L^k_t,U^k_t)_{b_k \leq t \leq t_k }$ initially begin by tracking the unique lower and upper paths in the Brownian web emitted from the (almost surely) $(0,2)$ point $(b_k,C^k_{b_k})$. 
\item For $t \in [b_k,t_k]$, if there exists a point $(t,y)$ of the Poisson process $\mathcal{P}_k$ for which
\begin{align} \label{eq:RR}
y \leq R(U^k_{t-}-L^{k}_{t-}),
\end{align}
then we have a discontinuity of the $k^{\text{th}}$ vein at time $t$: namely, we reset $(L^k_u,U^k_u)_{t \leq u \leq t_k }$ to be the stochastic processes tracking the unique lower and upper paths in the Brownian web emitted from the $(0,2)$ point $(t,C^k_{t})$. 
\item  After time $t_k$, $(L^k_s)_{s \geq t_k}$ and $(U^k_s)_{s \geq t_k}$ continue fluctuating according to independent Brownian motions that are killed at a random time $\tau_k$ after $t_k$. This random time $\tau_k$ occurs either as soon as $L^k$ and $U^k$ meet, or if there is a point $(t,y)$ of the Poisson process $\mathcal{P}_k$ with $t \geq t_k$ for which $y \leq R(U_{t-}^k-L_{t-}^k)$. 
\item Finally, $\sigma_k := \sup \{t \leq t_k : L_t^k = U^k_t \}$ is the birth time of the bubble containing $z_k$. 
\end{enumerate}

We note that for $t \in [b_k,\tau_k]$, the upper and lower paths $L^k$ and $U^k$ diffuse according to Brownian web paths that reflect off the path of $(C^k_t)_{t \in [0,t_k]}$ (which is only defined up until time $t_k$ rather than $\tau_k$), and may coalesce with the other lower and upper paths $\{L^j,U^j : 1 \leq j \leq k-1 \}$ forwards in time.

The rate function $R$ influences the construction only through the equation \eqref{eq:RR}. We will be interested in constructing the $k$ veins simultaneously for different values of $R$, and to emphasise this in the notation we will sometimes write 
\begin{align} \label{eq:VR}
    V^k := V^k(R) := (L^{k}_t(R),C^{k}_t,U^{k}_t(R))_{t \in [0,t_k]} \quad \text{and} \quad \tilde{V}^k = \tilde{V}^k(R) = (L^k_t(R),U^k_t(R))_{t \in [t_k,\tau_k(R)]}
\end{align}
for the vein to $z_k$ and its continuation. By Lemma \ref{lem:veiny}, the stochastic process $(L^k_t(R) - C^k_{t_k}, C^k_t - C^k_{t_k}, U^k_t(R) - C^k_{t_k})_{t \in [0,t_k]}$ has the law of a standard $R$-vein run until time $t_k$.

It is clear from the construction that the law of the collection of the stochastic processes $\{V^1,\ldots,V^k\}$ does not depend on the ordering of the set $\{z_1,\ldots,z_k\}$, and coincides with the law of the $k$ veins described in Section \ref{sec:extrinsic}. 

As in Section \ref{sec:extrinsic}, we may now construct the $R$-marble as a limit from these veins. Recalling that $\sigma_k(R)$ is the birth time of the bubble containing $z_k$, as in Section \ref{sec:extrinsic} let
\begin{align*}
\alpha^k(R),\beta^k(R) := \text{Lower, Upper web paths emitted from the $(0,2)$ point $(\sigma_k(R),C^k_{\sigma_k(R)})$}.
\end{align*}
Set
\begin{align} \label{eq:MkR}
\mathcal{M}_k(R) := \bigcup_{j=1}^k \{ \alpha^j(R),\beta^j(R) \}
\end{align}
and define
\begin{align} \label{eq:MkRlim}
\mathcal{M}(R) := \overline{\lim_{k \to \infty}} \mathcal{M}_k(R).
\end{align}
Since each $\mathcal{M}_k(R)$ has the law of the $k$ veins in the $R$-marble, it follows from Proposition \ref{prop:reconstruct} that this new random variable $\mathcal{M}(R)$ constructed from the veins has the law of the $R$-marble.

Moreover, for different bounded rate functions $R$, the random variables $\mathcal{M}(R)$ are coupled in that they are constructed using the same copy of the Brownian web and its dual $(\mathcal{W},\hat{\mathcal{W}})$ and the same sequence $(\mathcal{P}_i : i \geq 1)$ of Poisson processes.

\subsection{Proof of Theorem \ref{thm:main} in the upper regular case} 

In this section, we are ready to prove the easier half of Theorem \ref{thm:main}. That is, we now prove Theorem \ref{thm:main} in the case where $R$ is upper quadratic-regular, i.e.\ $R$ is decreasing and either
\begin{align*}
\liminf_{g \downarrow 0} g^2 R(g)  > 6  \qquad \text{or} \qquad R(g) = 6/g^2 \quad \text{for $g > 0$}.
\end{align*}

Before the proof, we require the following proposition:

\begin{proposition} \label{prop:webby}
Let $R$ be upper quadratic-regular. Let $R_n(g) := R(g) \wedge n$. Then as $n \to \infty$ we have the almost sure convergence
\begin{align*}
\mathcal{M}_k(R_n ) \to \mathcal{W}_k := \{ \pi^{z_1},\ldots,\pi^{z_k} \}
\end{align*}
where $\pi^{z_k}$ is simply the Brownian web path emitted from $z_k$.
\end{proposition}
\begin{proof}
Let $V^j(R_n)$ be the $j^{\text{th}}$ vein for the $n^{\text{th}}$ truncation $R_n$ of the rate function $R$. Then since the stochastic process $U^j(R_n) - L^j(R_n)$ has the same law as $\sqrt{2}X$ under $\mathbf{Q}_0^{R_n}$ on $[0,\tau_j(R_n)]$, by Theorem \ref{thm:newtransition} it converges to the zero process on $[0,t_j]$ as $n \to \infty$. In particular, it follows that the birth time $\sigma_j(R_n)$ of the bubble containing $z_j$ converges to $t_j$ as $n$ tends to infinity. In particular, the lower and upper paths $\alpha^j(R_n)$ and $\beta^j(R_n)$ emitted from the $(0,2)$ point $(\sigma_j, C^j_{\sigma_j(R_n)})$ both converge (in the path metric) as $n \to \infty$ to the unique Brownian web path emitted from $(t_j,C^j_{t_j}) = (t_j,x_j) = z_j$. 

\end{proof}

\begin{proof}[Proof of Theorem \ref{thm:main} in the upper quadratic-regular case]

Since $\mathcal{M}(R_n)$ is a subset of the Brownian web, it is sufficient to establish that for every  $\pi$ in $\mathcal{W}$ and $\varepsilon > 0$ there exists $n_0$ such that for all $n \geq n_0$, $\mathcal{M}(R_n)$ contains a path $\pi_n$ satisfying $\mathrm{d}_{\mathrm{path}}(\pi, \pi_n) < \varepsilon$. 
Fix $j$ such that the (almost surely) unique Brownian web path $\pi^{z_j}$ emitted from $z_j$ satisfies $\mathrm{d}_{\mathrm{path}}(\pi^{z_j},\pi) < \varepsilon/2$. 
Now according to the proof of Proposition \ref{prop:webby}, we can choose $n_0$ such that for all $n \geq n_0$ we have $\mathrm{d}_{\mathrm{path}}( \alpha^j(R_n) , \pi^{z_j} ) < \varepsilon/2$. 
(A similar statement is true for $\beta^j(R_n)$.) In particular, for all $n \geq n_0$, by the triangle inequality we have 
\begin{align*}
\mathrm{d}_{\mathrm{path}}( \alpha^j(R_n) , \pi ) \leq \mathrm{d}_{\mathrm{path}}( \alpha^j(R_n) , \pi^{z_j} ) + \mathrm{d}_{\mathrm{path}}( \pi^{z_j} , \pi ) < \varepsilon,
\end{align*}
thereby completing the proof.

\end{proof}

\subsection{Proof of Theorem \ref{thm:main} in the lower quadratic-regular case}
In this section we work towards our proof of Theorem \ref{thm:main} in the more difficult case where $R$ is lower quadratic-regular, i.e.\ where $R:(0,\infty) \to [0,\infty)$ is a decreasing function satisfying $\limsup_{ g \downarrow 0} g^2 R(g) < 6$.

The main task at hand is establishing the convergence of the first $k$ veins associated with the truncated rate function $R_n(g)$ as $n \to \infty$, where we recall from the previous section the construction of the set $\mathcal{M}_k(R)$ of veins for any bounded $R$. Namely, most of our work in this section involves establishing the following proposition:

\begin{proposition} \label{prop:local}
Suppose that $R$ is lower quadratic-regular. Let $R_n(g) = R(g) \wedge n$. Let $V^j(R_n)$ for $j =1,\ldots,k$ be the veins constructed in Section \ref{sec:intrinsic} for bounded rate functions. Then almost surely:

\begin{enumerate}
\item The first $k$ veins are \textbf{ultimately nested}: That is, there exists a random $n_0 := n_0(z_1,\ldots,z_k)$ such that for all $n_2 \geq n_1 \geq n_0$ we have 
\begin{align*}
L^j_t(R_{n_1}) \leq L^j_t(R_{n_2}) \leq C_t^j \leq U^j_t(R_{n_2}) \leq U^j_t(R_{n_1}) \quad \text{for $ t \in [0,t_j].$}
\end{align*}
In particular, for each $j=1,\ldots,k$ and $t \in [0,t_j]$, the following limits exist:
\begin{align*}
L_t(R) := \lim_{n \to \infty} L_t^j(R_n) \quad \text{and} \quad U_t(R) := \lim_{n \to \infty} U_t^j(R_n).
\end{align*}
\item The first $k$ veins are \textbf{ultimately equal away from zero}: That is, for every $\varepsilon > 0$, there exists a random $n_0 = n_0(\varepsilon)$ such that for all $n \geq n_0$, for each $j =1,\ldots,k$ we in fact have the equalities 
\begin{align*}
L^j_t(R_n) = L^j_t(R) \quad \text{and} \quad U^j_t(R_n) = U^j_t(R)
\end{align*}
for all $t \in [0,t_j]$ not contained in the initial $\varepsilon$ time period of any given excursion of $(U^j_t(R) - L^j_t(R))_{t \in [0,t_j]}$.

Moreover, there exists $n_0'$ such that for all $n \geq n_0'$, the continued veins satisfy $\tau_j(R_n) = \tau_j(R)$ and $\tilde{V}^j(R_n) = \tilde{V}^j(R)$. 

\item With the notation of \eqref{eq:MkR}, there is a random variable $\mathcal{M}_k(R)$ such that 
as $n \to \infty$, almost surely we have the convergence
\begin{align} \label{eq:Mkconv}
\mathcal{M}_k(R_n) \to \mathcal{M}_k(R)
\end{align}
in the $\mathrm{d}_{\mathcal{H}}$ metric.
\end{enumerate}

\end{proposition}

\begin{figure}[h]
\begin{center}
\includegraphics[width=9cm]{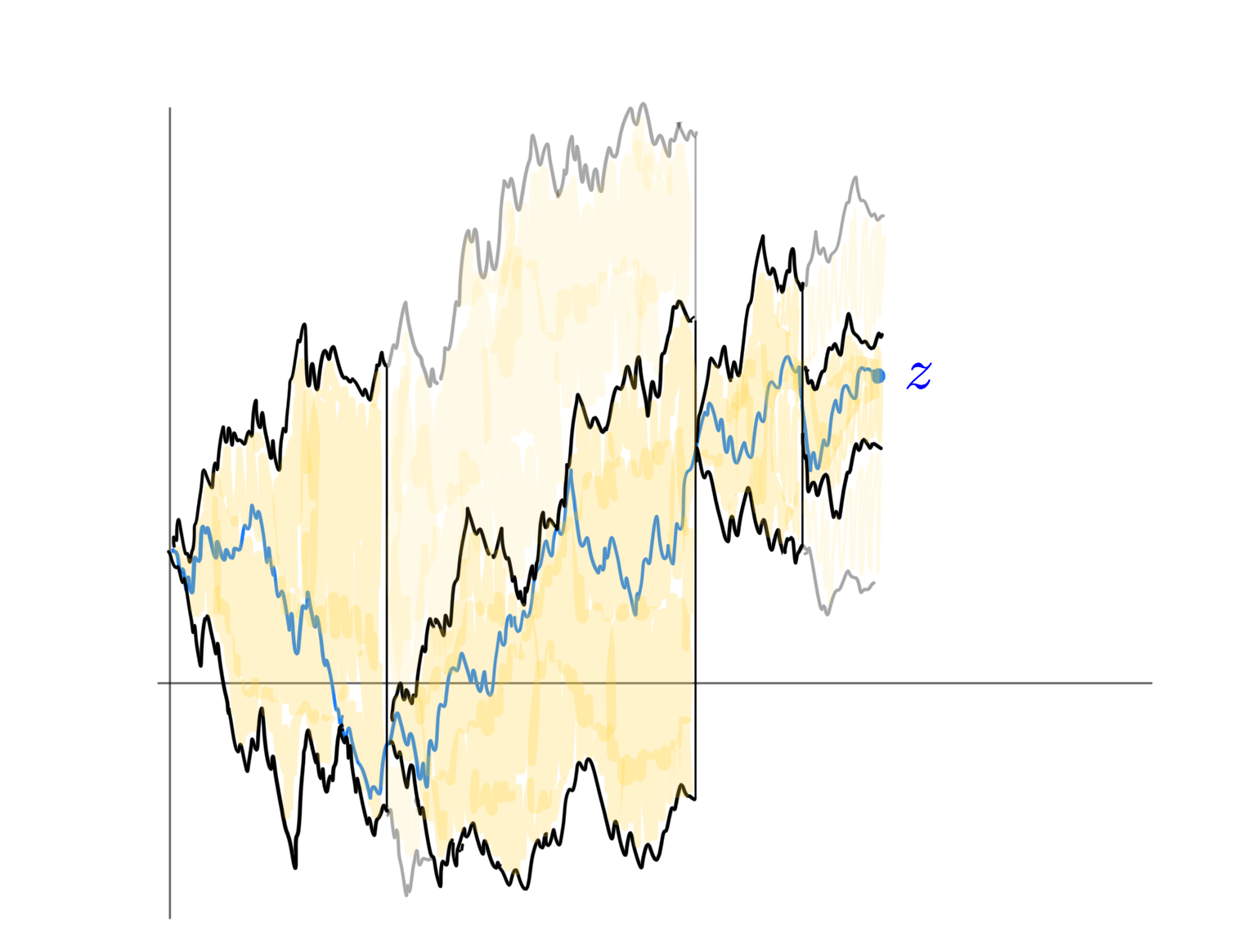}
\caption{Given two decreasing fragmentation rates $R_1$ and $R_2$ with $R_2(g) \geq R_1(g)$, there is a natural coupling of any given vein such that the vein associated with the larger fragmentation rate sits inside the vein with the smaller fragmentation rate. In the diagram here, the vein associated with $R_1$ is depicted in grey and the vein associated with $R_2$ is depicted in black.}
\label{fig:dveinmulti99}
\end{center}
\end{figure}

In order to prepare the proof of Proposition \ref{prop:local}, let us touch further on the idea of \textbf{nestedness} of veins.
This is the idea that under certain conditions on fragmentation rates $R_1$ and $R_2$, we have the containment of one vein inside another, in that for each $t \in [0,t_j]$ we will have the inclusion of intervals
\begin{align*}
[L^j_t(R_2),U^j_t(R_2)] \subseteq[L^j_t(R_1),U^j_t(R_1)]. 
\end{align*}
See Figure \ref{fig:dveinmulti99} for a depiction of the nestedness of two veins. 
We will ultimately be interested in characterising the possible nestedness of the first $k$ veins for rate functions $R_{n}(g)$ and $R_{n'}(g)$ where we have two different degrees $n$ and $n'$ of truncation. The story is reasonably straightforward when $k = 1$:
\begin{lemma} \label{lem:nested}
Let $R,R':(0,\infty) \to (0,\infty)$ be nonincreasing bounded functions satisfying $R'(g) \geq R(g)$ for all $g \in (0,\infty)$. Then the vein $V^1(R')$ is nested inside the vein $V^1(R)$ in that:
\begin{align*}
[L^1_t(R'),U^1_t(R')] \subseteq [L^1_t(R),U^1_t(R)]. 
\end{align*}
\end{lemma}

\begin{proof}
This proof is analogous to that of Lemma \ref{lem:beslim}. 
Provided that at some moment in time we have $[L^1_t(R'),U^1_t(R')] \subseteq [L^1_t(R),U^1_t(R)]$, we also have the inequality of collapse rates
\begin{align*}
R'( U^1_t(R') - L^1_t(R')) \geq R'( U^1_t(R) - L^1_t(R)) \geq R( U^1_t(R) - L^1_t(R)),
\end{align*}
where to obtain the first inequality above we used the fact that $R'$ is decreasing, and to establish the second we used $R'(g) \geq R(g)$. Since the collapses of the two veins are coupled by the same Poisson process $\mathcal{P}^1$, it follows that if the interval $[L^1_t(R),U^1_t(R)]$ collapses, so does $[L^1_t(R'),U^1_t(R')]$. Thus we have the purported inclusion of intervals for all time $t \in [0,t_1]$. 
\end{proof}

Unfortunately, the simple generalisation of Lemma \ref{lem:nested} to several veins is not true. That is, if $R$ and $R'$ are as in the statement of Lemma \ref{lem:nested}, it is not necessarily true that for all $j \geq 2$ we have the inclusion $[L^j_t(R'),U^j_t(R')] \subseteq [L^j_t(R),U^j_t(R)]$. 

\begin{figure}[h]
\begin{center}
\includegraphics[width=9cm]{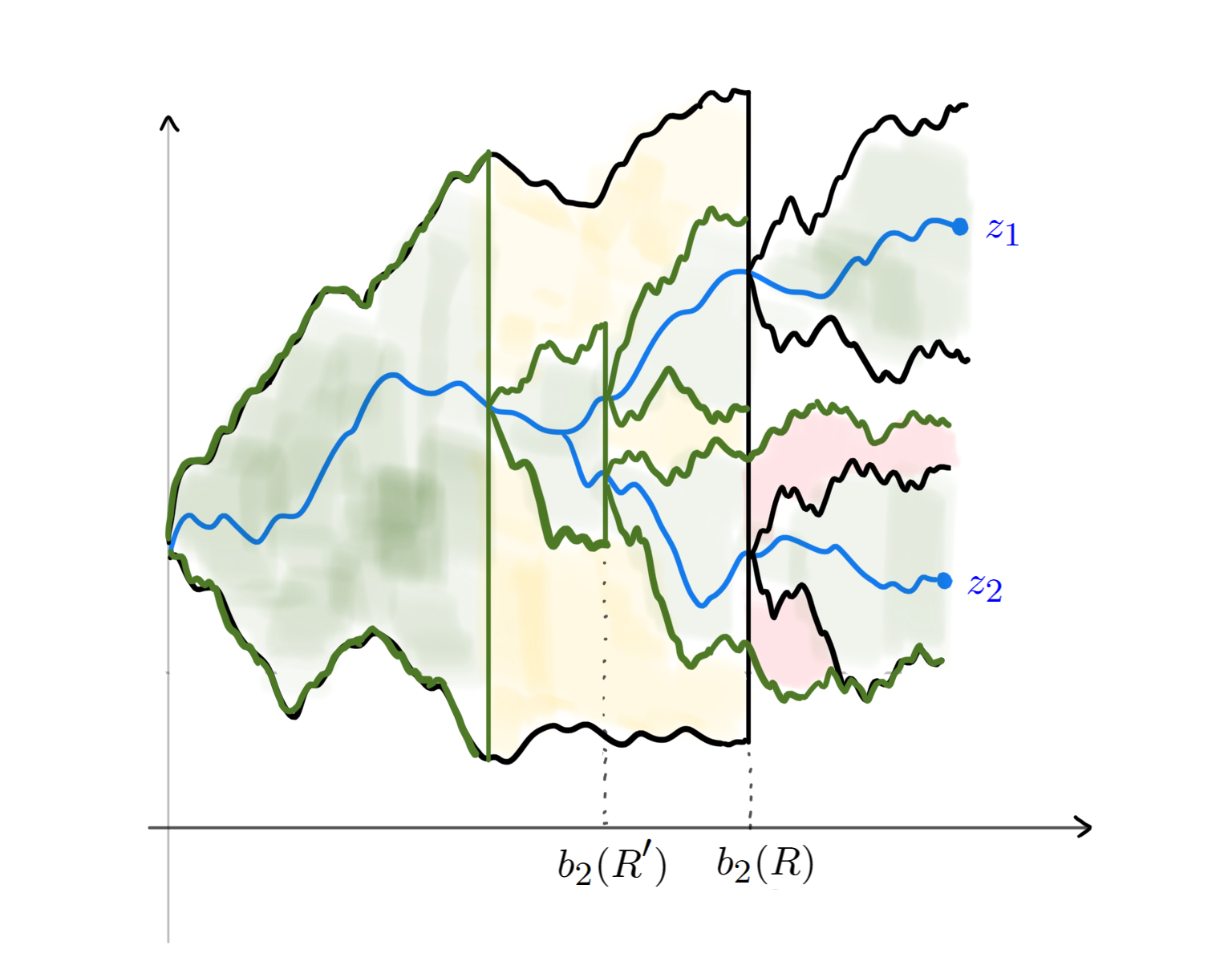}
\caption{In this figure, we have traced out veins to two distinct elements $v_1$ and $v_2$ of spacetime, for rate $R$ in black and a larger rate $R'$ in green. In the diagram, a green area denotes $(t,x)$ such that $x \in [L^j_t(R'),U^j_t(R')] \subseteq [L^j_t(R),U^j_t(R)]$ for some $j$. The yellow areas denote $(t,x)$ such that $x \in [L^j_t(R),U^j_t(R)] - [L^j_t(R'),U^j_t(R')]$, and the red areas denote the areas that breach the nestedness, i.e., $(t,x)$ such that $x \in  [L^2_t(R'),U^2_t(R')] - [L^2_t(R),U^2_t(R)]$. We have smoothed the Brownian paths for clarity.}\label{fig:nonnest}
\end{center}
\end{figure}

To see what can go wrong, the reader may inspect Figure \ref{fig:nonnest}, where we illustrate this phenomenon for $j =2$. By Lemma \ref{lem:nested}, it is true that the vein $V^1(R')$ is nested inside that of $V^1(R)$. However, as is depicted in Figure \ref{fig:nonnest}, it may well happen that the first time $b_2(R')$ that the dual web path to $z_2$ exits the $R'$-bubble $[L^1_t(R'),U^1_t(R')]$ occurs strictly before the first time $b_2(R)$ that this path exits the $R$-bubble $[L^1_t(R),U^1_t(R)]$ associated with the first vein. After this moment $b_2(R')$, the fragmentation event that occurs at $b_2(R)$ is not experienced by the vein $V^2(R')$, and as such, initially for $t \geq b_2(R)$, the vein $V^2(R')$ will be wider than $V^2(R)$. 

Nonetheless, in the course of our proof of Proposition \ref{prop:local}, we will overcome this possibility asymptotically:

\begin{proof}[Proof of Proposition \ref{prop:local}]
\textbf{Proof of base case $k=1$ of Proposition \ref{prop:local}}.\\
First, we prove (1). We note that when $k=1$, the ultimate nestedness stated in (1) follows immediately from Lemma \ref{lem:nested}. Indeed, Lemma \ref{lem:nested} states that for every $n' \geq n$ we have $R_{n'} \geq R_n$ so that we have the inequality
\begin{align} \label{eq:inclus}
L^1_t(R_n) \leq L^1_t(R_{n'})  \leq C^1_t \leq  U^1_t(R_{n'})   \leq U^1_t(R_{n}) \quad t \in [0,t_j].
\end{align}
Thus in the case $k=1$, (1) follows with $n_0 = 1$. 

Turning to the proof of (2), let us introduce the shorthand
$\ell_t := L_t(R) := \lim_{n \to \infty} L_t^j(R_n)$ and $u_t := U_t(R) := \lim_{n \to \infty} U_t^j(R_n)$.
Consider the limiting height process $(u_t-\ell_t)_{t \in [0,t_1]}$, which is a nonnegative stochastic process. By Lemma \ref{lem:law} we have
\begin{align*}
X^1 := (X_t^1)_{t \in [0,t_1]} := \left( \frac{u_t - \ell_t}{\sqrt{2}} \right)_{t \in [0,t_1]} \sim \mathbf{Q}_0^R.
\end{align*}

In the case where $R$ is lower quadratic-regular, $\mathbf{Q}_0^R$ governs a nontrivial stochastic process with a finite number of excursions of length $> \varepsilon$ on the time window $[0,t_1]$.
Note that each $L_t(R_n)$ as well as $\ell_t$ itself follows a segment of a path in the Brownian web. Since each $L_t(R_n)$ converges pointwise to $\ell_t$ for each $t$, and likewise $U_t(R_n)$ converges pointwise to $u_t$, it follows that there exists $n_0 = n_0(\varepsilon)$ such that for all $n \geq n_0$ we have
\begin{align*}
L_{t}^1(R_{n}) = \ell_{t} \qquad \text{and} \qquad U_{t}^1(R_{n}) = u_{t}
\end{align*} 
for all $t$ not in the initial time period $\varepsilon$ of any excursion. We note in particular that if $\varepsilon$ is chosen sufficiently small so that the final excursion of $(X_t)_{0 \leq t \leq t_1}$ (the one straddling $t_1$) has length $> \varepsilon$, then for all $n \geq n_0(\varepsilon)$ we will also have $L_t^1(R_n) = \ell_t$ and $U_t^1(R_n) = u_t$, so that we will also have the agreement of continued veins $\tilde{V}^1(R_n) = \tilde{V}^1(R)$ for all such $n$. That establishes (2) in the case $k=1$. 

As for (3), note again that by the nestedness that the final bubble associated with $R_n$ contains the final bubble containing $R$. Moreover, using the pointwise convergence, the upper and lower paths of the bubble for $R_n$ converge pointwise at each point to the upper and lower paths for $R$. Since all of these paths are continuous, by closure it again follows that
\begin{align} \label{eq:pathconv}
\mathrm{d}_{\mathrm{path}}(\alpha^1(R_n),\alpha^1(R)) \to 0 \quad \text{and} \quad \mathrm{d}_{\mathrm{path}}(\beta^1(R_n),\beta^1(R)) \to 0,
\end{align}
which implies that $\mathcal{M}_1(R_n) \to \mathcal{M}_1(R)$ as $n \to \infty$.

\vspace{2mm}
\textbf{Proof of inductive step}.\\
With the base case $k=1$ now proved, we prove the inductive step. Suppose now the statement of the result has been proven up to $k-1$. To lighten notation we write
\begin{align*}
    \ell_t^j = L_t^j(R) \quad \text{and} \quad u_t^j = U_t^j(R)
\end{align*}
for the limit processes (which are guaranteed to exist, by the inductive hypothesis, for $j=1,\ldots,k-1$). 

 Consider the dual web path backwards in time from $z_k$ and its interaction with the limiting veins $V^1(R),\ldots,V^{k-1}(R)$ and their continuations $\tilde{V}^1(R),\ldots,\tilde{V}^{k-1}(R)$. Write $b_k(R_n)$ and $b_k(R)$ for the respective separation times of the $k^{\text{th}}$ vein for rates $R_n$ and $R$.

 Suppose that $b_k(R) < +\infty$ and the path of $C^k$ lies in the $j^{\text{th}}$ vein for $R$ just before time $t$, i.e., there exists $\varepsilon > 0$ such that for all $b_k(R) - \varepsilon < t < b_k(R)$ we have $C_t^k \in [\ell_t^j,u_t^j]$. Then by ultimate equality away from zero, there exists some $n$ sufficiently large so that in that $L_t^j(R_n) = \ell_t$ and $U_t^j(R_n)$ for all $t \in [b_k(R)-\varepsilon,b_k(R))$. In particular, by construction for all such $n$ we will have $V_t^k(R_n) = V_t^j(R_n)$, and parts (1) and (2) on nestedness and equality follow for all $t \in [0,b_k(R)]$. 

 As for the proofs of (1) and (2) for the $k^{\text{th}}$ vein on the time interval $t \geq b_k(R)$, the proof here is identical to the $k=1$ case since the vein is separated. The proof of (3) also follows from identical arguments to the $k=1$ case. 

\end{proof}

\begin{proof}[Proof of Theorem \ref{thm:main} in the case where $R$ is lower-quadratic-regular]
This essentially boils down to showing that various limits commute. Indeed, on the one hand, we see that as a consequence of Proposition \ref{prop:local} part (3) that for each $k \geq 1$ the limit \begin{align*}
\mathcal{M}_k(R) := \lim_{n \to \infty} \mathcal{M}_k(R_n)
\end{align*}
exists.
Plainly $\mathcal{M}_k(R) \subseteq \mathcal{M}_{k+1}(R)$ from the construction, and accordingly the limit 
\begin{align*}
\mathcal{M}(R) := \overline{ \lim_{k \to \infty} } \mathcal{M}_k(R) 
\end{align*}
also exists.

On the other hand, since each $R_n$ is bounded, by Proposition \ref{prop:reconstruct}, the $R_n$-marble is constructed via
\begin{align*}
\mathcal{M}(R_n) := \overline{\lim_{k \to \infty}} \mathcal{M}_k(R_n).
\end{align*}
We now claim that $\mathcal{M}(R_n)$ converges to $\mathcal{M}(R)$ as $n \to \infty$. To establish this claim, we need to verify that $\mathrm{d}_{\mathcal{H}} \left( \mathcal{M}(R_n) , \mathcal{M}(R)  \right)  \to 0$ as $n \to \infty$, which amounts to the verification that we have both
\begin{align}  \label{eq:haus1}
\lim_{n \to \infty} \sup_{\pi \in \mathcal{M}(R)} \inf_{\pi' \in \mathcal{M}(R_n)} \mathrm{d}_{\mathrm{path}}( \pi , \pi') = 0
\end{align}
and
\begin{align}  \label{eq:haus2}
\lim_{n \to \infty} \sup_{\pi \in \mathcal{M}(R_n)} \inf_{\pi' \in \mathcal{M}(R)} \mathrm{d}_{\mathrm{path}}( \pi , \pi') = 0
\end{align}

With a view to first proving \eqref{eq:haus1}, we first establish the weaker claim that
\begin{align} \label{eq:accumulate}
\text{For every $\pi$ in $\mathcal{M}(R)$, $\inf_{ \pi' \in \mathcal{M}(R_n)} \mathrm{d}_{\mathrm{path}}(\pi,\pi') \to 0$ as $n \to \infty$}. 
\end{align}

To see this, let $\pi \in \mathcal{M}(R)$. Since $\mathcal{M}(R):= \overline{\lim_{k \uparrow \infty}} \mathcal{M}_k(R)$, we may choose $k > 0, \pi'' \in \mathcal{M}_k(R)$ such that $\mathrm{d}_{\mathrm{path}}(\pi,\pi'') \leq \varepsilon/2$. Now since $\mathcal{M}_k(R_n) \to \mathcal{M}_k(R)$ as $n \to \infty$, we may choose $\pi' \in \mathcal{M}_k(R_n) \subseteq \mathcal{M}(R_n)$ such that $\mathrm{d}_{\mathrm{path}}(\pi', \pi'') \leq \varepsilon/2$. It follows that there exists $\pi' \in \mathcal{M}(R_n)$ such that 
\begin{align*}
\mathrm{d}_{\mathrm{path}}(\pi', \pi) \leq \mathrm{d}_{\mathrm{path}}(\pi',\pi'') + \mathrm{d}_{\mathrm{path}}(\pi'',\pi) \leq \varepsilon,
\end{align*}
by the triangle inequality. That proves \eqref{eq:accumulate}. 

We now claim that \eqref{eq:accumulate} can be upgraded to the stronger statement \eqref{eq:haus1} using the compactness of the underlying path space. Indeed, let $\varepsilon > 0$ be arbitrary. By compactness, we may choose $\pi_1,\ldots,\pi_\ell$ in $\mathcal{M}(R)$ such that for all $\pi \in \mathcal{M}(R)$ there exists $1 \leq i \leq \ell$ such that $\mathrm{d}_{\mathrm{path}}(\pi,\pi_i) \leq \varepsilon/2$. Now using \eqref{eq:accumulate} we may choose $n_0$ sufficiently large such that for all $n \geq n_0$ and for each $1 \leq i \leq k$ there exists $\pi'_i$ in $\mathcal{M}(R_n)$ such that $\mathrm{d}_{\mathrm{path}}(\pi_i,\pi_i') \leq \varepsilon/2$. It then follows from the triangle inequality that for every $\varepsilon > 0$, there exists $n_0 >0$ such that for every $\pi \in \mathcal{M}(R)$, there exists $\pi' \in \mathcal{M}(R_n)$ such that $\mathrm{d}_{\mathrm{path}}(\pi,\pi') \leq \varepsilon$, thereby establishing \eqref{eq:haus1}. 

We turn to the proof of \eqref{eq:haus2}. Let $\mathcal{S}$ denote the set of accumulation points of the sequence of sets $\mathcal{M}(R_n)$. In other words, $\mathcal{S}$ is the set of paths $\pi$ for which there exists a subsequence of paths $(\pi_{n_j})_{j \geq 1}$, $\pi_{n_j} \in \mathcal{M}(R_{n_j})$ with $n_j \to \infty$ as $j \to \infty$, such that $\mathrm{d}_{\mathrm{path}}(\pi_{n_j},\pi) \to 0$.

We claim that to establish \eqref{eq:haus2}, it is sufficient to show that $\mathcal{S} \subseteq \mathcal{M}(R)$. Indeed, if $\mathcal{S} \subseteq \mathcal{M}(R)$, but \eqref{eq:haus2} does not hold, then we can find $\varepsilon > 0$ and a sequence $(\pi_j)_{j \geq 1}$ of paths, $\pi_j \in \mathcal{M}(R_{n_j})$, such that 
\begin{align} \label{eq:ineq11}
\inf_{ \pi ' \in \mathcal{S}} \mathrm{d}_{\mathrm{path}}(\pi_j, \pi') \geq \inf_{ \pi ' \in \mathcal{M}(R) } \mathrm{d}_{\mathrm{path}}(\pi_j, \pi') \geq \varepsilon.  
\end{align}
By the compactness of the space $\Pi$ of paths, the sequence $(\pi_j)_{j \geq 1} $ has a convergent subsequence. But this subsequence must converge to an element of $\mathcal{S}$, contradicting \eqref{eq:ineq11}. Thus $\mathcal{S} \subseteq \mathcal{M}(R)$ implies \eqref{eq:haus2}.

Thus it remains to prove $S \subseteq \mathcal{M}(R)$. Let $\pi$ in $\mathcal{S}$. 
Let $z_j = (t_j,x_j)$ be a point of the dense set $\mathcal{D}$ lying above the path $\pi$, in that $t_j \geq \sigma_\pi$, and $x_j > \pi_{t_j}$. There exists an integer $n_j$ such that for all $n \geq n_j$, the first $j$ veins are nested. In particular, for all $n \geq n_j$ we have
\begin{align*}
\sigma_j(R_n) \leq \sigma_j(R) \qquad \text{and} \qquad \alpha^j_s(R_n) \leq \alpha^j_s(R) \quad \text{for all $s \in [\sigma_j(R), \infty]$}.
\end{align*}
In particular, the path $\pi$, which is an accumulation point of paths in $\mathcal{M}(R_n)$, must satisfy 
\begin{align*}
\pi_s \leq \alpha^j_s(R) \qquad \text{for all $s \geq \sigma_j(R) \vee \sigma_\pi$}
\end{align*}
In particular, we have the inequality
\begin{align*}
x_j > \alpha^j_{t_j}(R) \geq \pi_{t_j}.
\end{align*}
A similar argument says that if $z_k = (t_k,x_j)$ is a point lying below the path of $\pi$ (in that $t_k \geq \sigma_\pi$ and $x_k < \pi_{t_k}$), then 
\begin{align*}
\pi_s \geq \beta^k_s(R) \qquad \text{for all $s \geq \sigma_k(R) \vee \sigma_\pi$}.
\end{align*}
    It follows, in particular, since $z_i$ are dense, for any point $(s,\pi_s)$ on the limit path, there are paths of $\mathcal{M}(R)$ travelling arbitrarily close to $(s,\pi_s)$ from both above and below. Since $\mathcal{M}(R)$ is closed, it follows that the path $\pi$, considered as a set of points $\{ (t, \pi_t) : t \geq \sigma_\pi \}$, is a subset of a path $\pi'$ of $\mathcal{M}(R)$. However, we note that $\pi'$ may not start before $\pi$, since for each arbitrary point $z_j$ above the path of $\pi'$, for all sufficiently large $n$, the start time $\sigma_j(R_n)$ of the path $\alpha^j(R_n)$ satisfies $\sigma_j(R_n) \leq \sigma_J(R)$. Thus we in fact have $\sigma_\pi = \sigma_{\pi'}$, and thus $\pi' = \pi$. That establishes that  $\mathcal{S} \subseteq \mathcal{M}(R)$, and accordingly, proves \eqref{eq:haus2}.

Finally, to see that each point $(t,x)$ of $(0,\infty) \times \mathbb{R}$ is almost surely contained in a bubble of $\mathcal{M}(R)$, note that the height $U^j_{t_j}-L^j_{t_j}$ of the $j^{\text{th}}$ vein leading to $z_j$ is an almost surely positive random variable for each $j > 0$. In particular, the bubble containing each $z_j$ has an almost surely positive Lebesgue measure. In particular, as $\varepsilon \to 0$, the probability that the bubble containing $z_j$ contains a Euclidean ball around $z_j$ of radius $\varepsilon$ converges to $1$. Since the $\mathcal{D} = \{ z_1,z_2,\ldots\}$ is a dense subset of $(0,\infty) \times \mathbb{R}$, it follows that each point of $(0,\infty) \times \mathbb{R}$ is almost surely contained in a non-empty bubble.
\end{proof}

\subsection{Distributional formulas for the Brownian marble}
We close this section by deriving the results stated in Section \ref{sec:distributional}.

Let $z = (t,x) \in (0,\infty) \times \mathbb{R}$. To describe the probabilistic properties of the bubble containing $z$ in the Brownian marble $\mathcal{M}(R_\lambda)$, consider the vein and its continuation to $z$. These are given by 
\begin{align*}
V = (L_s,C_s,U_s)_{s \in [0,t]} \quad \text{and} \quad \tilde{V} = (L_s,U_s)_{s \in [t,\tau]},
\end{align*}
where $\tau$ is the death time of the bubble containing $z$. Let $\sigma$ be the birth time of the bubble containing $z$. Then the bubble containing $z$ takes the form
\begin{align*}
\mathcal{B}^z = \{ (s,y) : \sigma < s < \tau , L_s < y < U_s \}.
\end{align*}
The stochastic process $X := (X_s)_{s \in [0,t]} :=  ((U_s-L_s)/\sqrt{2})_{s \in [0,t]}$ is distributed according to $\mathbf{Q}_0^{R_\lambda}$. Accordingly, by part (1) of Proposition \ref{prop:sslaw}, the birth time $\sigma$ of the bubble containing $(t,x)$, which is identical to the last visit time to $0$ of $X$ before time $t$, satisfies \eqref{eq:bubblestart}. Likewise, \eqref{eq:bubbleheight} follows from Proposition \ref{prop:sslaw} part (3). 

Finally, Theorem \ref{thm:dbessel} follows from Lemma \ref{lem:veiny} and \eqref{eq:Funct400}.

\section{The $R$-Bessel process seen as a spine}
\label{sect:spine}

In this brief final section, we expound further on how the growth-fragmentation process defined in Section \ref{sect:branching} can be related with the $R$-Bessel process and the vein construction of the $R$-marble.


Let $p^N(t;y,x)$ be the density of mass at time $t$ starting with a single cluster of mass $y>0$ i.e., for every continuous bounded test function
$$
\left<\mu^N_t,f\right> = \mathbb{E}_{\delta_{y}}\left(
\sum_{u\in{\mathcal N}_t} f(X_u)\right) \ = \ \int_0^\infty f(x) p(t;y,x)dx
$$
where the sum is taken over the clusters alive at time $t$
and $X_u$ is the mass of the cluster labelled $u$.
Then $(t,x)\to p(t; y,x)$ is the fundamental solution of the non-local equation
$$
\partial_t u \ \ = \ \frac{1}{2} \partial_{xx} u \ + \  R(x)\left(N  u(t,\frac{x}{N}) - u(t,x)\right).
$$
Define the operator 
on twice differentiable functions 
on $\mathbb{R}_+$
$$
{\mathcal B}^N f(x) \ =  \  \frac{1}{2} \partial_{xx} f \ + \  R(x)\left( N f(\frac{x}{N}) - f(x)\right), \ \ f(0) = 0
$$
 The total mass is preserved in expectation which translates into the property that the identity function $h(x)=x$ is harmonic for ${\mathcal B}^N$ in the sense that ${\mathcal B}^N h(x) = 0$. Consider the Doob's $h$ transform associated to $h$
\begin{eqnarray}
\label{eq:spine}
\frac{{\mathcal B}^N(hf)}{h}(x) &  = & \frac{1}{2} f''(x) + \frac{1}{x} f'(x) + R(x)(f(\frac{x}{N}) - f(x)) 
\end{eqnarray}
and let $(\bar X_t^N)_{t\geq0}$ be the Markov process
whose generator is given by \eqref{eq:spine}. This is a $3$-Bessel process 
jumping at rate $R(x)$ from $x$ to $x/N$.
We call this process the spine associated to the growth-fragmentation process. By the many-to-one formula (Lemma 2 in \cite{bertoin2020strong}), we can relate the spine process to the additive functionals of the growth-fragmentation process: 
for every continuous, bounded and compactly supported function on $(0,\infty)$, 
$$
\mathbb{E}_{\delta_{y}}\left(
\sum_{u\in{\mathcal N}_t} f(X_u)\right) \ = \ h(x) \mathbb{E}_{y}\left(\frac{f(\bar X^N_t)}{h(\bar X^N_t)}\right)
$$
By the same argument as in Theorem \ref{thm:newtransition}, one can argue that as $N\to\infty$, the process $\bar X^N$ converges to a limiting process $\bar X$, and further that this limit is identical to the $R$-Bessel process $X$ as defined in Theorem \ref{thm:newtransition}. Note that in combination with the many-to-one lemma, this shows one part of Conjecture \ref{conj:branching}. For every $t>0$, the sequence of measures 
$\mu_t^N$ converges in the vague topology to the trivial measure for $\lambda\geq 6$.


\begin{thebibliography}{35}
\bibitem{ANM}
\textsc{Adler, M., Nordenstam, E. and Van Moerbeke, P.} (2014). The Dyson Brownian minor process. \emph{Ann. Inst. Fourier} \textbf{64}(3): 971--1009.

\bibitem{arratia}
\textsc{Arratia, R. A.} (1979). Coalescing Brownian Motions on the Line. PhD Thesis, University of Wisconsin--Madison.

\bibitem{assiotis}
\textsc{Assiotis, T.} (2020). Determinantal structures in space-inhomogeneous dynamics on interlacing arrays. \emph{Ann. Henri Poincaré} \textbf{21}: 909--940.

\bibitem{barbato}
\textsc{Barbato, D.} (2005). FKG inequality for Brownian motion and stochastic differential equations. \emph{Electron. Commun. Probab.} \textbf{10}: 7--16.

\bibitem{BMS}
\textsc{Barnes, C., Mytnik, L. and Sun, Z.} (2024). On the coming down from infinity of coalescing Brownian motions. \emph{Ann. Probab.} \textbf{52}(1): 67--92.

\bibitem{BGS}
\textsc{Berestycki, N., Garban, C. and Sen, A.} (2015). Coalescing Brownian flows: a new approach. \emph{Ann. Probab.} \textbf{43}(6): 3177--3215.

\bibitem{bertoinbook}
\textsc{Bertoin, J.} (2006). \emph{Random Fragmentation and Coagulation Processes}. Cambridge University Press.

\bibitem{berestycki}
\textsc{Berestycki, J.} (2004). Exchangeable fragmentation-coalescence processes and their equilibrium measures. \emph{Electron. J. Probab.} \textbf{9}: 770--824.

\bibitem{bertoin2020strong}
\textsc{Bertoin, J. and Watson, A.} (2020). The strong Malthusian behavior of growth-fragmentation processes. \emph{Ann. Henri Lebesgue} \textbf{3}: 795--823.

\bibitem{BY}
\textsc{Bertoin, J. and Yor, M.} (2001). On subordinators, self-similar Markov processes and some factorizations of the exponential variable. \emph{Electron. Commun. Probab.} \textbf{6}: 95--106.

\bibitem{BN}
\textsc{Burdzy, K. and Nualart, D.} (2002). Brownian motion reflected on Brownian motion. \emph{Probab. Theory Relat. Fields} \textbf{122}(4): 471--493.

\bibitem{CC}
\textsc{Caballero, M. E. and Chaumont, L.} (2006). Weak convergence of positive self-similar Markov processes and overshoots of Lévy processes. \emph{Ann. Probab.} \textbf{34}(3): 1012--1034.

\bibitem{CH2}
\textsc{Cannizzaro, G. and Hairer, M.} (2023). The Brownian Castle. \emph{Commun. Pure Appl. Math.} \textbf{76}(10): 2693--2764.

\bibitem{CH}
\textsc{Cannizzaro, G. and Hairer, M.} (2023). The Brownian Web as a random $\mathbb R$-tree. \emph{Electron. J. Probab.} \textbf{28}: 1--47.

\bibitem{CKPR}
\textsc{Chaumont, L., Kyprianou, A., Pardo, J. C. and Rivero, V.} (2012). Fluctuation theory and exit systems for positive self-similar Markov processes. \emph{Ann. Probab.} \textbf{40}(1): 245--279.

\bibitem{CPR}
\textsc{Chaumont, L., Pantí, H. and Rivero, V.} (2013). The Lamperti representation of real-valued self-similar Markov processes. \emph{Bernoulli} \textbf{19}(5B): 2494--2523.

\bibitem{CSST}
\textsc{Coupier, D., Saha, K., Sarkar, A. and Tran, V. C.} (2021). The 2D-directed spanning forest converges to the Brownian web. \emph{Ann. Probab.} \textbf{49}(1): 435--484.

\bibitem{dubedat}
\textsc{Dubédat, J.} (2004). Reflected planar Brownian motions, intertwining relations and crossing probabilities. \emph{Ann. Inst. Henri Poincaré Probab. Stat.} \textbf{40}(5): 539--552.

\bibitem{EF}
\textsc{Ellis, T. and Feldheim, O. N.} (2016). The Brownian web is a two-dimensional black noise. \emph{Ann. Inst. Henri Poincaré Probab. Stat.} \textbf{52}(1): 162--172.

\bibitem{etheridge}
\textsc{Etheridge, A.} (2000). \emph{An Introduction to Superprocesses}. American Mathematical Society.

\bibitem{FINR}
\textsc{Fontes, L. R. G., Isopi, M., Newman, C. M. and Ravishankar, K.} (2004). The Brownian web: Characterization and convergence. \emph{Ann. Probab.} \textbf{32}(4): 2857--2883.

\bibitem{FINR2}
\textsc{Fontes, L. R. G., Isopi, M., Newman, C. M. and Ravishankar, K.} (2006). Coarsening, nucleation, and the marked Brownian web. \emph{Ann. Inst. Henri Poincaré Probab. Stat.} \textbf{42}(1): 37--60.

\bibitem{FKZ}
\textsc{Foss, S., Konstantopoulos, T. and Zachary, S.} (2007). Discrete and continuous time modulated random walks with heavy-tailed increments. \emph{J. Theor. Probab.} \textbf{20}: 581--612.

\bibitem{GS}
\textsc{Gorin, V. and Shkolnikov, M.} (2015). Multilevel Dyson Brownian motions via Jack polynomials. \emph{Probab. Theory Relat. Fields} \textbf{163}(3): 413--463.

\bibitem{HW}
\textsc{Howitt, C. and Warren, J.} (2009). Dynamics for the Brownian web and the erosion flow. \emph{Stochastic Process. Appl.} \textbf{119}(6): 2028--2051.

\bibitem{JY}
\textsc{Jacobsen, M. and Yor, M.} (2003). Multi-self-similar Markov processes on $\mathbb{R}_+$ and their Lamperti representations. \emph{Probab. Theory Relat. Fields} \textbf{126}(1): 1--28.

\bibitem{kingman}
\textsc{Kingman, J. F. C.} (1982). The coalescent. \emph{Stochastic Process. Appl.} \textbf{13}(3): 235--248.

\bibitem{kyprianou}
\textsc{Kyprianou, A. E.} (2006). \emph{Introductory Lectures on Fluctuations of Lévy Processes with Applications}. Springer.

\bibitem{KPRS}
\textsc{Kyprianou, A. E., Pagett, S. W., Rogers, T. and Schweinsberg, J.} (2017). A phase transition in excursions from infinity of the “fast” fragmentation-coalescence process. \emph{Ann. Probab.} \textbf{45}(6A): 3829--3849.

\bibitem{KPR}
\textsc{Kyprianou, A. E., Pagett, S. W. and Rogers, T.} (2017). Universality in a class of fragmentation-coalescence processes. \emph{Ann. Inst. Henri Poincaré Probab. Stat.} \textbf{54}(2): 755--784.

\bibitem{lawler}
\textsc{Lawler, G.} Notes on the Bessel process. Available at: \url{https://www.math.uchicago.edu/~lawler/bessel18new.pdf}.

\bibitem{legrand}
\textsc{Legrand, A.} (2024). Some FKG inequalities for stochastic processes. \emph{Preprint}, arXiv:2407.13871.

\bibitem{li}
\textsc{Li, Z.} (2011). \emph{Measure-Valued Branching Processes}. Springer.

\bibitem{MRTZ}
\textsc{Munasinghe, R., Rajesh, R., Tribe, R. and Zaboronski, O.} (2005). Multi-scaling of the $n$-point density function for coalescing Brownian motions. \emph{Commun. Math. Phys.} \textbf{268}: 717--725.

\bibitem{NRS}
\textsc{Newman, C. M., Ravishankar, K. and Schertzer, E.} (2010). Marking (1,2) points of the Brownian web and applications. \emph{Ann. Inst. Henri Poincaré Probab. Stat.} \textbf{46}(2): 537--574.

\bibitem{oconnell}
\textsc{O’Connell, N.} (2012). Directed polymers and the quantum Toda lattice. \emph{Ann. Probab.} \textbf{40}(2): 437--458.

\bibitem{oksendal}
\textsc{Øksendal, B.} (2003). \emph{Stochastic Differential Equations: An Introduction with Applications}. Springer.

\bibitem{patie}
\textsc{Patie, P.} (2012). Law of the absorption time of some positive self-similar Markov processes. \emph{Ann. Probab.} \textbf{40}(2): 765--787.

\bibitem{PPR}
\textsc{Pantí, H., Pardo, J. C. and Rivero, V. M.} (2020). Recurrent extensions of real-valued self-similar Markov processes. \emph{Potential Anal.} \textbf{53}(3): 899--920.

\bibitem{RY}
\textsc{Revuz, D. and Yor, M.} (1999). \emph{Continuous Martingales and Brownian Motion}, 3rd ed. Springer.

\bibitem{RSS}
\textsc{Roy, R., Saha, K. and Sarkar, A.} (2016). Random directed forest and the Brownian web. \emph{Ann. Inst. Henri Poincaré Probab. Stat.} \textbf{52}(3): 1106--1143.

\bibitem{rivero1}
\textsc{Rivero, V.} (2005). Recurrent extensions of self-similar Markov processes and Cramér's condition. \emph{Bernoulli} \textbf{11}(3): 471--509.

\bibitem{rivero2}
\textsc{Rivero, V.} (2007). Recurrent extensions of self-similar Markov processes and Cramér’s condition II. \emph{Bernoulli} \textbf{13}(4): 1053--1070.

\bibitem{SSS}
\textsc{Schertzer, E., Sun, R. and Swart, J. M.} (2017). The Brownian web, the Brownian net, and their universality. In \emph{Advances in Disordered Systems, Random Processes and Some Applications}, 270--368.

\bibitem{SS2}
\textsc{Sun, R. and Swart, J. M.} (2008). The Brownian net. \emph{Ann. Probab.} \textbf{36}(3): 1153--1208.

\bibitem{SS}
\textsc{Schramm, O. and Smirnov, S.} (2011). On the scaling limits of planar percolation. \emph{Ann. Probab.} \textbf{39}: 1768--1814.

\bibitem{SW}
\textsc{Soucaliuc, F. and Werner, W.} (2002). A note on reflecting Brownian motions. \emph{Electron. Commun. Probab.} \textbf{7}: 117--122.

\bibitem{STW00}
\textsc{Soucaliuc, F., Tóth, B. and Werner, W.} (2000). Reflection and coalescence between independent one-dimensional Brownian paths. \emph{Ann. Inst. Henri Poincaré Probab. Stat.} \textbf{36}(4): 509--545.

\bibitem{TZ}
\textsc{Tribe, R. and Zaboronski, O.} (2011). Pfaffian formulae for one-dimensional coalescing and annihilating systems. \emph{Electron. J. Probab.} \textbf{16}: 2080--2103.

\bibitem{toth}
\textsc{Tóth, B. and Werner, W.} (1998). The true self-repelling motion. \emph{Probab. Theory Relat. Fields} \textbf{111}(3): 375--452.

\bibitem{VV}
\textsc{Vető, B. and Virág, B.} (2023). The geometry of coalescing random walks, the Brownian web distance and KPZ universality. \emph{Preprint}, arXiv:2305.15246.

\bibitem{warren}
\textsc{Warren, J.} (2007). Dyson's Brownian motions, intertwining and interlacing. \emph{Electron. J. Probab.} \textbf{12}: 573--590.


\end{thebibliography}
\end{document}